\documentclass[11pt]{article}
\usepackage[hscale=0.8,vscale=0.8]{geometry}
\usepackage[pdftex]{pict2e}
\usepackage{amsmath}
\usepackage{amsfonts}
\usepackage{latexsym}
\usepackage{graphicx}
\usepackage{overpic}
\usepackage{caption}
\usepackage{subcaption}
\usepackage{multicol}
\usepackage{multirow}
\usepackage{amssymb}
\usepackage{mathtools}
\usepackage{hhline}
\usepackage{amsthm}
\usepackage{float}
\usepackage{cite}
\usepackage{color}
\usepackage{chngcntr}
\usepackage{comment}
\counterwithin{equation}{section}

\usepackage{hyperref}
\hypersetup{
	colorlinks=true,
	linkcolor=blue,
        citecolor=blue,
	filecolor=magenta,      
	urlcolor=cyan,
}

\allowdisplaybreaks
\graphicspath{{./figures/}}

\def\u{{\bf u}}
\def\bv{{\bf v}}

\def\n{{\bf n}}
\def\f{{\bf f}}

\def\x{{\bf x}}

\def\btau{\boldsymbol{\tau}}
\def\bbeta{\boldsymbol{\eta}}
\def\bs{\boldsymbol{\sigma}}
\def\bxi{\boldsymbol{\xi}}

\def\bchi{\boldsymbol{\chi}}
\def\bmu{\boldsymbol{\mu}}
\def\bvarphi{\boldsymbol{\varphi}}

\def\bL{\boldsymbol{\Lambda}}
\def\U{{\bf U}}
\def\V{{\bf V}}
\def\M{{\bf M}}
\def\D{{\bf D}}
\def\X{{\bf X}}

\def\A{{\bf A}}
\def\B{{\bf B}}

\def\E{\boldsymbol{H}}

\def\bI{{\bf I}}
\def\I{{\cal I}}

\def\d{\partial}
\def\T{{\cal T}}

\def\R{{\bf R}}
\def\grad{\nabla}
\def\div{\grad\cdot}

\def\O{\Omega}

\def\<{\langle}
\def\>{\rangle}
\def\CO2{{CO$_{2}$}}

\def\ol{\overline}

\def\cl {\nonumber \\}
\def\el {\nonumber }

\usepackage{xcolor}

\def\ddt{d_t}

\def\dt{d_t}
\def\dtt{d_{tt}}

\def\ds{\displaystyle}

\newtheorem{theorem}{Theorem}[section]

\newtheorem{remark}{Remark}[section]

\begin{document}

\title{A Robin-Robin splitting method for the Stokes-Biot fluid-poroelastic structure interaction model}

\author{ Aashi Dalal\thanks{Department of Mathematics, University of Pittsburgh, Pittsburgh, PA 15260, USA; email: {\tt aad100@pitt.edu}.}
~\and
Rebecca Durst\thanks{Argonne National Laboratory, Lemont, Il 60439, USA; email: {\tt rdurst@anl.gov}; partially supported by the U.S. Department of Energy, Office of Science, Office of Advanced Scientific Computing Research (ASCR) program under Contract DE-AC02-06CH11357.}
\thanks{Department of Mathematics, University of Pittsburgh, Pittsburgh, PA 15260, USA.}
~\and
Annalisa Quaini\thanks{Department of Mathematics, University of Houston,
Houston, TX 77204, USA; email: {\tt aquaini@central.uh.edu};
partially supported by NSF grant {DMS-1953535}.}
  ~\and
  Ivan Yotov\thanks{Department of Mathematics,
    University of Pittsburgh, Pittsburgh, PA 15260, USA;
    email: {\tt yotov@math.pitt.edu}; partially supported by NSF grants DMS-2111129 and DMS-2410686.}}
\date{\today}
\maketitle
\begin{abstract}
\noindent 
We develop and analyze a splitting method for fluid-poroelastic structure interaction. The fluid is described using the Stokes equations and the poroelastic structure is described using the Biot equations. The transmission conditions on the interface are mass conservation, balance of stresses, and the Beavers-Joseph-Saffman condition. The splitting method involves single and decoupled Stokes and Biot solves at each time step. The subdomain problems use Robin boundary conditions on the interface, which are obtained from the transmission conditions. The Robin data is represented by an auxiliary interface variable. We prove that the method is unconditionally stable and establish that the time discretization error is {$\mathcal{O}(\sqrt{T}\Delta t)$}, where $T$ is the final time and $\Delta t$ is the time step. We further study the iterative version of the algorithm, which involves an iteration between the Stokes and Biot sub-problems at each time step. We prove that the iteration converges to a monolithic scheme with a Robin Lagrange multiplier used to impose the continuity of the velocity. Numerical experiments are presented to illustrate the theoretical results.
\end{abstract}
\noindent {\bf Keywords}: Stokes-Biot; fluid-poroelastic structure interaction; domain decomposition;  Robin interface condition.

\section{Introduction}

We consider the interaction of an incompressible, viscous, and Newtonian fluid
with a poroelastic medium, referred to as fluid-poroelastic structure interaction (FPSI). This phenomenon occurs in a wide range of applications, including
surface-groundwater flows, geomechanics, reservoir engineering, filter design, seabed-wave interaction, and arterial blood flows. The modeling leads to coupled problems that present significant mathematical and computational challenges. 

We model the incompressible flow with the Stokes equations and the poroelastic medium with the
Biot system \cite{b1941}. In the latter, the equation describing the deformation of the elastic porous matrix is complemented with the Darcy equations that describe the average
velocity and pressure of the fluid in the pores. The Stokes and Biot problems are coupled at the interface between the fluid and porous regions through dynamic and kinematic transmission conditions. The Stokes-Biot model combines two classical well-studied kinds of coupling: the fluid-(elastic)structure interaction (FSI) with thick structure and the Stokes–Darcy coupling. While historically the Stokes-Biot and Navier-Stokes-Biot couplings have been less studied, they have received increasing attention in recent years.

Early works on the coupled Stokes-Biot problem include \cite{muradg2,s2005}. Well posedness of the fully dynamic model was established in \cite{s2005}. In \cite{bqq2009},
both monolithic solvers and partitioned approaches based on 
domain decomposition methods \cite{B-quarteroniv1} 
were applied to the Navier-Stokes-Biot problem, whose well-posedness with a non-mixed Darcy formulation is established in \cite{cesm2017}.
A non-iterative partitioned method based on 
operator splitting for the Navier-Stokes-Biot problem
with non-mixed Darcy formulation was introduced in 
\cite{byz2015} and extended in \cite{Bukac-JCP}. The model with a mixed 
Darcy formulation was studied in \cite{fpsi-nitsche}. In this formulation, the continuity of flux
across the interface is a condition of essential type, 
which is enforced with the Nitsche’s interior penalty method. A mixed formulation for the Darcy problem in the Stokes-Biot coupling was also adopted 
in \cite{akyz2018}, where the the continuity of flux is
imposed via a Lagrange multiplier method.
A more complex Stokes-Biot problem involving a non-Newtonian fluid is considered in \cite{aeny2019}.
A finite element method for the four-field Stokes velocity-pressure and Biot displacement-pressure formulation is presented in \cite{Cesm-Chid}. A total pressure formulation for the Stokes-Biot problem is introduced in 
\cite{Stokes-Biot-eye}, a stress-displacement mixed elasticity formulation is studied in \cite{fpsi-mixed-elast}, a fully mixed formulation is developed in \cite{fpsi-msfmfe}, and a HDG method is presented in \cite{HDG-SB}. An augmented finite element method for the fully mixed Navier-Stokes-Biot problem is developed in \cite{fpsi-augmented}.
Several interesting extensions of the Stokes-Biot problem have also been proposed, including a dimensionally-reduced model for flow through fractures \cite{Buk-Yot-Zun-fracture}, coupling with transport \cite{fpsi-transport,Canic-stent},
multi-layered porous media \cite{Bociu-etal-2021}, 
and porohyperelastic media \cite{Seboldt-etal-2021}. An optimization-based decoupling method is presented in \cite{Cesm-etal-optim}. Second order in time split schemes are developed in
\cite{Kunwar-etal,pb2024}. Parameter-robust preconditioners are studied in \cite{Boon-etal-precond}.

In this paper, we study a new Robin-Robin partitioned method for the Stokes-Biot problem. We employ a velocity-pressure Stokes formulation, a displacement-based elasticity formulation, and a mixed velocity-pressure Darcy formulation. The starting point is a rewriting of the coupling conditions for the Stokes-Biot problem, which state mass conservation, balance of stresses, and slip with friction (Beavers-Joseph-Saffman condition). These conditions are combined to generate two sets of Robin boundary conditions on the interface - one for the Stokes problem and the other for the Biot problem. Such approach was first utilized for FPSI in \cite{bqq2009}, and later used in \cite{Seboldt-etal-2021,pb2024}. The motivation is to alleviate the difficulty of the so-called added-mass-effect, which has been observed in FSI for certain parameter regimes \cite{Causin-added-mass-effect}. 
This effect may cause classical Neumann-Dirichlet split methods to become unstable or the convergence of their iterative versions to deteriorate \cite{Causin-added-mass-effect,bqq2009}. It has been shown that Robin-Robin schemes are more robust for wider ranges of physical parameters \cite{Badia-Robin-FSI,bqq2009}. The algorithm developed in this paper differs from the methods in \cite{bqq2009,Seboldt-etal-2021,pb2024}. It is inspired by a sequential Robin-Robin domain decomposition method for the Stokes-Darcy coupling introduced in \cite{Disc-Quart-Valli-2007}.
The key ingredient in the design of the algorithm is an auxiliary vector variable used
to approximate the interface Robin data, which is modeled in a suitable norm. This avoids the explicit appearance of the normal stress in the interface terms, which does not have sufficient regularity for the stability of the subdomain problems and lead to suboptimal approximation in space and time. In its non-iterative form, our FPSI algorithm requires the sequential solution of one Stokes problem and one Biot problem per time step. The appealing features of this algorithm
are modularity (one can use their favorite Stokes and Biot finite element approximations), the option to use non-matching meshes, and the reduced computational cost that comes from the lack of sub-iterations. It is not unusual for partitioned methods to have stability issues or suffer from a loss of accuracy. We prove that our non-iterative method is unconditionally stable and establish an error estimate for the quasistatic model showing that the time discretization error is $\mathcal{O}(\sqrt{T}\Delta t)$, where $T$ is the final time and $\Delta t$ is the time step.
We remark that in \cite{bqq2009} the Robin-Robin method is only studied computationally, while in \cite{bqq2009,Seboldt-etal-2021,pb2024} only stability analysis is performed. Thus, to the best of our knowledge, this is the first result in the literature on both unconditional stability and optimal time discretization error for a non-iterative Robin-Robin split scheme for FPSI.

We also present the iterative version of the algorithm, i.e., at every time step one iterates over the Stokes and Biot sub-problems until convergence. We prove that the solution of the iterative method converges to the solution of a monolithic scheme. The auxiliary interface variable converges to a Robin Lagrange multiplier used to impose weakly the velocity continuity condition. To the best of our knowledge, the resulting monolithic scheme has not been studied in the literature. We prove that a unique and stable solution to the monolithic scheme exists.
While the iterative FPSI algorithm has an increased computational cost, it has the advantage of not introducing a splitting error while still allowing to recycle existing fluid and structure solvers. 
For the monolithic scheme instead, one needs to implement a coupled solver. 
We assess the convergence, robustness, and accuracy
of the non-iterative and iterative Robin-Robin methods and the monolithic scheme with 
two benchmarks: a test that features an exact solution and a simplified blood flow problem.

We further remark that the FPSI problem is a generalization of FSI with thick structure. Therefore, the techniques developed in this paper apply to the corresponding Robin-Robin algorithm for FSI, which is also new. In recent years, alternative unconditionally stable Robin-Robin methods for FSI have been developed in \cite{Seb-Buc-FSI,BDFG22,BDFG22-discrete}, where discretization error of order $\mathcal{O}(\sqrt{\Delta t})$ is established. An improved convergence of order $\mathcal{O}(\Delta t\sqrt{T + \log\frac{1}{\Delta t}})$ is obtained in \cite{BDFG23} for a related model in a specific geometry. Our work provides a generalization and improvement of these results. 

The remaining of the paper is structured as follows. Sec.~\ref{sec:problem} describes the coupled Stokes-Biot problems. Sec.~\ref{sec:RR} presents the non-iterative Robin-Robin algorithm, whose stability analysis is carried out in Sec.~\ref{sec:stab}. The time discretization error analysis for the quasistatic model is presented in Sec.~\ref{sec:error}. The iterative version of the algorithm is developed in Sec.~\ref{sec:iter}. The numerical results are discussed in Sec.~\ref{sec:num_res}. Finally, conclusions are drawn in Sec.~\ref{sec:concl}.

\section{Problem setting}\label{sec:problem}

We consider a multiphysics model problem studied in \cite{akyz2018} that describes the interaction of a free fluid  with a flow in a 
deformable porous media. The spatial domain $\Omega \subset
\R^d$, $d = 2,3$ is the union of non-overlapping 
regions $\O_f$ and $\O_p$, see Figure~\ref{fig:domain_sketch}.
Here, $\O_f$ is a free fluid region
with flow governed by the Stokes equations and $\O_p$ is a poroelastic
material governed by the Biot system. For simplicity of notation, we assume 
that each region is connected. The extension to non-connected regions is 
straightforward. Let $\Gamma_{fp} = \d \O_f \cap
\d \O_p$ be the interface. Let $(\u_\star,p_\star)$ be the velocity-pressure pair in
$\O_\star$, $\star = f$, $p$, and let $\bbeta_p$ be the displacement
in $\O_p$.  Let $\mu_{f} > 0$ 
be the fluid viscosity,
let $\f_\star$ be the body force terms, and let
$q_\star$ be external source or sink terms.  Finally, let $\D(\u_f)$ and
$\bs_f(\u_f,p_f)$ denote, respectively, the deformation rate tensor
and the Cauchy stress tensor:
\begin{equation}\label{sigma-f-defn}
\D(\u_f) = \frac 12 (\grad \u_f + \grad\u_f^T), \qquad
\bs_f(\u_f,p_f) = -p_f \bI + 2\mu_{f} \D(\u_f).
\end{equation}

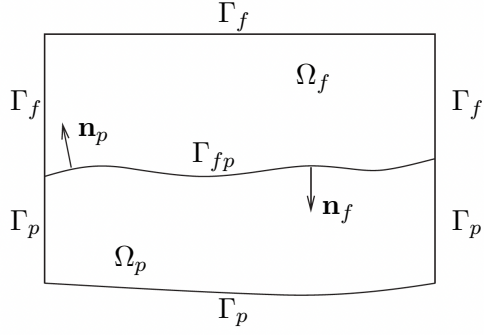
\begin{figure}[htb!]
\centering
\begin{overpic}[width=.4\textwidth, grid=false]{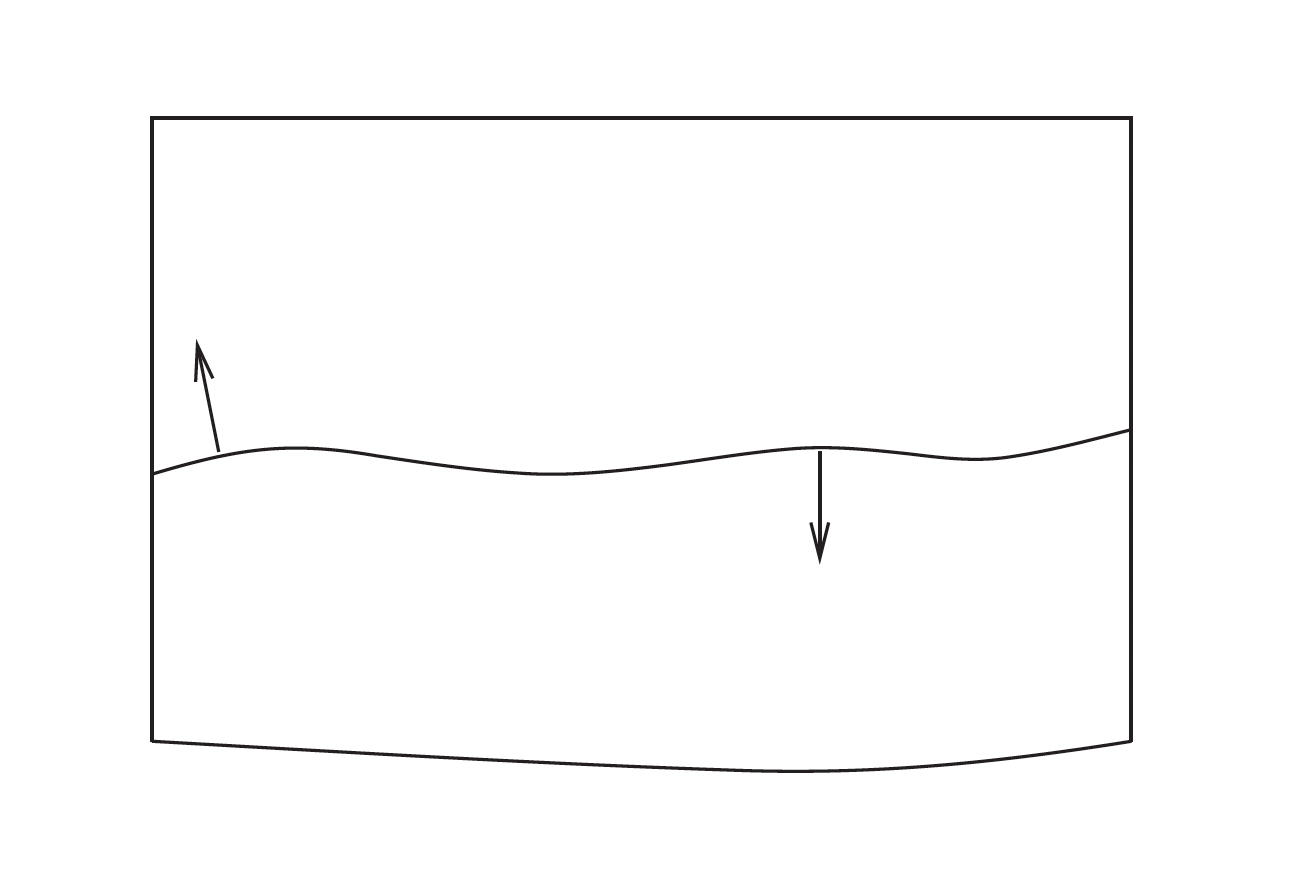}
\put(60,50){$\Omega_f$}
\put(25,15){$\Omega_p$}
\put(40,35){$\Gamma_{fp}$}
\put(65,25){$\n_f$}
\put(18,40){$\n_p$}
\put(5,45){$\Gamma_{f}$}
\put(90,45){$\Gamma_{f}$}
\put(45,62){$\Gamma_{f}$}
\put(5,22){$\Gamma_{p}$}
\put(90,22){$\Gamma_{p}$}
\put(45,5){$\Gamma_{p}$}
\end{overpic}
\caption{Schematic representation of a 2D computational domain.}
\label{fig:domain_sketch}
\end{figure}

In the free fluid region $\O_f$, $(\u_f,p_f)$ satisfy 
the Stokes equations
\begin{align}
	\rho_f {\partial_t \u_f}-\div \bs_f(\u_f,p_f)  &= \f_f \quad \mbox{in } \O_f \times (0,T] \label{stokes1} \\
	 \div \u_f &=  q_f   \quad \mbox{in } \O_f \times (0,T], \label{stokes2}
\end{align}
where  $\d_t = \frac{\d}{\d t}$ and $T > 0$ is the end of the time interval of interest.
Let $\bs_e(\bbeta_p)$ and $\bs_p(\bbeta_p,p_p)$ be the elastic and poroelastic stress tensors, 
respectively:
\begin{equation}\label{stress-defn}
	\bs_e(\bbeta_p) = \lambda_p (\div \bbeta_p) \bI + 2\mu_p \D(\bbeta_p), \qquad
	\bs_p(\bbeta_p,p_p) = \bs_e(\bbeta_p) - \alpha p_p \bI,
\end{equation}
where $ 0 < \lambda_{min} \le \lambda_p(\x) \le \lambda_{max}$ and 
$0 < \mu_{min} \le \mu_p(\x) \le \mu_{max} $ are the Lam\'{e} parameters and 
$0 \le \alpha \le 1$ is the Biot-Willis constant. 
The poroelasticity region $\O_p$ is governed by the Biot system \cite{b1941}
\begin{eqnarray}
\rho_p {\partial_{tt} \bbeta}- \div \bs_p(\bbeta_p,p_p) &= \f_p \quad \mbox{in } \O_p\times (0,T],  \label{eq:biot0} \\
\mu_{f} K^{-1}\u_p + \grad p_p &= 0 \quad \mbox{in } \O_p\times (0,T], \label{eq:biot1}\\ 
	{\d_t} \left(s_0 p_p + \alpha \div \bbeta_p\right) + \div \u_p &= q_p 
	\quad \mbox{in } \O_p\times (0,T], \label{eq:biot2}
\end{eqnarray}
where $\d_{tt} = \frac{\d^2}{\d t^2}$, $s_0 > 0$ is a storage coefficient and $K$ the symmetric and uniformly positive definite permeability tensor, satisfying, for some constants 
$0 < k_{min} \le k_{max}$,
$$
\forall \, \bxi \in \R^d, \,\,  
k_{min} \bxi^T\bxi \le \bxi^T K(\x) \bxi \le k_{max} \bxi^T\bxi, \,\, \forall \, 
\x \in \Omega_p.
$$

Following \cite{bqq2009,s2005}, the 
{\it interface conditions} on the 
fluid-poroelasticity interface $\Gamma_{fp}$ are {\it mass conservation},
{\it balance of stresses}, and the Beavers-Joseph-Saffman (BJS) condition
\cite{bj1967,s1971} modeling {\it slip with friction}: 
\begin{align}
	\label{eq:mass-conservation}
	& \u_f\cdot\n_f + \left({\d_t \bbeta_p} + \u_p\right)\cdot\n_p = 0 
&& \mbox{on } \Gamma_{fp}\times (0,T], 
	\\
        \label{balance-stress}
	& - (\bs_f \n_f)\cdot\n_f =  p_p, \qquad 	\bs_f\n_f + \bs_p\n_p = 0
&& \mbox{on } \Gamma_{fp}\times (0,T],
	\\
	\label{Gamma-fp-1}
	& - (\bs_f\n_f)\cdot\btau_{f,j} = \mu_{f}\alpha_{BJS}\sqrt{K_j^{-1}}
	\left(\u_f - {\d_t \bbeta_p}\right)\cdot\btau_{f,j} 
&& \mbox{on } \Gamma_{fp}\times (0,T],
\end{align}
where $\n_f$ and $\n_p$ are the outward unit normal vectors to $\d
\O_f$, and $\d \O_p$, respectively, $\btau_{f,j}$, $1 \le j \le d-1$,
is an orthogonal system of unit tangent vectors on $\Gamma_{fp}$, $K_j
= (K \btau_{f,j}) \cdot \btau_{f,j}$, and $\alpha_{BJS} \ge 0$ is an
experimentally determined friction coefficient.  We note that the
continuity of flux \eqref{eq:mass-conservation} constrains the normal velocity of the solid
skeleton, while the BJS condition \eqref{Gamma-fp-1} accounts for its tangential
velocity.

The above system of equations needs to be complemented by a set of
boundary and initial conditions. Let $\Gamma_f = \partial \O_f \cap \d\Omega$
and $\Gamma_p = \partial \O_p \cap \d\Omega$, see Figure~\ref{fig:domain_sketch}.
Let {$\Gamma_f = \Gamma_f^D \cup \Gamma_f^N$ with $|\Gamma_f^D| > 0$} and $\Gamma_p = \Gamma_p^D \cup \Gamma_p^N = \tilde\Gamma_p^D \cup \tilde\Gamma_p^N$.
We assume for simplicity homogeneous boundary conditions: for every $t \in [0,T]$,
$$
{\u_f = \boldsymbol{0} \mbox{ on } \Gamma_f^D, \ \ \bs_f\n_f = \boldsymbol{0} \mbox{ on } \Gamma_f^N,}
{\ \ \bbeta_p = \boldsymbol{0} \mbox{ on } \Gamma_p^D,
\ \ \bs_p\n_p = \boldsymbol{0} \mbox{ on } \Gamma_p^N,}
\ \ p_p = 0 \mbox{ on } \tilde\Gamma_p^D, \ \
\u_p \cdot \n_p = 0 \mbox{ on } \tilde\Gamma_p^N.
$$
{To simplify the characterization of the space for the trace $\u_f|_{\Gamma_{fp}}$, we assume that $\Gamma_f^D$ is not adjacent to $\Gamma_{fp}$. In the case when they touch, the boundary condition $\u_f = \boldsymbol{0}$ on $\Gamma_f^D$ needs to be imposed weakly by introducing a Lagrange multiplier $\bvarphi = \bs_f\n_f$ on $\Gamma_f^D$. We omit further details for this case.} Finally, we set the initial conditions
\begin{equation}\label{init-cond}
\u_f(\mathbf{x},0) =\u_{f,0} \mbox{ in } \O_f, \,\, 
p_p(\mathbf{x},0) = p_{p,0}(\mathbf{x}), \,\, 
\bbeta_p(\mathbf{x},0) = \bbeta_{p,0}(\mathbf{x}), \,\,  \d_{t}\bbeta_p(\mathbf{x},0) = \u_{s,0} (\mathbf{x})\mbox{ in } \O_p.
\end{equation}

{Compatible initial data $p_{f,0}$ for $p_f(0)$ and $\u_{p,0}$ for $\u_p(0)$ can be obtained by solving at $t = 0$ a sequence of single-physics sub-problems satisfying the interface conditions \eqref{eq:mass-conservation}--\eqref{Gamma-fp-1}, see \cite{aeny2019}.
}

{The solvability of the fully dynamic Stokes-Biot system with compressible Stokes fluid 
was discussed in \cite{s2005}. The well posedness analysis of the incompressible quasistatic system has been carried out in \cite{aeny2019}. The proof extends easily to the fully dynamic incompressible system
\eqref{stokes1}--\eqref{Gamma-fp-1} considered here.}

Let $(\cdot,\cdot)_S$, $S \subset \R^d$, be the
$L^2(S)$ inner product and let $\<\cdot,\cdot\>_F$, $F \subset
\R^{d-1}$, be the $L^2(F)$ inner product or duality pairing. We will
use the standard notation for Sobolev spaces, see, e.g.
\cite{c1978}. Let
\begin{align}
\V_{f} &= \{ \bv_f \in (H^1(\O_{f}))^d : \bv_f = \boldsymbol{0} \text{ on }\Gamma_f^D\}, \quad W_{f} = L^2(\O_{f}), \nonumber
\\
\V_{p} &= \{ \bv_p \in H({\rm div}; \O_{p}) : {\bv_p \cdot \n_p|_{\Gamma_{fp}} \in L^2(\Gamma_{fp})}
\text{ and } \bv_p \cdot \n_p = 0 \text{ on } \tilde\Gamma_p^N \}, \quad W_{p} = L^2(\O_{p}), \nonumber
\\
\X_{p} &= \{ \bxi_p \in H^1(\O_{p})^d : \bxi_p= \boldsymbol{0} \text{ on } \Gamma_p^D \}, \label{spaces}
\end{align}
where $H({\rm div}; \O_{p})$ is the space of $(L^2(\O_p))^d$-vectors with
divergence in $L^2(\O_p)$ with a norm
$$
\|\bv\|_{H({\rm div}; \O_{p})}^2 = \|\bv\|_{L^2(\Omega_p)}^2 + \|\div\bv\|_{L^2(\Omega_p)}^2,
$$
and the space $\V_p$ is equipped with the norm
\begin{equation}\label{Vp-norm}
{\|\bv\|_{\V_p}^2 = \|\bv\|_{H({\rm div}; \O_{p})}^2 + \|\bv\cdot\n_p\|_{L^2(\Gamma_{fp})}^2.}
\end{equation}

\section{Robin-Robin partitioned algorithm}\label{sec:RR}

For the solution of the coupled problem presented in Section~\ref{sec:problem}, {we consider
a Robin-Robin splitting algorithm motivated by the method proposed in \cite{bqq2009}}. For this, we rewrite the coupling conditions \eqref{eq:mass-conservation}--\eqref{Gamma-fp-1} in the equivalent form of Robin conditions. To simplify the notation, consider a single
tangential vector $\btau$ and let $\gamma_{BJS} = \sqrt{K_{\tau}}/(\mu_{f} \alpha_{BJS})$. {Let $\gamma_f > 0$ and $\gamma_p > 0$ be given combination parameters.} For the fluid subproblem, we consider the following Robin transmission conditions
\begin{subequations}\label{stokes-robin}
\begin{align}
  & \gamma_f\u_f\cdot\n_f + (\bs_f \n_f)\cdot\n_f =
  -\gamma_f \left(\d_t \bbeta_p + \u_p\right)\cdot\n_p + (\bs_p \n_p)\cdot\n_p,
  \label{stokes-robin-1} \\
  &\gamma_f\u_f\cdot\btau_{f}
  + \left(1 + \gamma_f\gamma_{BJS}  \right)(\bs_f\n_f)\cdot\btau_{f}
  = - \gamma_f \d_t \bbeta_p\cdot\btau_{p} + (\bs_p\n_p)\cdot\btau_{p},
  \label{stokes-robin-2}
\end{align}
\end{subequations}
while the transmission conditions for the poroelastic structure are
\begin{subequations}\label{biot-robin}
\begin{align}
& \gamma_p(\u_p + \d_t\bbeta_p)\cdot\n_p + (\bs_p \n_p)\cdot\n_p =
- \gamma_p \u_f\cdot\n_f + (\bs_f \n_f)\cdot\n_f, \label{biot-robin-1}\\
& \gamma_p \d_t \bbeta_p\cdot\btau_{p} + (\bs_p\n_p)\cdot\btau_{p} =
- \gamma_p \u_f\cdot\btau_{f}
+ \left(1 -\gamma_p \gamma_{BJS} \right)
(\bs_f\n_f)\cdot\btau_{f}, \label{biot-robin-2}\\
& \gamma_p (\u_p + \d_t\bbeta_p)\cdot\n_p - p_p =
-\gamma_p \u_f\cdot\n_f + (\bs_f \n_f)\cdot\n_f. \label{biot-robin-3}
\end{align}
\end{subequations}

To simplify the notation, let
\begin{align*}
a_{f}(\u_{f},\bv_{f}) &= (2\mu_{f}\D(\u_{f}),\D(\bv_{f}))_{\O_{f}},
\\
a^d_{p}(\u_{p},\bv_{p}) &= (\mu_{f} K^{-1}\u_{p},\bv_{p})_{\O_{p}},
\\
a^e_{p}(\bbeta_p,\bxi_p) &= (2\mu_p\D(\bbeta_p),\D(\bxi_p))_{\O_{p}}+(\lambda_p\nabla\cdot\bbeta_p,\nabla\cdot\bxi_p)_{\O_{p}}
\end{align*}
be the bilinear forms related to Stokes, Darcy, and the elasticity
operators, respectively.  Let 
$$
b_{\star}(\bv,w) = -(\div \bv,w)_{\O_\star}.
$$
Using standard techniques involving multiplication by suitable test functions and integration by parts, as well as the interface transmission conditions \eqref{stokes-robin}, we obtain that the solution of the system \eqref{stokes1}--\eqref{Gamma-fp-1}, $(\u_{f}, p_{f}, \bbeta_{p},\u_p,p_{p}): [0,T] \to \V_f\times W_f\times \X_p \times \V_p \times W_p$
satisfies the following variational formulation for each $t \in (0,T]$ and for all $(\bv_f,w_f,\bxi_p,\bv_p,w_p) \in \V_f\times W_f \times \X_p \times \V_p \times W_p$:
\begin{align}
  & \left(\rho_f {\partial_t \u_{f}},\bv_{f}\right)_{\O_f}  + a_{f}(\u_{f},\bv_{f}) + b_f(\bv_{f},p_{f})
  + \gamma_f\<\u_f\cdot\n_f,\bv_f\cdot\n_f\>_{\Gamma_{fp}}
  +  \gamma_f\<\u_f\cdot\btau_{f},\bv_f\cdot\btau_{f}\>_{\Gamma_{fp}}
  \nonumber \\
  & \qquad  
  = (\f_{f},\bv_{f})_{\O_f}
  - \gamma_f\<(\d_t \bbeta_p + \u_p)\cdot\n_p,\bv_f\cdot\n_{f}\>_{\Gamma_{fp}} 
  + \<(\bs_p \n_p)\cdot\n_p,\bv_f\cdot\n_{f}\>_{\Gamma_{fp}}
  \nonumber \\
  & \qquad\quad
  - \gamma_f\<\d_t\bbeta_p\cdot\btau_p,\bv_f\cdot\btau_{f}\>_{\Gamma_{fp}}
  + \<(\bs_p\n_p)\cdot\btau_p,\bv_f\cdot\btau_{f}\>_{\Gamma_{fp}} 
  - \gamma_f\gamma_{BJS}\<(\bs_f\n_f)\cdot\btau_{f},\bv_f\cdot\btau_{f}\>_{\Gamma_{fp}}, \label{semi-1}\\
  & - b_f(\u_{f},w_{f}) = (q_{f},w_{f})_{\O_f}, \label{semi-2}\\
  &\big( \rho_p \d_{tt} \bbeta_{p},\bxi_{p}\big)_{\O_p}
  + a^e_{p}(\bbeta_{p},\bxi_{p})
  + a^d_{p}(\u_{p},\bv_{p})
  + \alpha b_p(\bxi_{p},p_{p})
  + b_p(\bv_{p},p_{p})
  \nonumber \\
  & \qquad\qquad  + \gamma_p\<(\u_p + \d_t\bbeta_p)\cdot\n_p,
        (\bv_p + \bxi_p)\cdot\n_p\>_{\Gamma_{fp}}
  + \gamma_p\<\d_t\bbeta_p\cdot\btau_p,\bxi_p\cdot\btau_p\>_{\Gamma_{fp}}
  \nonumber \\
  & \qquad = (\f_{p},\bxi_{p})_{\O_p} + \<(-\gamma_p\u_f + \bs_f\n_f)\cdot\n_f,
  (\bv_p + \bxi_p)\cdot\n_p\>_{\Gamma_{fp}} \nonumber \\
  & \qquad\qquad -\gamma_p\<\u_f\cdot\btau_f,\bxi_p\cdot\btau_p\>_{\Gamma_{fp}}
  +(1- \gamma_p\gamma_{BJS})\<(\bs_f\n_f)\cdot\btau_f,\bxi_p\cdot\btau_p\>_{\Gamma_{fp}}, \label{semi-3}\\
  & s_0 (\d_t p_{p},w_{p})_{\O_p} - \alpha b_p(\d_t\bbeta_{p},w_{p}) - b_p(\u_{p},w_{p}) = (q_{p},w_{p})_{\O_p}. \label{semi-4}
\end{align}
In the above, we assume that the solution to \eqref{stokes1}--\eqref{Gamma-fp-1} is sufficiently smooth, so that all interface bilinear forms are continuous.

  {Let $\mathcal{T}_{f,h}$ and $\mathcal{T}_{p,h}$ be shape-regular finite element partitions of $\Omega_f$ and $\Omega_p$, respectively, consisting of affine elements with maximal diameter $h$. The two meshes may not match along $\Gamma_{fp}$.} For the space discretization we consider a stable Stokes finite element pair $\V_{f,h}\times W_{f,h} \subset \V_f\times W_f$, a finite element space $\X_{p,h} \subset \X_p$ for the displacement, and any stable Darcy pair $\V_{p,h}\times W_{p,h} \subset \V_p\times W_p$. Examples include include the Taylor-Hood or the MINI elements for $\V_{f,h}\times W_{f,h}$, continuous piecewise polynomials for $\X_{p,h}$ and the Raviart-Thomas or the Brezzi-Douglas-Marini elements for $\V_{p,h}\times W_{p,h}$. See \cite{boffi2013mixed} for further details.

  \begin{remark}
    In the following, in order to simplify the notation, we suppress the subscript $h$ from the variables.
  \end{remark}

For the time discretization, we consider a uniform partition of $[0,T]$ with time step $\Delta t = T/N$ and $t_n = n\Delta t$, $n = 0,\ldots,N$. Let $\varphi^{n} = \varphi(t_n)$.
For $n \ge 0$, let $\dt \varphi^{n+1} := (\varphi^{n+1} - \varphi^{n})/{\Delta t}$. Let
$\dtt\bbeta_p^{n+1} := \dt\dt\bbeta_p^{n+1} = (\dt\bbeta_p^{n+1} - \dt\bbeta_p^n)/\Delta t$ for $n \ge 0$. Note that $\dtt\bbeta_p^{n+1} = (\bbeta_p^{n+1} -2 \bbeta_p^{n}+\bbeta_p^{n-1})/{\Delta t^2}$ for $n \ge 1$, while for $n = 0$, $\dt\bbeta_p^0 := P_X\u_{s,0}$, using the initial condition \eqref{init-cond}, where $P_X:(L^2(\Omega_p))^d \to \X_p$ is the $L^2$-orthogonal projection.

{Next, we present a time-splitting Robin-Robin algorithm that is similar to the iterative scheme introduced in \cite{bqq2009}.} At every time $t^{n+1}$, $n = 0,\ldots,N-1$, the Robin-Robin algorithm involves solving decoupled fluid and poroelastic structure sub-problems:

\medskip

1. Stokes problem: find $(\u_{f}^{n+1}, p_{f}^{n+1}) \in \V_{f,h}\times W_{f,h}$ such that for all $(\bv_f,w_f) \in \V_{f,h}\times W_{f,h}$,
\begin{align}
  & \left(\rho_f {\dt \u_{f}^{n+1}},\bv_{f}\right)_{\O_f} \! + \! a_{f}(\u_{f}^{n+1},\bv_{f}) + b_f(\bv_{f},p_{f}^{n+1})
  + \gamma_f\<\u_f^{n+1}\cdot\n_f,\bv_f\cdot\n_f\>_{\Gamma_{fp}}
  +  \gamma_f\<\u_f^{n+1}\cdot\btau_{f},\bv_f\cdot\btau_{f}\>_{\Gamma_{fp}}
  \nonumber \\
  & \quad  
  = (\f_{f},\bv_{f})_{\O_f}
  - \gamma_f\<(\dt\bbeta_p^n + \u_p^n)\cdot\n_p,\bv_f\cdot\n_{f}\>_{\Gamma_{fp}} 
  + \<(\bs_p^n \n_p)\cdot\n_p,\bv_f\cdot\n_{f}\>_{\Gamma_{fp}}
  \nonumber \\
  & \qquad
  - \gamma_f\<\dt\bbeta_p^n\cdot\btau_p,\bv_f\cdot\btau_{f}\>_{\Gamma_{fp}}
  + \<(\bs_p^n\n_p)\cdot\btau_p,\bv_f\cdot\btau_{f}\>_{\Gamma_{fp}} 
  - \gamma_f\gamma_{BJS}\<(\bs_f^{n}\n_f)\cdot\btau_{f},\bv_f\cdot\btau_{f}\>_{\Gamma_{fp}}, \\
  & - b_f(\u^{n+1}_{f},w_{f}) = (q_{f},w_{f})_{\O_f},
\end{align}
which corresponds to the {finite element approximation} of the weak formulation of problem \eqref{stokes1}--\eqref{stokes2} supplemented with
the following Robin boundary conditions on $\Gamma_{fp}$:
\begin{subequations}\label{stokes-robin-mu}
\begin{align}
  & \gamma_f\u_f^{n+1}\cdot\n_f + (\bs_f^{n+1} \n_f)\cdot\n_f =
  -\gamma_f \left(\dt\bbeta_p^n + \u_p^n\right)\cdot\n_p + (\bs_p^n \n_p)\cdot\n_p,
  \label{stokes-robin-mu-1} \\
  &\gamma_f\u_f^{n+1}\cdot\btau_{f}
  + (\bs_f^{n+1}\n_f)\cdot\btau_{f}
  + \gamma_f\gamma_{BJS}(\bs_f^{n}\n_f)\cdot\btau_{f}
  = - \gamma_f \dt\bbeta_p^n\cdot\btau_{p} + (\bs_p^n\n_p)\cdot\btau_{p}.
  \label{stokes-robin-mu-2}
\end{align}
\end{subequations}
It is obtained by applying the backward Euler time discretization to \eqref{semi-1}--\eqref{semi-2} and time-lagging the interface quantities from the Biot region as well as the term $\gamma_f\gamma_{BJS}(\bs_f\n_f)\cdot\btau_{f}$.

Using the initial conditions \eqref{init-cond}, at $n = 0$ we set
  $\phi^0 := P_\phi\phi_0$ for any variable $\phi$, where $P_\phi$ is the $L^2$-orthogonal projection onto the corresponding finite element space and 
$\bs_f^0$ and $\bs_p^0$ are obtained from \eqref{sigma-f-defn} and \eqref{stress-defn}, respectively. We also recall that $\dt\bbeta_{p}^0 = P_X\u_{s,0}$.

\medskip
2. Biot problem: find $(\bbeta_{p}^{n+1},\u_p^{n+1},p_{p}^{n+1}) \in \X_{p,h} \times \V_{p,h} \times W_{p,h}$ such that for all $(\bxi_p,\bv_p,w_p) \in \X_{p,h}\times\V_{p,h}\times W_{p,h}$,
\begin{align}
  &\big( \rho_p \dtt \bbeta_{p}^{n+1},\bxi_{p}\big)_{\O_p}
  + a^e_{p}(\bbeta_{p}^{n+1},\bxi_{p})
  + a^d_{p}(\u_{p}^{n+1},\bv_{p})
  + \alpha b_p(\bxi_{p},p_{p}^{n+1})
  + b_p(\bv_{p},p_{p}^{n+1})
  \nonumber \\
  & \qquad\qquad  + \gamma_p\<(\u_p^{n+1} + \dt\bbeta_p^{n+1})\cdot\n_p,
        (\bv_p + \bxi_p)\cdot\n_p\>_{\Gamma_{fp}}
  + \gamma_p\<\dt\bbeta_p^{n+1}\cdot\btau_p,\bxi_p\cdot\btau_p\>_{\Gamma_{fp}}
  \nonumber \\
  & \qquad = (\f_{p},\bxi_{p})_{\O_p} + \<(-\gamma_p\u_f^{n+1} + \bs_f^{n+1}\n_f)\cdot\n_f,
  (\bv_p + \bxi_p)\cdot\n_p\>_{\Gamma_{fp}} \nonumber \\
  & \qquad\qquad -\gamma_p\<\u_f^{n+1}\cdot\btau_f,\bxi_p\cdot\btau_p\>_{\Gamma_{fp}}
  +(1- \gamma_p\gamma_{BJS})\<(\bs_f^{n+1}\n_f)\cdot\btau_f,\bxi_p\cdot\btau_p\>_{\Gamma_{fp}}, \\
  & s_0 (\dt p_{p}^{n+1},w_{p})_{\O_p} - \alpha b_p(\dt\bbeta_{p}^{n+1},w_{p}) - b_p(\u_{p}^{n+1},w_{p}) = (q_{p},w_{p})_{\O_p},
\end{align}
which corresponds to the {finite element approximation} of the weak formulation of problem \eqref{eq:biot1}--\eqref{eq:biot2} supplemented with
the following Robin boundary conditions on $\Gamma_{fp}$:
\begin{subequations}\label{biot-robin-mu}
\begin{align}
& \gamma_p(\u_p^{n+1} + \dt\bbeta_p^{n+1})\cdot\n_p + (\bs_p^{n+1} \n_p)\cdot\n_p =
- \gamma_p \u_f^{n+1}\cdot\n_f + (\bs_f^{n+1} \n_f)\cdot\n_f, \label{biot-robin-mu-1}\\
& \gamma_p \dt \bbeta_p^{n+1}\cdot\btau_{p} + (\bs_p^{n+1}\n_p)\cdot\btau_{p} =
- \gamma_p \u_f^{n+1}\cdot\btau_{f}
+
(1- \gamma_p\gamma_{BJS})(\bs_f^{n+1}\n_f)\cdot\btau_{f}, \label{biot-robin-mu-2}\\
& \gamma_p (\u_p^{n+1} + \dt\bbeta_p^{n+1})\cdot\n_p - p_p^{n+1} =
-\gamma_p \u_f^{n+1}\cdot\n_f + (\bs_f^{n+1} \n_f)\cdot\n_f. \label{biot-robin-mu-3}
\end{align}
\end{subequations}
It is obtained by applying the backward Euler time discretization to \eqref{semi-3}--\eqref{semi-4} and using the interface quantities from the Stokes region obtained in the previous Stokes solve.

{We note that the above algorithm involves $\bs_f\n_f$ and $\bs_p\n_p$ on $\Gamma_{fp}$, which in the continuous setting may not belong to $L^2(\Gamma_{fp})$, while in the finite element method they require postprocessing from the primary variables and may result in loss of accuracy. Moreover, the algorithm is difficult to analyze in this form. Motivated by the Robin-Robin iterative algorithm for the coupled Stokes-Darcy flow problem developed in \cite{Disc-Quart-Valli-2007}, we consider the following modified algorithm.}

{
Let $\T_{fp,h}$ be the partition of $\Gamma_{fp}$ obtained from the trace of $\T_{f,h}$.
Let
\begin{equation}\label{Lambda-h-defn}
  {\bL_{h} := \V_{f,h}|_{\Gamma_{fp}}.}
\end{equation}
{For $0 \le n \le N$, we introduce an auxiliary interface variable $\bmu^n= (\mu_n^n,\mu_\tau^n) \in \bL_h$, 
where  
$\mu_n^n  = \bmu^n\cdot\n_f$ and $\mu_\tau^n = \bmu^n\cdot\btau_f$ are used to approximate the Robin data on $\Gamma_{fp}$.} In particular,
\begin{subequations}\label{mu-approx}
\begin{align}
& \mu_n^n \simeq -\gamma_f \left(\dt\bbeta_p^n + \u_p^{n}\right)\cdot\n_p
+ (\bs_p^{n} \n_p)\cdot\n_p, \label{mu-n-approx} \\
& \mu_{\tau}^{n} \simeq - \gamma_f \dt\bbeta_p^n\cdot\btau_{p}
+ (\bs_p^{n}\n_p)\cdot\btau_{p} - \gamma_f \gamma_{BJS} (\bs_f^{n}\n_f)\cdot\btau_{f}. \label{mu-t-approx}
\end{align}
\end{subequations}
}
{The Robin-Robin partitioned algorithm is as follows.}

\medskip
Let, for all $\bchi \in \bL_h$,
\begin{align*}
    \<\mu_n^0,\chi_n\>_{\Gamma_{fp}} &= \<-\gamma_f(\dt\bbeta_p^0+\u_p^0)\cdot\n_p + \bs_p^0\n_p\cdot\n_p,\chi_n\>_{\Gamma_{fp}}, \cl
    \<\mu_{\tau}^0,\chi_\tau\>_{\Gamma_{fp}}
    &= \<-\gamma_f\dt\bbeta_p^0\cdot\tau_{p}+\bs_p^0\n_p\cdot\tau_{p}  -\gamma_f\gamma_{BJS}(\bs_f^{0}\n_f)\cdot\btau_{f},\chi_\tau\>_{\Gamma_{fp}}. \el
  \end{align*}
  
In the above, and in the following, we write the two equations separately for clarity, with the understanding that they represent a single equation for $\bmu^0$, i.e., the sum of the two equations holds.

For $n = 0,\dots,N-1$, solve
\begin{enumerate}
\item Stokes problem: find $(\u_{f}^{n+1}, p_{f}^{n+1}) \in \V_{f,h}\times W_{f,h}$ such that for all $(\bv_f,w_f) \in \V_{f,h}\times W_{f,h}$,
\begin{align}
  & \left(\rho_f {\dt \u_{f}^{n+1}},\bv_{f}\right)_{\O_f} + a_{f}(\u_{f}^{n+1},\bv_{f}) + b_f(\bv_{f},p_{f}^{n+1}) + \gamma_f\<\u_f^{n+1}\cdot\n_f,\bv_f\cdot\n_f\>_{\Gamma_{fp}}
  \nonumber \\
  & \qquad  +  \gamma_f \<\u_f^{n+1}\cdot\btau_{f},\bv_f\cdot\btau_{f}\>_{\Gamma_{fp}} = (\f_{f},\bv_{f})_{\O_f} + \<\mu_n^n,\bv_f\cdot\n_{f}\>_{\Gamma_{fp}} 
   + \<\mu_{\tau}^n,\bv_f\cdot\btau_{f}\>_{\Gamma_{fp}}\label{stokes-weak-1}\\
 & - b_f(\u^{n+1}_{f},w_{f}) = (q_{f},w_{f})_{\O_f}, \label{stokes-weak-2}
\end{align}
which corresponds to the weak imposition of the Robin boundary conditions on $\Gamma_{fp}$:
\begin{subequations}\label{stokes-robin-mu-modif}
\begin{align}
  & \gamma_f\u_f^{n+1}\cdot\n_f + (\bs_f^{n+1} \n_f)\cdot\n_f = \mu_n^n, 
  \label{stokes-robin-mu-1-modif} \\
  &\gamma_f\u_f^{n+1}\cdot\btau_{f}
 +  (\bs_f^{n+1}\n_f)\cdot\btau_{f}
  = \mu_{\tau}^n. 
  \label{stokes-robin-mu-2-modif}
\end{align}
\end{subequations}

\item Biot problem: find $(\bbeta_{p}^{n+1},\u_p^{n+1},p_{p}^{n+1}) \in \X_{p,h},\V_{p,h},W_{p,h}$ such that for all
$(\bxi_p,\bv_p,w_p) \in \X_{p,h}\times\V_{p,h}\times W_{p,h}$,
\begin{align}
  &\big(\rho_p\dtt\bbeta_{p}^{n+1},\bxi_{p}\big)_{\O_p}
  + a^e_{p}(\bbeta_{p}^{n+1},\bxi_{p})
  + a^d_{p}(\u_{p}^{n+1},\bv_{p})
  + \alpha b_p(\bxi_{p},p_{p}^{n+1})
  + b_p(\bv_{p},p_{p}^{n+1})
  \nonumber \\
  & \qquad\qquad  + \gamma_p\<(\u_p^{n+1} + \dt\bbeta_p^{n+1})\cdot\n_p,
        (\bv_p + \bxi_p)\cdot\n_p\>_{\Gamma_{fp}}
  + \gamma_p\<\dt\bbeta_p^{n+1}\cdot\btau_p,\bxi_p\cdot\btau_p\>_{\Gamma_{fp}}
  \nonumber \\
  & \qquad = (\f_{p},\bxi_{p})_{\O_p}
  + \<\mu_n^n - (\gamma_p + \gamma_f)\u_f^{n+1}\cdot\n_f,
  (\bv_p + \bxi_p)\cdot\n_p\>_{\Gamma_{fp}} \nonumber \\
  & \qquad\qquad + \left\<\mu_{\tau}^n
- (\gamma_p + \gamma_f)\u_f^{n+1}\cdot\btau_f -\gamma_p\gamma_{BJS}(\bs_f^{n+1}\n_f)\cdot\btau_{f},
\bxi_p\cdot\btau_p\right\>_{\Gamma_{fp}}, \label{biot-weak-1} \\
& s_0(\dt p_{p}^{n+1},w_{p})_{\O_p} - \alpha b_p(\dt\bbeta_{p}^{n+1},w_{p}) - b_p(\u_{p}^{n+1},w_{p}) = (q_{p},w_{p})_{\O_p}, \label{biot-weak-2}
\end{align}
which corresponds to the weak imposition of the Robin boundary conditions on $\Gamma_{fp}$:
\begin{subequations}\label{biot-robin-mu-modif}
\begin{align}
  & \gamma_p(\u_p^{n+1} + \dt\bbeta_p^{n+1})\cdot\n_p + (\bs_p^{n+1} \n_p)\cdot\n_p
   = \mu_n^n - (\gamma_p + \gamma_f)\u_f^{n+1}\cdot\n_f
  \nonumber \\
  & \qquad = - \gamma_p \u_f^{n+1}\cdot\n_f + (\bs_f^{n+1} \n_f)\cdot\n_f,
  \label{biot-robin-mu-1-modif}\\
& \gamma_p \dt \bbeta_p^{n+1}\cdot\btau_{p} + (\bs_p^{n+1}\n_p)\cdot\btau_{p}
= \mu_{\tau}^n
- (\gamma_p + \gamma_f)\u_f^{n+1}\cdot\btau_f -\gamma_p \gamma_{BJS}(\bs_f^{n+1}\n_f)\cdot\btau_{f}
\nonumber \\
& \qquad = - \gamma_p \u_f^{n+1}\cdot\btau_{f}+(1- \gamma_p\gamma_{BJS})(\bs_f^{n+1}\n_f)\cdot\btau_{f},
\label{biot-robin-mu-2-modif}\\
& \gamma_p (\u_p^{n+1} + \dt\bbeta_p^{n+1})\cdot\n_p - p_p^{n+1}
 = \mu_n^n - (\gamma_p + \gamma_f)\u_f^{n+1}\cdot\n_f
\nonumber \\
& \qquad = -\gamma_p \u_f^{n+1}\cdot\n_f + (\bs_f^{n+1} \n_f)\cdot\n_f,
\label{biot-robin-mu-3-modif}
\end{align}
\end{subequations}
  where the second equalities in \eqref{biot-robin-mu-1-modif} and \eqref{biot-robin-mu-3-modif} 
  are obtained from \eqref{stokes-robin-mu-1-modif} and the second equality in \eqref{biot-robin-mu-2-modif} is obtained from \eqref{stokes-robin-mu-2-modif}.

\item Update: for all $\bchi \in \bL_h$,
%
\begin{subequations}\label{mu-defn}    
\begin{align}
  \<\mu_n^{n+1},\chi_n\>_{\Gamma_{fp}} & = \big\<\mu_n^n -(\gamma_f + \gamma_p)\big((\dt \bbeta_p^{n+1}
  + \u_p^{n+1})\cdot\n_p + \u_f^{n+1}\cdot\n_f\big),\chi_n\big\>_{\Gamma_{fp}},
  \label{mu-n-defn} \\
  \<\mu_{\tau}^{n+1},\chi_\tau\>_{\Gamma_{fp}}  & = \big\<\mu_{\tau}^n -(\gamma_f + \gamma_p)\big(
  \dt \bbeta_p^{n+1} \cdot\btau_p + \u_f^{n+1}\cdot\btau_f
  + \gamma_{BJS} (\bs_f^{n+1}\n_f)\cdot\btau_{f}\big),\chi_\tau\big\>_{\Gamma_{fp}}.
  \label{mu-t-defn} 
\end{align}
\end{subequations}
\end{enumerate}

The above equations, combined with the first equalities in \eqref{biot-robin-mu-1-modif} and \eqref{biot-robin-mu-2-modif}, imply that
\begin{subequations}\label{mu-approx-np1}    
  \begin{align}
     & \mu_n^{n+1} \simeq -\gamma_f \left(\dt \bbeta_p^{n+1} + \u_p^{n+1}\right)\cdot\n_p
  + (\bs_p^{n+1} \n_p)\cdot\n_p,
  \label{mun-approx-np1}\\
  & \mu_{\tau}^{n+1} \simeq  - \gamma_f \dt \bbeta_p^{n+1}\cdot\btau_{p}
  + (\bs_p^{n+1}\n_p)\cdot\btau_{p} - \gamma_f \gamma_{BJS} (\bs_f^{n+1}\n_f)\cdot\btau_{f}.
  \label{mut-approx-np1}
  \end{align}
\end{subequations}

\begin{remark}
    The second equalities in \eqref{biot-robin-mu-modif} indicate that the Robin conditions \eqref{biot-robin-mu}  are weakly imposed in the Biot sub-problem. The expressions \eqref{mun-approx-np1}, combined with \eqref{stokes-robin-mu-modif}  indicate that the Robin conditions \eqref{stokes-robin-mu} are weakly imposed in the Stokes sub-problem. We emphasize that \eqref{stokes-robin-mu-modif}, \eqref{biot-robin-mu-modif}, and \eqref{mu-approx-np1} are not used in the forthcoming analysis. Only \eqref{mu-defn} is used.
    \end{remark}
  \begin{remark}
Assuming that  $\f_f \in (L^2(\Omega_f))^d$, $q_f \in L^2(\Omega_f)$, $\f_p \in (L^2(\Omega_p))^d$, and $q_p \in L^2(\Omega_p)$, the well-posedness of the subdomain problems \eqref{stokes-weak-1}--\eqref{stokes-weak-2} and
\eqref{biot-weak-1}--\eqref{biot-weak-2} can be shown using standard techniques for the Stokes and Biot systems, respectively, using the classical Babuska-Brezzi theory \cite{boffi2013mixed}. We emphasize the inclusion of the term $\|\bv\cdot\n_p\|_{L^2(\Gamma_{fp})}$ in the norm of $\V_p$, cf. \eqref{Vp-norm}. Control of this term is obtained from the term $\gamma_p\<(\u_p^{n+1} + \dt\bbeta_p^{n+1})\cdot\n_p,(\bv_p + \bxi_p)\cdot\n_p\>_{\Gamma_{fp}}$ in \eqref{biot-weak-1}. More precisely, this gives control on $\|(\u_p^{n+1} + \bbeta_p^{n+1})\cdot\n_p\|_{L^2(\Gamma_{fp})}$. Then, the bound on $\|\u_p^{n+1}\cdot\n_p\|_{L^2(\Gamma_{fp})}$ follows from the triangle inequality, the trace inequality, and the bound on $\|\bbeta_p^{n+1}\|_{H^1(\Omega_p)}$.
\end{remark}

\section{Stability analysis}\label{sec:stab}

In the stability and error analysis we consider the case the case $\gamma_{BJS} = 0$, {which corresponds to a no-slip condition as it is typical in fluid-structure interaction}. Moreover,
we assume that $\gamma_f = \gamma_p = \gamma$. Also, for simplicity, in the stability analysis presented in this section we consider no forcing terms, i.e., $\f_f = \f_p = \boldsymbol{0}$ and $q_f = q_p = 0$.
In the analysis, we will use the identities
\begin{align}
& ab = \frac{1}{4}\left((a+b)^2 - (a-b)^2\right), \label{identity-A1}\\
& a(a-b) = \frac{1}{2}(a^2 - b^2 + (a-b)^2). \label{identity-A2}
\end{align}

We note that, due to the choice of $\bL_h$ \eqref{Lambda-h-defn}, \eqref{mu-defn} implies
\begin{align}
  \begin{pmatrix}\mu_n^{n+1}\\ \mu_{\tau}^{n+1} \end{pmatrix}  & =
  \begin{pmatrix} \mu_n^n\\ \mu_{\tau}^{n} \end{pmatrix} 
  - 2\gamma \left( P_{\bL_h} \begin{pmatrix} (\dt \bbeta_p^{n+1} + \u_p^{n+1})\cdot\n_p \\
    \dt \bbeta_p^{n+1} \cdot\btau_p \end{pmatrix}
  + \begin{pmatrix}\u_f^{n+1}\cdot\n_f \\\u_f^{n+1}\cdot\btau_f \end{pmatrix} \right)
  \label{mu-strong}  
\end{align}
where $P_{\bL_h}:(L^2(\Gamma_{fp}))^d \to \bL_h$ is the $L^2$-orthogonal projection satisfying, for any $\bvarphi \in (L^2(\Gamma_{fp}))^d$,
\begin{equation}\label{L2proj-lambda}
\big\<P_{\bL_h}\bvarphi - \bvarphi,\bchi\big\>_{\Gamma_{fp}} = 0 \quad \forall \bchi \in \bL_h.
\end{equation}

A scaling argument similar to the one in \cite[Lemma~5.1]{gvy2014}
shows that that $P_{\bL_h}$ is stable in $\|\cdot\|_{H^{1/2}(\Gamma_{fp})}$:
\begin{equation}\label{H12-stable}
\|P_{\bL_h}\bvarphi\|_{H^{1/2}(\Gamma_{fp})} \le C_\Lambda \|\bvarphi\|_{H^{1/2}(\Gamma_{fp})} \quad \forall \bvarphi \in (H^{1/2}(\Gamma_{fp}))^d.
\end{equation}  

We define the following energy terms, which will be used in the stability estimate.
Let
\begin{align}
  & \mathcal{E}^n = \frac{\rho_f}{2}\|\u_f^n\|_{L^2(\O_f)}^2 + \frac{\rho_p}{2}\|\dt\bbeta_{p}^n\|_{L^2(\O_p)}^2 + \frac{1}{2}\|\bbeta_p^n\|_e^2
  + \frac{s_0}{2}\|p_{p}^n\|_{L^2(\O_p)}^2
  \nonumber \\
  & \qquad\qquad
  + \frac{\Delta t}{4\gamma}\|\mu_n^n\|_{L^2(\Gamma_{fp})}^2
  + \frac{\Delta t}{4\gamma}\|\mu_\tau^n\|_{L^2(\Gamma_{fp})}^2, \label{En-energy} \\
  & \mathcal{D}^n = \|\u_{f}^n\|_f^2 + \|\u_{p}^n\|_d^2, \label{Dn-energy}\\
  & \mathcal{S}^n = \frac{\rho_p}{2}\|\dt\bbeta_{p}^n-\dt\bbeta_{p}^{n-1}\|_{L^2(\O_p)}^2 
+\frac{\rho_f}{2}\|\u_f^n-\u_f^{n-1}\|_{L^2(\O_f)}^2 
+ \frac{1}{2} \|\bbeta_{p}^n-\bbeta_{p}^{n-1}\|_e^2
\nonumber \\
& \qquad\qquad
+ \frac{s_0}{2} \|p_{p}^n-p_{p}^{n-1}\|_{L^2(\O_p)}^2, \label{Sn-energy}
\end{align}
where
\begin{gather*}
\|\bv_f\|_f^2 := a_f(\bv_f,\bv_f) = 2\mu_f\|\D(\bv_{f})\|_{\O_{f}}^2, \quad
\|\bv_p\|_d^2 = a^d_{p}(\bv_p,\bv_p) = \mu_f\|K^{-1/2}\bv_p\|_{\O_p}^2, \\
\|\bxi_p\|_e^2 := a^e_{p}(\bxi_p,\bxi_p) = \|\mu_p^{1/2}\D(\bxi_p)\|_{\O_p}^2
+ \|\lambda_p^{1/2}\div\bxi_p\|_{\O_p}^2.
\end{gather*}
The assumptions on the coefficients imply that there exist positive constants $c_f$, $c_d$, and $c_e$ such that
\begin{equation}\label{coercivity}
  \begin{split}
& \|\bv_f\|_f^2 \ge c_f\|\bv_f\|_{H^1(\O_f)}^2 \ \forall \bv_f \in H^1(\O_f), \quad
\|\bv_p\|_d^2 \ge c_d\|\bv_p\|_{L^2(\O_p)}^2 \ \forall \bv_p \in L^2(\O_p), \\
& \hskip 3.5cm \|\bxi_p\|_e^2 \ge c_e\|\bxi_p\|_{H^1(\O_p)}^2 \ \forall \bxi_p \in H^1(\O_p),
\end{split}
\end{equation}
where Korn's and Poincar\'e inequalities have been utilized in the first and third inequalities, using that $|\Gamma_f^D| > 0$.

{
\begin{theorem}\label{stability}
The following energy inequality holds for the algorithm given in \eqref{stokes-weak-1}--\eqref{stokes-weak-2}, \eqref{biot-weak-1}--\eqref{biot-weak-2}, and \eqref{mu-defn}:
\begin{equation}\label{energy}
\mathcal{E}^N + \Delta t \sum_{n=1}^N \mathcal{D}^n + \sum_{n=1}^N \mathcal{S}^n \le \mathcal{E}^0.
\end{equation}
\end{theorem}
}

\begin{proof}
Taking $\bv_f = \u_f^{n+1}$ and $w_f = p_f^{n+1}$
in \eqref{stokes-weak-1}--\eqref{stokes-weak-2}, we obtain
\begin{align}
& \frac{\rho_f}{\Delta t}\big(\u_f^{n+1}-\u^{n}_{f}, \u^{n+1}_{f}\big)_{\O_f}+ a_{f}(\u_{f}^{n+1},\u_f^{n+1}) \nonumber \\
  & \quad =
  \frac{1}{\gamma}\left\<\mu_n^n - \gamma\u_f^{n+1}\cdot\n_f,
  \gamma\u_f^{n+1}\cdot\n_f\right\>_{\Gamma_{fp}}
  + \frac{1}{\gamma}\left\<\mu_{\tau}^n - \gamma\u_f^{n+1}\cdot\btau_f,
 \gamma\u_f^{n+1}\cdot\btau_f\right\>_{\Gamma_{fp}}\nonumber \\
& \quad = \frac{1}{4\gamma}\int_{\Gamma_{fp}} (\mu_n^n)^2
- \frac{1}{4\gamma}\int_{\Gamma_{fp}} (\mu_n^n - 2 \gamma \u_f^{n+1}\cdot\n_f)^2
+ \frac{1}{4\gamma}\int_{\Gamma_{fp}} (\mu_{\tau}^n)^2
- \frac{1}{4\gamma}\int_{\Gamma_{fp}} (\mu_{\tau}^n - 2 \gamma \u_f^{n+1}\cdot\btau_f)^2,\label{stokes-energy}
\end{align}
where we used \eqref{identity-A1} in the last equality.
Taking 
$\bv_p = \u_p^{n+1}$, $w_p = p_p^{n+1}$, and $\bxi_p = \dt\bbeta_p^{n+1}$ in
\eqref{biot-weak-1}--\eqref{biot-weak-2}, we obtain
\begin{align}
  & \frac{\rho_p}{\Delta t}\big(\dt\bbeta_{p}^{n+1}-\dt\bbeta_{p}^n, \dt\bbeta_{p}^{n+1}\big)_{\O_p}
  + \frac{1}{\Delta t}a^e_{p}(\bbeta_{p}^{n+1},\bbeta_{p}^{n+1}-\bbeta_{p}^n)
  + \frac{1}{\Delta t} s_0 (p_{p}^{n+1} - p_{p}^n,p_p^{n+1})_{\O_p}
  + a^d_{p}(\u_{p}^{n+1},\u_{p}^{n+1})
  \nonumber \\
  &\quad = \frac{1}{\gamma}\left\<\mu_n^n
  - \gamma(\u_p^{n+1} + \dt\bbeta_p^{n+1})\cdot\n_p
- 2 \gamma\u_f^{n+1}\cdot\n_f,
\gamma(\u_p^{n+1} + \dt\bbeta_p^{n+1})\cdot\n_p\right\>_{\Gamma_{fp}}\nonumber \\
& \quad\qquad + \frac{1}{\gamma}\<\mu_{\tau}^n
  - \gamma\dt\bbeta_p^{n+1}\cdot\btau_p - 2 \gamma\u_f^{n+1}\cdot\btau_f,
  \gamma\dt\bbeta_p^{n+1}\cdot\btau_p\>_{\Gamma_{fp}} \nonumber \\
  &\quad
{
  =  \frac{1}{\gamma} \left\<\begin{pmatrix} \mu_n^n\\ \mu_{\tau}^{n} \end{pmatrix}
  - \gamma P_{\bL_h} \begin{pmatrix} (\dt \bbeta_p^{n+1} + \u_p^{n+1})\cdot\n_p \\
    \dt \bbeta_p^{n+1} \cdot\btau_p \end{pmatrix}
  - 2\gamma\begin{pmatrix}\u_f^{n+1}\cdot\n_f \\\u_f^{n+1}\cdot\btau_f \end{pmatrix},
  \gamma P_{\bL_h} \begin{pmatrix} (\dt \bbeta_p^{n+1} + \u_p^{n+1})\cdot\n_p \\
    \dt \bbeta_p^{n+1} \cdot\btau_p \end{pmatrix}
  \right\>_{\Gamma_{fp}} } \nonumber \\
& \quad\qquad
{- \gamma \left\<
\begin{pmatrix} (\dt \bbeta_p^{n+1} + \u_p^{n+1})\cdot\n_p \\
  \dt \bbeta_p^{n+1} \cdot\btau_p \end{pmatrix}
- P_{\bL_h} \begin{pmatrix} (\dt \bbeta_p^{n+1} + \u_p^{n+1})\cdot\n_p \\
  \dt \bbeta_p^{n+1} \cdot\btau_p \end{pmatrix}, \right. } \nonumber\\
& \qquad\qquad\qquad
{\left.
\begin{pmatrix} (\dt \bbeta_p^{n+1} + \u_p^{n+1})\cdot\n_p \\
  \dt \bbeta_p^{n+1} \cdot\btau_p \end{pmatrix}
- P_{\bL_h} \begin{pmatrix} (\dt \bbeta_p^{n+1} + \u_p^{n+1})\cdot\n_p \\
  \dt \bbeta_p^{n+1} \cdot\btau_p \end{pmatrix}
\right\>_{\Gamma_{fp}}, }
  \label{biot-energy}
    \end{align}
where we used \eqref{L2proj-lambda} in the last equality.
Then, using \eqref{identity-A1} and \eqref{mu-strong}, and dropping the last term in
\eqref{biot-energy}, we obtain
\begin{align}
   & \frac{\rho_p}{\Delta t}\big(\dt\bbeta_{p}^{n+1}-\dt\bbeta_{p}^n, \dt\bbeta_{p}^{n+1}\big)_{\O_p}
  + \frac{1}{\Delta t}a^e_{p}(\bbeta_{p}^{n+1},\bbeta_{p}^{n+1}-\bbeta_{p}^n)
  + \frac{1}{\Delta t} s_0 (p_{p}^{n+1} - p_{p}^n,p_p^{n+1})_{\O_p}
  + a^d_{p}(\u_{p}^{n+1},\u_{p}^{n+1}) 
\nonumber \\
  & \quad \le  \frac{1}{4\gamma}\int_{\Gamma_{fp}} (\mu_n^n - 2 \gamma \u_f^{n+1}\cdot\n_f)^2
  - \frac{1}{4\gamma}\int_{\Gamma_{fp}} (\mu_n^{n+1})^2
  + \frac{1}{4\gamma}\int_{\Gamma_{fp}} (\mu_{\tau}^n - 2 \gamma \u_f^{n+1}\cdot\btau_f)^2
  - \frac{1}{4\gamma}\int_{\Gamma_{fp}} (\mu_{\tau}^{n+1})^2. \label{biot-energy-1}
    \end{align}
Summing \eqref{stokes-energy} and \eqref{biot-energy} and using \eqref{identity-A2} gives
\begin{align*}
& \frac{\rho_f}{2\Delta t} (\u_f^{n+1}, \u_f^{n+1})_{\O_f} + \frac{\rho_f}{2\Delta t} (\u_f^{n+1} - \u_f^{n}, \u_f^{n+1} - \u_f^{n})_{\O_f} + a_{f}(\u_{f}^{n+1},\u_f^{n+1}) \\
& \quad \quad + \frac{\rho_p}{2\Delta t}(\dt\bbeta_{p}^{n+1},\dt\bbeta_{p}^{n+1})_{\O_p}+\frac{\rho_p}{2\Delta t}(\dt\bbeta_{p}^{n+1}-\dt\bbeta_{p}^n,\dt\bbeta_{p}^{n+1}-\dt\bbeta_{p}^n)_{\O_p} + a^d_{p}(\u_{p}^{n+1},\u_{p}^{n+1}) \\
 & \quad \quad 
 + \frac{1}{2\Delta t}a^e_{p}(\bbeta_{p}^{n+1},\bbeta_{p}^{n+1})+
 \frac{1}{2\Delta t} a^e_{p}(\bbeta_{p}^{n+1}-\bbeta_{p}^{n},\bbeta_{p}^{n+1}-\bbeta_{p}^{n})+ \frac{s_0}{2\Delta t}(p_{p}^{n+1},p_p^{n+1})_{\O_p}
  \\
 & \quad \quad + \frac{s_0}{2\Delta t} (p_{p}^{n+1}-p_{p}^{n},p_p^{n+1}-p_{p}^{n})_{\O_p} + \frac{1}{4\gamma}\int_{\Gamma_{fp}} (\mu_n^{n+1})^2 + \frac{1}{4\gamma}\int_{\Gamma_{fp}} (\mu_{\tau}^{n+1})^2\\
 &\quad \le \frac{\rho_f}{2\Delta t} (\u_f^{n}, \u_f^{n})_{\O_f}+ \frac{\rho_p}{2\Delta t}(\dt\bbeta_{p}^n,\dt\bbeta_{p}^n)_{\O_p} + \frac{1}{2\Delta t}a^e_{p}(\bbeta_{p}^{n},\bbeta_{p}^{n}) + \frac{s_0}{2\Delta t} (p_{p}^{n},p_p^{n})_{\O_p}\\
 &\quad \quad +\frac{1}{4\gamma}\int_{\Gamma_{fp}} (\mu_n^{n})^2
  + \frac{1}{4\gamma}\int_{\Gamma_{fp}} (\mu_{\tau}^{n})^2.
\end{align*}
Finally, multiplying by $\Delta t$ and summing over $n$ implies
\begin{align*}
  & \frac{\rho_f}{2}(\u_f^N,\u_f^N)_{\O_f}+ \frac{\rho_p}{2}(\dt\bbeta_{p}^N,\dt\bbeta_{p}^N)_{\O_p}+ \frac{1}{2}a^e_{p}(\bbeta_{p}^N,\bbeta_{p}^N)+ \frac{s_0}{2} (p_{p}^N,p_p^N)_{\O_p} \\
& \qquad  + \frac{\Delta t}{4\gamma}\int_{\Gamma_{fp}} (\mu_n^N)^2 
  + \frac{\Delta t}{4\gamma}\int_{\Gamma_{fp}} (\mu_{\tau}^N)^2  
  + \Delta t \sum_{n=0}^{N-1}\Big(a_{f}(\u_{f}^{n+1},\u_f^{n+1}) + a^d_{p}(\u_{p}^{n+1},\u_{p}^{n+1})\Big) \\
& \qquad +\sum_{n=0}^{N-1} 
  \Big(\frac{\rho_p}{2}(\dt\bbeta_{p}^{n+1}-\dt\bbeta_{p}^n,\dt\bbeta_{p}^{n+1}-\dt\bbeta_{p}^n)_{\O_p}
  +\frac{\rho_f}{2}(\u_f^{n+1}-\u_f^{n},\u_f^{n+1}-\u_f^{n})_{\O_f}
  \\
  & \quad\qquad\qquad
  +\frac{1}{2} a^e_{p}(\bbeta_{p}^{n+1}-\bbeta_{p}^{n},\bbeta_{p}^{n+1}-\bbeta_{p}^{n})
  + \frac{s_0}{2} (p_{p}^{n+1}-p_{p}^{n},p_p^{n+1}-p_{p}^{n})_{\O_p}\Big) \\
  & =\frac{\rho_f}{2}(\u_f^{0},\u_f^{0})_{\O_f} + \frac{\rho_p}{2}(\dt\bbeta_{p}^0,\dt\bbeta_{p}^0)_{\O_p}+ \frac{1}{2}a^e_{p}(\bbeta_{p}^{0},\bbeta_{p}^{0}) + \frac{s_0}{2} (p_{p}^{0},p_p^{0})_{\O_p}
 +\frac{\Delta t}{4\gamma}\int_{\Gamma_{fp}} (\mu_n^{0})^2
  + \frac{\Delta t}{4\gamma}\int_{\Gamma_{fp}} (\mu_{\tau}^{0})^2.
\end{align*}
which gives \eqref{energy}.
\end{proof}

Bound \eqref{energy} provides control on $\u_{f}^{n}$, $\u_{p}^{n}$, $p_p^n$, 
$\bbeta_{p}^{n}$, and $\dt\bbeta_{p}^n$. Bound on $\|p_f^n\|_{L^2(\Omega_f)}$ can be obtained using that $\V_{f,h}\times W_{f,h}$ is a stable Stokes pair satisfying the inf-sup condition \cite[Lemma~4.1]{LSY}
\begin{equation}\label{stokes-inf-sup}
\forall w_f \in W_{f,h}, \quad \sup_{\bv_f \in \V_{f,h}: \bv_f|_{\Gamma_{fp} = 0}} \frac{b_f(\bv_f,w_f)}{\|\bv_f\|_{H^1(\Omega_f)}} \ge \beta_f \|w_f\|_{L^2(\Omega_f)}.
\end{equation}
Additionally, the control on $p_p^n$ depends on $s_0$, which in practice can be very small. A bound on $p_p^n$ independent of $s_0$ can be obtained from the discrete Darcy inf-sup condition \cite{boffi2013mixed}
\begin{equation}\label{darcy-inf-sup}
\forall w_p \in W_{p,h}, \quad \sup_{\bv_p \in \V_{p,h}: \bv_p|_{\Gamma_{fp} = 0}} \frac{b_p(\bv_p,w_p)}{\|\bv_p\|_{H({\rm div}; \O_{p})}} \ge \beta_p \|w_p\|_{L^2(\Omega_p)}.
\end{equation}
Furthermore, we note that the $\Delta t$ scaling in the terms $\Delta t \|\mu_n^N\|_{L^2(\Gamma_{fp})}^2$ and $\Delta t \|\mu_\tau^N\|_{L^2(\Gamma_{fp})}^2$ in $\mathcal{E}^N$, cf. \eqref{En-energy}, implies that the stability bound on $\|\mu_n^N\|_{L^2(\Gamma_{fp})}$ and $\|\mu_\tau^N\|_{L^2(\Gamma_{fp})}$ obtained in \eqref{energy} scales like $\sqrt{\Delta t^{-1}}$. A stability bound on $\bmu^n$ that is optimal with respect to $\Delta t$ can be obtained in the norm $\|\cdot\|_{H^{-1/2}(\Gamma_{fp})}$, which is the dual of $\|\cdot\|_{H^{1/2}(\Gamma_{fp})}$. For simplicity, we present the arguments for the quasistatic Stokes model, where the term $\rho_f (\dt \u_{f}^{n+1},\bv_{f})_{\O_f}$ in \eqref{stokes-weak-1} is not present. Let $\widetilde{\mathcal{E}}^n$ be $\mathcal{E}^n$ without the term $\frac{\rho_f}{2}\|\u_f^n\|_{L^2(\O_f)}^2$ and let
$$
\widetilde{\mathcal{D}}^n = \|\u_{f}^n\|_f^2 + \|\u_{p}^n\|_d^2 + \|p_f^{n}\|_{L^2(\Omega_f)}^2 + \|\bmu^{n}\|_{H^{-1/2}(\Gamma_{fp})}^2 + \|p_p^{n}\|_{L^2(\Omega_p)}^2.
$$

\begin{theorem}\label{stability-improved}
The following energy inequality holds for the quasistatic Stokes version of the algorithm given in \eqref{stokes-weak-1}--\eqref{stokes-weak-2}, \eqref{biot-weak-1}--\eqref{biot-weak-2}, and \eqref{mu-defn}:
\begin{equation}\label{energy-improved}
\widetilde{\mathcal{E}}^N + \Delta t \sum_{n=1}^N \widetilde{\mathcal{D}}^n + \sum_{n=1}^N \mathcal{S}^n \le C\widetilde{\mathcal{E}}^0.
\end{equation}
\end{theorem}

\begin{proof}
A bound on $\|p_f^n\|_{L^2(\Omega_f)}$ can be obtained from \eqref{stokes-inf-sup} and \eqref{stokes-weak-1}. Noting that the restriction $\bv_f|_{\Gamma_{fp}} = 0$ in \eqref{stokes-inf-sup} eliminates all interface terms in \eqref{stokes-weak-1}, we obtain
\begin{equation}\label{pf-stab}
\|p_f^{n+1}\|_{L^2(\Omega_f)} \le C \|\u_f^{n+1}\|_{H^1(\Omega_f)}.
\end{equation}
Similarly, \eqref{darcy-inf-sup} and \eqref{biot-weak-1} with $\bxi_{p} = 0$ imply
\begin{equation}\label{pp-stab}
  \|p_p^{n+1}\|_{L^2(\Omega_p)} \le C \|\u_p^{n+1}\|_{L^2(\Omega_p)}.
\end{equation}
Next, we have
\begin{align*}
  \|\bmu^n\|_{H^{-1/2}(\Gamma_{fp})} & = \sup_{\bvarphi \in (H^{1/2}(\Gamma_{fp}))^d} \frac{\<\bvarphi,\bmu^n\>_{\Gamma_{fp}}}{\|\bvarphi\|_{H^{1/2}(\Gamma_{fp})}}
  = \sup_{\bvarphi \in (H^{1/2}(\Gamma_{fp}))^d} \frac{\<P_{\bL_h}\bvarphi,\bmu^n\>_{\Gamma_{fp}}}{\|\bvarphi\|_{H^{1/2}(\Gamma_{fp})}} \\
&  \le C_\Lambda\sup_{\bvarphi \in (H^{1/2}(\Gamma_{fp}))^d} \frac{\<P_{\bL_h}\bvarphi,\bmu^n\>_{\Gamma_{fp}}}{\|P_{\bL_h}\bvarphi\|_{H^{1/2}(\Gamma_{fp})}}.
\end{align*}
Since $\bL_h = \V_{f,h}|_{\Gamma_{fp}}$, there exists a discrete Stokes extension ${\bf E}_{f,h}: \bL_h \to \V_{f,h}$ such that for each $\bchi \in \bL_h$, ${\bf E}_{f,h} \, \bchi|_{\Gamma_{fp}} = \bchi$ and
$\|{\bf E}_{f,h}\bchi\|_{H^1(\Omega_f)} \le C_{ext}\|\bchi\|_{H^{1/2}(\Gamma_{fp})}$. Therefore,
\begin{align}\label{mu-stab}
  \|\bmu^n\|_{H^{-1/2}(\Gamma_{fp})} & \le C_\Lambda C_{ext}\sup_{\bvarphi \in (H^{1/2}(\Gamma_{fp}))^d} \frac{\<{\bf E}_{f,h} P_{\bL_h}\bvarphi,\bmu^n\>_{\Gamma_{fp}}}{\|{\bf E}_{f,h}P_{\bL_h}\bvarphi\|_{H^1(\Omega_f)}}
  \le C_\Lambda C_{ext}\sup_{\bv_f \in \V_{f,h}}\frac{\<\bv_f,\bmu^n\>_{\Gamma_{fp}}}{\|\bv_f\|_{H^1(\Omega_f)}}
  \nonumber \\
&  \le C\left(\|\u_f^{n+1}\|_{H^1(\Omega_f)} + \|p_f^{n+1}\|_{L^2(\Omega_f)}\right),
\end{align}
using \eqref{stokes-weak-1} for the last inequality. Bound \eqref{energy-improved} follows from combining \eqref{pf-stab}--\eqref{mu-stab} with \eqref{energy}.
\end{proof}

\begin{remark}
In the fully dynamic model, in order to control $\|p_f^{n+1}\|_{L^2(\Omega_f)}$, we need first to bound $\|\dt \u_f^{n+1}\|_{L^2(\Omega_f)}$, which can be done by applying the discrete time differentiation operator $\dt$ to the entire system. We omit further details.
\end{remark}

\section{Time discretization error analysis for the quasistatic model}\label{sec:error}
{
In this section we analyze the time discretization error in the Robin-Robin method given in \eqref{stokes-weak-1}--\eqref{stokes-weak-2}, \eqref{biot-weak-1}--\eqref{biot-weak-2}, and \eqref{mu-defn}. To this end, we consider the semidiscrete continuous in space version of the algorithm and compare its solution to the variational formulation 
\eqref{semi-1}--\eqref{semi-4}. In addition, for the sake of simplicity, we focus on the quasistatic model, where the terms $\rho_f \partial_t \u_{f}$ and $\rho_p \partial_{tt} \bbeta_{p}$ are omitted. In this section we assume that $\Gamma_{fp}$ is at least $C^1$.

Introducing the continuous versions of $\mu_n$ and $\mu_{\tau}$ from \eqref{mu-approx},
\begin{subequations}
\begin{alignat}{1}
\mu_n(t) :=& -\gamma(\d_t \bbeta_p(t) + \u_p(t))\cdot \n_p + \bs_p(t)\n_p \cdot \n_p, \label{muN}\\
\mu_\tau(t) :=& -\gamma \d_t \bbeta_p(t) \cdot \btau_p + \bs_p(t)\n_p \cdot \btau_p, \label{muT}
\end{alignat}
\end{subequations}  
the system \eqref{semi-1}--\eqref{semi-4} in the quasistatic case can be rewritten as
\begin{align}
  & a_{f}(\u_{f},\bv_{f}) + b_f(\bv_{f},p_{f}) +  \gamma \<\u_{f}\cdot \n_f, \bv_{f}\cdot \n_f\>_{\Gamma_{fp}}
+ \gamma \< \u_f \cdot \btau_f, \bv_f \cdot \btau_f\>_{\Gamma_{fp}}
  \nonumber \\
  & \qquad
  = (\f_{f},\bv_{f})_{\O_f} + \<\mu_n, \bv_f\cdot \n_f\>_{\Gamma_{fp}}+ \<\mu_\tau, \bv_f \cdot \btau_f\>_{\Gamma_{fp}} \label{semi-1-mu}\\
 & - b_f(\u_{f},w_{f}) = (q_{f},w_{f})_{\O_f}, \label{semi-2-mu} \\
  & a_p^e(\bbeta_p, \bxi_p)
  + a_p^d (\u_p, \bv_p)
  + \alpha b_p(\bxi_p, p_p)
  + b_p(\bv_p, p_p)
  \nonumber \\
  &\qquad\quad \ + \gamma \<(\u_p + \dt \bbeta_p)\cdot \n_p, (\bv_p + \bxi_p)\cdot \n_p\>_{\Gamma_{fp}}+ \gamma \<\dt \bbeta_p\cdot \btau_p, \bxi_p\cdot \btau_p\>_{\Gamma_{fp}}
  \nonumber \\
&\qquad = (\f_p, \bv_p)_{\O_p} + \<\mu_n - 2\gamma \u_f\cdot \n_f, (\bv_p + \bxi_p)\cdot \n_p\>_{\Gamma_{fp}}+ \<\mu_\tau - 2\gamma \u_f\cdot \btau_f, \bxi_p \cdot \btau_p\>_{\Gamma_{fp}} \label{semi-3-mu}\\
&s_0(\dt p_p, w_p)_{\O_p} - \alpha b_p(\dt \bbeta_p, w_p)- b_p(\u_p, w_p) =(q_p,w_p)_{\O_p}. \label{semi-4-mu}
\end{align}
}

Let us define the error terms for $i=f,p$ as follows:

\begin{align*}
  \U_i^{n+1} = \u_i(t_{n+1}) - \u_i^{n+1},
  \quad P_i^{n+1} =& p_i(t_{n+1}) - p_i^{n+1},
  \quad \E_p^{n+1} = \bbeta_p(t_{n+1}) - \bbeta_p^{n+1}.
\end{align*}
Additionally, we define a splitting error and a time discretization error operators acting on the exact solutions as follows:
\begin{alignat*}{1}
\mathbb{S}^{n+1}(\mathbf{\phi}) &:= \mathbf{\phi}(t_{n+1}) - \mathbf{\phi}(t_n), \\
\mathbb{T}^{n+1}(\phi) &:= \d_t \phi(t_{n+1}) - \dt  \phi(t_{n+1}), 
\end{alignat*}
{where we recall $\dt \phi(t_{n+1}) := \frac{\phi(t_{n+1}) - \phi(t_n)}{\Delta t}$.}
We note that the argument of $\mathbb{T}^{n+1}$ can be either a vector or a scalar. 
Next, we define variations of $\mu_n$ and $\mu_\tau$ with discrete time derivatives:
\begin{subequations}\label{muTildes}
\begin{alignat}{1}
\tilde{\mu}_n(t_n) :=& -\gamma (\dt \bbeta_p (t_n) + \u_p(t_n)) \cdot \n_p + \bs_p(t_n)\n_p \cdot \n_p, \label{muTildeN} \\
\tilde{\mu}_\tau (t_n) :=& -\gamma \dt \bbeta_p (t_n) \cdot \btau_p + \bs_p(t_n)\n_p \cdot \btau_p \label{muTildeT}.
\end{alignat}
\end{subequations}
Consequently, it follows that
\begin{subequations}\label{muTs}
\begin{alignat}{1}
\mu_n(t_n) =& \tilde{\mu}_n(t_n) - \gamma \mathbb{T}^n (\bbeta_p)\cdot \n_p, \label{muTn} \\
\mu_\tau (t_n) =&  \tilde{\mu}_\tau(t_n) - \gamma \mathbb{T}^n (\bbeta_p)\cdot \btau_p. \label{muTt}
\end{alignat}
\end{subequations}

We define the interface error terms:
\begin{subequations}
\begin{alignat}{1}
M_n^n :=& \tilde{\mu}_n(t_n) - \mu_n^n,\label{Mn}\\
M_{\tau}^{n} :=&  \tilde{\mu}_\tau(t_n) - \mu_\tau^n.\label{Mtau}
\end{alignat}
\end{subequations}
Consequently, we have
\begin{align}
\mu_n (t_{n+1}) - \mu_n^n &= (\mu_n (t_{n+1}) - \mu_n(t_n) )+ \mu_n(t_n) - \mu_n^n \nonumber\\
&= \mathbb{S}^{n+1} (\mu_n) + \mu_n(t_n) - \tilde{\mu}_n(t_n) + \tilde{\mu}_n(t_n) - \mu_n^n \nonumber \\
&= \mathbb{S}^{n+1}(\mu_n) - \gamma \mathbb{T}^n(\bbeta_p)\cdot \n_p + M_n^n \label{MnRelat},
\end{align}
and, by the same argument,
\begin{equation}\label{MtRelat}
\mu_\tau (t_{n+1}) - \mu_\tau^n = \mathbb{S}^{n+1}(\mu_\tau) - \gamma \mathbb{T}^n(\bbeta_p)\cdot \btau_p + M_\tau^n.
\end{equation}
{We note that in the continuous in space case considered in this section, the update \eqref{mu-strong} simplifies to}
\begin{subequations}\label{mu-strong-cont}  
\begin{align}
  \mu_n^{n+1}& = \mu_n^n - 2\gamma\big((\dt \bbeta_p^{n+1}
  + \u_p^{n+1})\cdot\n_p + \u_f^{n+1}\cdot\n_f\big),
  \label{mu-n-strong-cont} \\
  \mu_{\tau}^{n+1}  & = \mu_{\tau}^n - 2\gamma\big(
  \dt \bbeta_p^{n+1} \cdot\btau_p + \u_f^{n+1}\cdot\btau_f \big).
  \label{mu-t-strong-cont} 
\end{align}
\end{subequations}

{
\begin{theorem}
The following error estimate holds for the quasistatic continuous in space version of the algorithm given in \eqref{stokes-weak-1}--\eqref{stokes-weak-2}, \eqref{biot-weak-1}--\eqref{biot-weak-2}, and \eqref{mu-defn}:  
\begin{alignat}{1}
  & \max_{1\leq n \leq N}\big(
  \|\E_p^{n}\|_e
  + \|P_p^{n}\|_{L^2(\Omega_p)} 
  \big)
  + \left(\Delta t \sum_{n=1}^{N}\|\U_f^{n} \|_f^2\right)^{1/2}
  + \left(\Delta t \sum_{n=1}^{N} \| \U_p^{n}\|_d^2 \right)^{1/2}
  + \left(\Delta t \sum_{n=1}^{N}\|P_f^{n} \|_{L^2(\Omega_f)}^2 \right)^{1/2}
  \nonumber \\
  & \quad
  + \left(\Delta t \sum_{n=1}^{N}\|P_p^{n} \|_{L^2(\Omega_p)}^2 \right)^{1/2}
  + \left(\Delta t \sum_{n=1}^{N}\|\M^{n} \|_{H^{-1/2}(\Gamma_{fp})}^2 \right)^{1/2}
  = \mathcal{O}(\sqrt{T}\Delta t).
\label{error-estimate}
\end{alignat}  
\end{theorem}
}

\begin{proof}
{Subtracting \eqref{stokes-weak-1}--\eqref{stokes-weak-2} from \eqref{semi-1-mu}--\eqref{semi-2-mu} at $t = t_{n+1}$} and applying \eqref{MnRelat} and \eqref{MtRelat} yields
\begin{align}
  &
  a_f(\U_f^{n+1},\bv_f) + b_f(\bv_f, P_f^{n+1})
  + \gamma \< \U_f^{n+1} \cdot \n_f, \bv_f\cdot \n_f\>_{\Gamma_{fp}} + \gamma \<\U_f^{n+1}\cdot \btau_f, \bv_f\cdot \btau_f\>_{\Gamma_{fp}}\nonumber\\
  & \quad = \<\mu_n(t_{n+1}) - \mu_n^n, \bv_f\cdot \n_f\>_{\Gamma_{fp}}+ \<\mu_\tau(t_{n+1}) - \mu_\tau^n, \bv_f\cdot \btau_f\>_{\Gamma_{fp}}\nonumber \\
&\quad = \<M_n^n, \bv_f\cdot \n_f\>_{\Gamma_{fp}}+ \<M_\tau^n, \bv_f\cdot \btau_f\>_{\Gamma_{fp}}  \nonumber \\
&\quad \quad  + \< \mathbb{S}^{n+1}(\mu_n) - \gamma \mathbb{T}^n(\bbeta_p)\cdot \n_p,\bv_f \cdot \n_f\>_{\Gamma_{fp}}+ \< \mathbb{S}^{n+1}(\mu_\tau) - \gamma \mathbb{T}^n(\bbeta_p)\cdot \btau_p,\bv_f\cdot \btau_f\>_{\Gamma_{fp}}, \label{eq:St_sub1} \\
& - b_f(\U_f^{n+1},w_f) = 0. \label{eq:St_sub2}
\end{align}
Take $\bv_f = \U_f^{n+1}$ and $w_f = P_f^{n+1}$ in \eqref{eq:St_sub1}-\eqref{eq:St_sub2} and add them up to obtain
\begin{align}
  &
  a_f(\U_f^{n+1},\U_f^{n+1}) + \gamma \< \U_f^{n+1} \cdot \n_f, \U_f^{n+1}\cdot \n_f\>_{\Gamma_{fp}} + \gamma \<\U_f^{n+1}\cdot \tau_f, \U_f^{n+1}\cdot \tau_f\>_{\Gamma_{fp}} \nonumber\\
&\quad = \<M_n^n, \U_f^{n+1}\cdot \n_f\>_{\Gamma_{fp}}+ \<M_\tau^n, \U_f^{n+1}\cdot \btau_f\>_{\Gamma_{fp}}+ I_1, \label{eq:St_sub3}
\end{align}
where 
\begin{align}
& I_1 := \< \mathbb{S}^{n+1}(\mu_n) - \gamma \mathbb{T}^n(\bbeta_p)\cdot \n_p,\U_f^{n+1} \cdot \n_f\>_{\Gamma_{fp}}+ \< \mathbb{S}^{n+1}(\mu_\tau) - \gamma \mathbb{T}^n(\bbeta_p)\cdot \btau_p,\U_f^{n+1}\cdot \btau_f\>_{\Gamma_{fp}}, \label{H1}
\end{align}
%
%

We can manipulate \eqref{eq:St_sub3} as follows:
\begin{align*}
  & 
  a_f(\U_f^{n+1},\U_f^{n+1})
  = \frac{1}{\gamma}\<M_n^n - \gamma \U_f^{n+1} \cdot \n_f, \gamma \U_f^{n+1}\cdot \n_f\>_{\Gamma_{fp}} +\frac{1}{\gamma} \<M_\tau^n - \gamma \U_f^{n+1} \cdot \btau_f, \gamma \U_f^{n+1}\cdot \btau_f\>_{\Gamma_{fp}} + I_1.
\end{align*}
By applying \eqref{identity-A1} , we have
\begin{align}
  &   a_f(\U_f^{n+1},\U_f^{n+1})
= \frac{1}{4\gamma} \big( \|M_n^n\|_{L^2(\Gamma_{f_p})}^2 - \|M_n^n - 2\gamma \U_f^{n+1}\cdot \n_f\|_{L^2(\Gamma_{f_p})}^2\big)
\nonumber \\
& \quad
+ \frac{1}{4\gamma}\big(\|M_\tau^n\|_{L^2(\Gamma_{f_p})}^2 - \|M_\tau^n -2\gamma \U_f^{n+1}\cdot \btau_f\|_{L^2(\Gamma_{f_p})}^2\big) + I_1. \label{boundSTOKES}
\end{align}

{For the Biot problem, subtracting \eqref{biot-weak-1}--\eqref{biot-weak-2} from \eqref{semi-3-mu}--\eqref{semi-4-mu} at $t = t_{n+1}$}, we obtain
\begin{align}
  &
  a_p^e(\E_p^{n+1}, \bxi_p)
  + a_p^d (\U_p^{n+1}, \bv_p)
  + \alpha b_p(\bxi_p, P_p^{n+1})
  + b_p(\bv_p, P_p^{n+1})
  \nonumber\\
  & \qquad
  + \gamma \<(\U_p^{n+1} + \ddt \E_p^{n+1})\cdot \n_p, (\bv_p + \bxi_p)\cdot \n_p\>_{\Gamma_{fp}}
  + \gamma \<\mathbb{T}^{n+1}(\bbeta_p)\cdot \n_p,(\bv_p + \bxi_p)\cdot \n_p\>_{\Gamma_{fp}}
  \nonumber\\
  &\qquad
  + \gamma \<\ddt \E_p^{n+1}\cdot \btau_p, \bxi_p\cdot \btau_p\>_{\Gamma_{fp}} + \gamma \< \mathbb{T}^{n+1}(\bbeta_p)\cdot \btau_p, \bxi_p \cdot \btau_p\>_{\Gamma_{fp}} \nonumber\\
&\quad =  \< \mu_n(t_{n+1})- \mu_n^n  - 2\gamma \U_f^{n+1}\cdot \n_f, (\bv_p + \bxi_p)\cdot \n_p\>_{\Gamma_{fp}} + \<\mu_\tau(t_{n+1}) - \mu_\tau^n - 2\gamma \U_f^{n+1}\cdot \btau_f, \bxi_p \cdot \btau_p\>_{\Gamma_{fp}} \label{eq:Bi_sub1} \\
&s_0(\ddt P_p^{n+1}, w_p)_{\O_p} + s_0(\mathbb{T}^{n+1}(p_p),w_p)_{\O_p} - \alpha b_p(\ddt\E_p^{n+1},w_p) - \alpha b_p(\mathbb{T}^{n+1}(\bbeta_p),w_p)  \nonumber\\
& \qquad - b_p(\U_p^{n+1},w_p) = 0. \label{eq:Bi_sub2}
\end{align}
Letting $\bv_p = \U_p^{n+1}$, $\bxi_p = \ddt \E_p^{n+1}$, and $w_p = P_p^{n+1}$ in \eqref{eq:Bi_sub1}--\eqref{eq:Bi_sub2}, adding them up and using \eqref{MnRelat} and \eqref{MtRelat}, we have
\begin{align}
  &
  a_p^e(\E_p^{n+1}, \ddt \E_p^{n+1})
  + s_0(\ddt P_p^{n+1}, P_p^{n+1})_{\O_p}
  + a_p^d (\U_p^{n+1}, \U_p^{n+1})
  \nonumber \\
&\qquad + \gamma \<(\U_p^{n+1} + \ddt \E_p^{n+1})\cdot \n_p, (\U_p^{n+1} + \ddt\E_p^{n+1})\cdot \n_p\>_{\Gamma_{fp}}+ \gamma \<\ddt \E_p^{n+1}\cdot \btau_p, \ddt \E_p^{n+1}\cdot \btau_p\>_{\Gamma_{fp}} \nonumber\\
&\quad =  \<M_n^n  - 2\gamma \U_f^{n+1}\cdot \n_f, (\U_p^{n+1} + \ddt\E_p^{n+1})\cdot \n_p\>_{\Gamma_{fp}} + \<M_\tau^n - 2\gamma \U_f^{n+1}\cdot \btau_f, \ddt \E_p^{n+1} \cdot \btau_p\>_{\Gamma_{fp}} \nonumber\\
  &\qquad + \< \mathbb{S}^{n+1}(\mu_n) - \gamma \mathbb{T}^n(\bbeta_p)\cdot \n_p,(\U_p^{n+1} + \ddt \E_p^{n+1})\cdot \n_p\>_{\Gamma_{fp}} \nonumber\\
&\qquad   
  + \< \mathbb{S}^{n+1}(\mu_\tau) - \gamma \mathbb{T}^n(\bbeta_p)\cdot \btau_p,\ddt \E_p^{n+1}\cdot \btau_p\>_{\Gamma_{fp}}\nonumber\\
&\qquad  - \gamma \<\mathbb{T}^{n+1}(\bbeta_p)\cdot \n_p,(\U_p^{n+1} + \ddt \E_p^{n+1})\cdot \n_p\>_{\Gamma_{fp}}- \gamma \< \mathbb{T}^{n+1}(\bbeta_p)\cdot \btau_p, \ddt \E_p^{n+1} \cdot \btau_p\>_{\Gamma_{fp}} \nonumber\\
    &\qquad
    - s_0(\mathbb{T}^{n+1}(p_p),P_p^{n+1})_{\O_p} + \alpha b_p(\mathbb{T}^{n+1}(\bbeta_p),P_p^{n+1}). \label{eq:Bi_sub3}
\end{align}
Eq.~\eqref{eq:Bi_sub3} can be rewritten as follows:
\begin{alignat}{1}
  &
  a_p^e(\E_p^{n+1}, \ddt \E_p^{n+1})
  + s_0(\ddt P_p^{n+1}, P_p^{n+1})_{\O_p}
  + a_p^d (\U_p^{n+1}, \U_p^{n+1})
  \nonumber \\
&\quad =  \<M_n^n  - 2\gamma \U_f^{n+1}\cdot \n_f- \gamma (\U_p^{n+1} + \ddt \E_p^{n+1})\cdot \n_p, (\U_p^{n+1} + \ddt\E_p^{n+1})\cdot \n_p\>_{\Gamma_{fp}} \nonumber \\
&\qquad  + \<M_\tau^n - 2\gamma \U_f^{n+1}\cdot \btau_f- \gamma \ddt \E_p^{n+1} \cdot \btau_p, \ddt \E_p^{n+1} \cdot \btau_p\>_{\Gamma_{fp}} \nonumber \\
&\qquad  + \< A_n^n ,(\U_p^{n+1} + \ddt \E_p^{n+1})\cdot \n_p\>_{\Gamma_{fp}}+ \< A_\tau,\ddt \E_p^{n+1}\cdot \btau_f\>_{\Gamma_{fp}}+ I_2, \label{eq:Bi_sub4}
\end{alignat}
where
\begin{subequations}
\begin{alignat}{1}
A_n^n & := \mathbb{S}^{n+1}(\mu_n) - \gamma (\mathbb{T}^n(\bbeta_p) + \mathbb{T}^{n+1}(\bbeta_p))\cdot \n_p, \label{aN}\\
A_\tau^n & := \mathbb{S}^{n+1}(\mu_\tau) - \gamma (\mathbb{T}^n(\bbeta_p)+\mathbb{T}^{n+1}(\bbeta_p))\cdot \btau_p, \label{aT}\\
I_2 & := 
- s_0(\mathbb{T}^{n+1}(p_p),P_p^{n+1})_{\O_p} + \alpha b_p(\mathbb{T}^{n+1}(\bbeta_p),P_p^{n+1}). \label{H3}
\end{alignat}
\end{subequations}
Using \eqref{identity-A1}, we obtain from \eqref{eq:Bi_sub4}:
\begin{align}
  &
   a_p^e(\E_p^{n+1}, \ddt \E_p^{n+1})
  + s_0(\ddt P_p^{n+1}, P_p^{n+1})_{\O_p}
  + a_p^d (\U_p^{n+1}, \U_p^{n+1})
  \nonumber\\
&\quad = \frac{1}{4\gamma}\|M_n^n - 2\gamma \U_f^{n+1}\cdot \n_f\|_{L^2(\Gamma_{f_p})}^2 - \frac{1}{4\gamma}\|M_n^n - 2\gamma \U_f^{n+1} \cdot \n_f - 2\gamma (\U_p^{n+1} + \ddt \E_p^{n+1})\cdot \n_p \|_{L^2(\Gamma_{f_p})}^2 \nonumber\\
&\qquad + \frac{1}{4\gamma}\|M_\tau^n - 2\gamma \U_f^{n+1} \cdot \btau_f\|_{L^2(\Gamma_{f_p})}^2 - \frac{1}{4\gamma} \| M_\tau^n - 2\gamma \U_f^{n+1} \cdot \btau_f - 2 \gamma \ddt \E_p^{n+1}\cdot \btau_p \|_{L^2(\Gamma_{f_p})}^2\nonumber\\
&\qquad + \<A_n^n,(\U_p^{n+1} + \ddt \E_p^{n+1})\cdot \n_p\>{\Gamma_{fp}}+ \<A_\tau^n,\ddt \E_p^{n+1}\cdot \btau_p\>_{\Gamma_{fp}}+ I_2. \label{eq:Biot_error}
\end{align}
For the second term on the right-hand side above, we have
\begin{alignat}{1}
&M_n^n - 2\gamma \U_f^{n+1} \cdot \n_f - 2\gamma(\U_p^{n+1} + \ddt \E_p^{n+1})\cdot \n_p  \\
&\quad = \tilde{\mu}_n(t_n) - \mu_n^n - 2\gamma \u_f(t_{n+1})\cdot \n_f -2\gamma(\u_p(t_{n+1}) + \d_t \bbeta_p(t_{n+1}))\cdot\n_p \nonumber\\
&\quad \quad \ + 2 \gamma \mathbb{T}^{n+1}( \bbeta_p)\cdot \n_p  + 2\gamma \u_f^{n+1}\cdot \n_f + 2\gamma(\u_p^{n+1} + \ddt \bbeta_p^{n+1})\cdot\n_p\nonumber\\
&\quad = \tilde{\mu}_n(t_n) - \mu_n^n + 2 \gamma \mathbb{T}^{n+1}( \bbeta_p)\cdot \n_p - \mu_n^{n+1} + \mu_n^n \nonumber\\
&\quad = \tilde{\mu}_n(t_n) + 2\gamma \mathbb{T}^{n+1}( \bbeta_p)\cdot \n_p - \mu_n^{n+1} - \tilde{\mu}_n(t_{n+1}) + \tilde{\mu}_n(t_{n+1})  
  - \gamma \mathbb{T}^n(\bbeta_p)\cdot \n_p + \gamma \mathbb{T}^n(\bbeta_p)\cdot \n_p
  \nonumber\\
  &\quad = \tilde{\mu}_n(t_{n+1}) - \mu_n^{n+1}
  - (\tilde{\mu}_n(t_{n+1}) - \gamma \mathbb{T}^{n+1}( \bbeta_p)\cdot \n_p) + (\tilde{\mu}_n(t_{n}) - \gamma \mathbb{T}^{n}( \bbeta_p)\cdot \n_p)
  \nonumber\\
&\quad \quad
+ \gamma(\mathbb{T}^{n+1} (\bbeta_p) + \mathbb{T}^{n}(\bbeta_p))\cdot \n_p \nonumber\\
&\quad = M_n^{n+1} - \mathbb{S}^{n+1}(\mu_n) + \gamma(\mathbb{T}^{n+1} (\bbeta_p) + \mathbb{T}^{n}(\bbeta_p))\cdot \n_p \nonumber\\
&\quad = M_n^{n+1} - A_n^n, \label{nSide}
\end{alignat}
where we have applied \eqref{Mn}, \eqref{eq:mass-conservation}, \eqref{mu-n-strong-cont}, \eqref{muTn}, and \eqref{aN}. By an analogous argument, we have
\begin{alignat}{1}
  &M_\tau^n - 2\gamma \U_f^{n+1} \cdot \btau_f - 2 \gamma \ddt \E_p^{n+1}\cdot \btau_p  = M_\tau^{n+1} - A_\tau^n.
  \label{tauSide}
\end{alignat}
Applying \eqref{nSide} and \eqref{tauSide} to \eqref{eq:Biot_error}, we obtain
\begin{alignat}{1}
  &
   a_p^e(\E_p^{n+1}, \ddt \E_p^{n+1})
  + s_0(\ddt P_p^{n+1}, P_p^{n+1})_{\O_p}
  + a_p^d (\U_p^{n+1}, \U_p^{n+1})
  \nonumber\\
& \quad = \frac{1}{4\gamma}\|M_n^n - 2\gamma \U_f^{n+1}\cdot \n_f\|_{L^2(\Gamma_{f_p})}^2 - \frac{1}{4\gamma} \|M_n^{n+1}\|_{L^2(\Gamma_{f_p})}^2 \nonumber \\
&\quad \ + \frac{1}{4\gamma}\|M_\tau^n - 2\gamma \U_f^{n+1} \cdot \btau_f\|_{L^2(\Gamma_{f_p})}^2 - \frac{1}{4\gamma} \|M_\tau^{n+1}\|_{L^2(\Gamma_{f_p})}^2 \nonumber \\
&\quad \ + \frac{1}{2\gamma}\<A_n^n, M_n^{n+1} + 2\gamma (\U_p^{n+1} + \ddt \E_p^{n+1})\cdot \n_p\>+ \frac{1}{2\gamma}\<A_\tau^n,M_\tau^{n+1} + 2\gamma \ddt \E_p^{n+1}\cdot \btau_p\>\nonumber \\
 &\quad \ - \frac{1}{4\gamma}\|A_n^n\|_{L^2(\Gamma_{f_p})}^2   -\frac{1}{4\gamma} \|A_\tau^n\|_{L^2(\Gamma_{f_p})}^2+ I_2. \label{boundBIOT}
\end{alignat}

We note that, using \eqref{aN} and \eqref{aT}, $I_1$ may be rewritten as
\begin{alignat*}{1}
I_1 =& \frac{1}{2\gamma}\<A_n^n, 2\gamma \U_f^{n+1} \cdot \n_f\>_{\Gamma_{fp}}+ \frac{1}{2\gamma}\<A_\tau^n, 2\gamma \U_f^{n+1}\cdot \btau_f\>_{\Gamma_{fp}} \\
& + \<\gamma \mathbb{T}^{n+1}(\bbeta_p)\cdot\n_p, \U_f^{n+1}\cdot \n_f\>_{\Gamma_{fp}}+ \<\gamma \mathbb{T}^{n+1}(\bbeta_p)\cdot \btau_p,\U_f^{n+1}\cdot \btau_f\>_{\Gamma_{fp}}.
\end{alignat*}

Next, we combine \eqref{boundBIOT} and \eqref{boundSTOKES} and use the above expression to get
\begin{alignat*}{1}
  &
  a_p^e(\E_p^{n+1}, \ddt \E_p^{n+1})
  + s_0(\ddt P_p^{n+1}, P_p^{n+1})_{\O_p}
  + a_f(\U_f^{n+1},\U_f^{n+1})
  + a_p^d (\U_p^{n+1}, \U_p^{n+1})
  \\
& \quad
  = \frac{1}{4\gamma}\big(\|M_n^n\|_{L^2(\Gamma_{f_p})}^2 -\|M_n^{n+1}\|_{L^2(\Gamma_{f_p})}^2\big) 
  +\frac{1}{4\gamma}\big(\|M_\tau^n\|_{L^2(\Gamma_{f_p})}^2 -\|M_\tau^{n+1}\|_{L^2(\Gamma_{f_p})}^2\big) \\
& \qquad \
  + \frac{1}{2\gamma} \<M_n^{n+1} + 2\gamma \U_f^{n+1}\cdot \n_f + 2\gamma (\U_p^{n+1} + \ddt \E_p^{n+1})\cdot \n_p, A_n^n\>_{\Gamma_{fp}}\\
  &\qquad \ + \frac{1}{2\gamma}\<M_\tau^{n+1} + 2\gamma \U_f^{n+1} \cdot \btau_f + 2\gamma \ddt \E_p^{n+1} \cdot \btau_p, A_\tau^n\>_{\Gamma_{fp}} - \frac{1}{4\gamma}\|A_n^n\|_{L^2(\Gamma_{f_p})}^2   -\frac{1}{4\gamma} \|A_\tau^n\|_{L^2(\Gamma_{f_p})}^2 \\
& \qquad \  + I_2 + I_3,
\end{alignat*}
where 
\begin{alignat}{1}\label{H4}
I_3:= \<\gamma \mathbb{T}^{n+1}(\bbeta_p)\cdot\n_p, \U_f^{n+1}\cdot \n_f\>_{\Gamma_{fp}}+ \<\gamma \mathbb{T}^{n+1}(\bbeta_p)\cdot \btau_p,\U_f^{n+1}\cdot \btau_f\>_{\Gamma_{fp}}.
\end{alignat}

From \eqref{nSide} and \eqref{tauSide}, it follows that
\begin{alignat*}{1}
&M_n^{n+1} + 2\gamma \U_f^{n+1}\cdot \n_f + 2\gamma (\U_p^{n+1} + \ddt \E_p^{n+1})\cdot \n_p = M_n^n + A_n^n, \\
&M_\tau^{n+1} + 2\gamma \U_f^{n+1} \cdot \btau_f + 2\gamma \ddt \E_p^{n+1} \cdot \btau_p = M_\tau^n + A_\tau^n.
\end{alignat*}
Therefore, we have
\begin{alignat}{1}
  &
  a_p^e(\E_p^{n+1}, \ddt \E_p^{n+1})
  + s_0(\ddt P_p^{n+1}, P_p^{n+1})_{\O_p}
  + a_f(\U_f^{n+1},\U_f^{n+1})
  + a_p^d (\U_p^{n+1}, \U_p^{n+1})
  \nonumber\\  
&\quad = \frac{1}{4\gamma}\big(\|M_n^n\|_{L^2(\Gamma_{f_p})}^2 -\|M_n^{n+1}\|_{L^2(\Gamma_{f_p})}^2\big) +\frac{1}{4\gamma}\big(\|M_\tau^n\|_{L^2(\Gamma_{f_p})}^2 -\|M_\tau^{n+1}\|_{L^2(\Gamma_{f_p})}^2\big) \nonumber\\
  &\quad \quad   + \frac{1}{2\gamma} \<M_n^n + A_n^n, A_n^n\>_{\Gamma_{fp}}+ \frac{1}{2\gamma}\<M_\tau^{n} + A_\tau^n, A_\tau^n\>_{\Gamma_{fp}}
  - \frac{1}{4\gamma}\|A_n^n\|_{L^2(\Gamma_{f_p})}^2   -\frac{1}{4\gamma} \|A_\tau^n\|_{L^2(\Gamma_{f_p})}^2
  + I_2 + I_3 \nonumber\\
&\quad = \frac{1}{4\gamma}\big(\|M_n^n\|_{L^2(\Gamma_{f_p})}^2 -\|M_n^{n+1}\|_{L^2(\Gamma_{f_p})}^2\big) +\frac{1}{4\gamma}\big(\|M_\tau^n\|_{L^2(\Gamma_{f_p})}^2 -\|M_\tau^{n+1}\|_{L^2(\Gamma_{f_p})}^2\big) \nonumber\\
  &\quad \quad  + \frac{1}{2\gamma} \<M_n^n, A_n^n\>_{\Gamma_{fp}}
  + \frac{1}{2\gamma}\<M_\tau^{n} , A_\tau^n\>_{\Gamma_{fp}}
  + \frac{1}{4\gamma}\|A_n^n\|_{L^2(\Gamma_{f_p})}^2
  + \frac{1}{4\gamma} \|A_\tau^n\|_{L^2(\Gamma_{f_p})}^2
  + I_2 + I_3. \label{stop1}
\end{alignat}
For the mixed terms on the right hand side, let $\M^n = (M_n^n,M_\tau^n)$ and $\A^n = (A_n^n,A_\tau^n)$. Since $\Gamma_{fp}$ is assumed to be $C^1$, it holds that $\A^n \in (H^{1/2}(\Gamma_{fp}))^d$. We have
\begin{align}
& \frac{1}{2\gamma}\<M_n^n, A_n^n\>_{\Gamma_{fp}}
+ \frac{1}{2\gamma}\<M_\tau^{n} , A_\tau^n\>_{\Gamma_{fp}} = \frac{1}{2\gamma}\<\M^n,\A^n\>_{\Gamma_{fp}} \le \frac{1}{2\gamma}\|\M^n\|_{H^{-1/2}(\Gamma_{fp})}\|\A^n\|_{H^{1/2}(\Gamma_{fp})} \nonumber \\
& \qquad \le \frac{\epsilon}{4\gamma}\|\M^n\|_{H^{-1/2}(\Gamma_{fp})}^2 + \frac{1}{4\epsilon\gamma}\|\A^n\|_{H^{1/2}(\Gamma_{fp})}^2,\label{mixed-1}
\end{align}
where we have used the duality of $\|\cdot\|_{H^{-1/2}(\Gamma_{fp})}$ and $\|\cdot\|_{H^{1/2}(\Gamma_{fp})}$, as well as Young's inequality. Using the continuous version of the inf-sup condition \eqref{stokes-inf-sup} and \eqref{eq:St_sub1}, we obtain
\begin{equation}\label{Pf-bound}
  \|P_f^{n+1}\|_{L^2(\Omega_f)} \le \frac{\sqrt{C_f}}{\beta_f}\|\U_f^{n+1}\|_f,
\end{equation}
where $C_f$ is the continuity constant,
$a_f(\u_f,\bv_f) \le C_f \|\u_f\|_{H^1(\Omega_f)}\|\bv_f\|_{H^1(\Omega_f)}$.
Next, similarly to \eqref{mu-stab}, from \eqref{eq:St_sub1} we have
\begin{equation}\label{M-bound-1}
  \frac{1}{C_{ext}}  \|\M^n\|_{H^{-1/2}(\Gamma_{fp})} \le \sqrt{C_f}(1+\gamma C_{tr}^2)\|\U_f^{n+1}\|_f
  + \|P_f^{n+1}\|_{L^2(\Omega_f)}
    + C_{tr}\|\mathbb{S}^{n+1}(\bmu) - \gamma \mathbb{T}^n(\bbeta_p)\|_{H^{-1/2}(\Gamma_{fp})},
\end{equation}
where $C_{tr}$ is the trace inequality constant,
$\|\bv_f\|_{H^{1/2}(\Gamma_{fp})} \le C_{tr} \|\bv_f\|_{H^1(\O_f)}$.
Combining \eqref{Pf-bound} and \eqref{M-bound-1} implies that
\begin{equation}\label{M-bound-2}
\|\M^n\|_{H^{-1/2}(\Gamma_{fp})}^2 \le \widetilde C_f \left( \|\U_f^{n+1}\|_f^2 + \|\B^n\|_{H^{-1/2}(\Gamma_{fp})}^2\right),
\end{equation}
where $\B^n:= \mathbb{S}^{n+1}(\bmu) - \gamma \mathbb{T}^n(\bbeta_p)$. Taking $\epsilon = \frac{\gamma}{\widetilde C_f}$ in \eqref{mixed-1} and using \eqref{M-bound-2} results in 
\begin{equation}\label{mixed-2}
\frac{1}{2\gamma}\<\M^n,\A^n\>_{\Gamma_{fp}} \le \frac14\|\U_f^{n+1}\|_f^2 + \frac14\|\B^n\|_{H^{-1/2}(\Gamma_{fp})}^2 + \frac{\widetilde C_f}{4\gamma^2}\|\A^n\|_{H^{1/2}(\Gamma_{fp})}^2.
\end{equation}

It remains to bound $I_2$ and $I_3$. To this end, we first note that the continuous version of the inf-sup condition \eqref{darcy-inf-sup} and \eqref{eq:Bi_sub1} with $\bxi_p = 0$ imply
\begin{equation}\label{Pp-bound}
\|P_p^{n+1}\|_{L^2(\Omega_p)}^2 \le \widetilde C_p \|\U_p^{n+1}\|_d^2.
\end{equation}
Now, Using the Cauchy-Schwarz and Young's inequalities, we have
\begin{alignat}{1}
  |I_2| & \leq
  s_0^2\widetilde C_p \|\mathbb{T}^{n+1}(p_p)\|_{L^2(\Omega_p)}^2
  + \frac{1}{4\widetilde C_p}\|P_p^{n+1}\|_{L^2(\Omega_p)}^2
  + \alpha^2\widetilde C_p\|\nabla\cdot \mathbb{T}^{n+1}(\bbeta_p)\|_{L^2(\Omega_p)}^2
  + \frac{1}{4\widetilde C_p}\|P_p^{n+1}\|_{L^2(\Omega_p)}^2
 \nonumber \\
 & \quad
 \le s_0^2\widetilde C_p \|\mathbb{T}^{n+1}(p_p)\|_{L^2(\Omega_p)}^2
 + \alpha^2\widetilde C_p\|\nabla\cdot \mathbb{T}^{n+1}(\bbeta_p)\|_{L^2(\Omega_p)}^2
 + \frac12\|\U_p^{n+1}\|_d^2,\label{H3Bound}\\
|I_3| & \quad \leq \frac{\gamma^2}{2\epsilon}\|\mathbb{T}^{n+1}(\bbeta_p)\|_{H^{-1/2}(\Gamma_{fp})}^2
  + \frac{\epsilon}{2}\|\U_f^{n+1}\|_{H^{1/2}(\Gamma_{fp})}^2
\leq \frac{\gamma^2}{2\epsilon}\|\mathbb{T}^{n+1}(\bbeta_p)\|_{H^{-1/2}(\Gamma_{fp})}^2
+ \frac{\epsilon C_{tr}^2}{2c_f}\|\U_f^{n+1}\|_f^2 \nonumber\\
& \quad \leq  \frac{\gamma^2C_{tr}^2}{c_f}\|\mathbb{T}^{n+1}(\bbeta_p)\|_{H^{-1/2}(\Gamma_{fp})}^2 + \frac14\|\U_f^{n+1}\|_f^2,
\label{H4Bound}
\end{alignat}
where we used the trace inequality and the coercivity bound for $a_f$ in \eqref{coercivity} and chose $\epsilon = {c_f}/{(2C_{tr}^2)}$.

Applying bounds \eqref{mixed-2} and \eqref{H3Bound}--\eqref{H4Bound} in \eqref{stop1}, using that
$a(a-b) \ge \frac{1}{2}(a^2 - b^2)$ (cf. \eqref{identity-A2}) for the first four terms on the left-hand side, as well as the coercivity bounds \eqref{coercivity} yields
\begin{alignat}{1}
& \frac{1}{ 2\Delta t}
\bigg( \|\E_p^{n+1}\|_e^2 - \|\E_p^{n}\|_e^2\bigg)
+ \frac{s_0}{2\Delta t}\bigg( \|P_p^{n+1}\|_{L^2(\Omega_p)}^2 - \|P_p^n\|_{L^2(\Omega_p)}^2\bigg)
\nonumber\\
& \qquad
+ \frac{1}{2}  \|\U_f^{n+1} \|_f^2
+ \frac12\| \U_p^{n+1}\|_d^2
+ \frac{1}{4\gamma}\bigg(\|\M^{n+1}\|_{L^2(\Gamma_{f_p})}^2 - \|\M^n\|_{L^2(\Gamma_{f_p})}^2 \bigg)
\le J^n,
\label{stop2}
\end{alignat}
where 
\begin{align*}
& J^n :=
\frac14\|\B^n\|_{H^{-1/2}(\Gamma_{fp})}^2 + \frac{\widetilde C_f}{4\gamma^2}\|\A^n\|_{H^{1/2}(\Gamma_{fp})}^2 +
\frac{1}{4\gamma}\|\A^n\|_{L^2(\Gamma_{f_p})}^2
+ s_0^2\widetilde C_p \|\mathbb{T}^{n+1}(p_p)\|_{L^2(\Omega_p)}^2  \nonumber \\
& \qquad + \alpha^2\widetilde C_p \|\nabla \cdot \mathbb{T}^{n+1}(\bbeta_p)\|_{L^2(\Omega_p)}^2
+ \frac{\gamma^2C_{tr}^2}{c_f}\|\mathbb{T}^{n+1}(\bbeta_p)\|_{L^2(\Gamma_{fp})}^2.
\end{align*}
We multiply \eqref{stop2} by $\Delta t$ and sum from $n=0$ to $m-1$, for $1\leq m \leq N$, which yields
\begin{alignat}{1}
& \frac{1}{ 2}
\|\E_p^{m}\|_e^2
+ \frac{s_0}{2}\|P_p^{m}\|_{L^2(\Omega_p)}^2 
+ \frac{\Delta t}{2}  \sum_{n=0}^{m-1}\|\U_f^{n+1} \|_f^2
+ \frac{\Delta t}{2}  \sum_{n=0}^{m-1} \| \U_p^{n+1}\|_d^2
+ \frac{\Delta t}{4\gamma}\|\M^{m}\|_{L^2(\Gamma_{f_p})}^2
\nonumber\\
& \qquad
\le \Delta t \sum_{n=0}^{m-1} J^n + \|\E_p^{0}\|_e^2
+ \frac{s_0}{2}\|P_p^{0}\|_{L^2(\Omega_p)}^2 + \frac{\Delta t}{4\gamma}\|\M^{0}\|_{L^2(\Gamma_{f_p})}^2.
\label{stop3}
\end{alignat}
{All initial error terms above are zero. In particular, $\E_p^{0} = 0$ and $P_p^0 = 0$ from the choice of initial numerical values, and $\ddt \bbeta_p(t_0) - \ddt\bbeta_p^0 = \u_{s,0} - \u_{s,0} = 0$ implies that $\M^0=0$. For the term $J^n$, which collects all time splitting and time discretization errors, assuming that the solution is sufficiently smooth, it is straightforward to show that
  $J^n = \mathcal{O}(\Delta t^2)$, hence $\Delta t\sum_{n=0}^{m-1} J^n = \mathcal{O}(T\Delta t^2)$.
  Therefore, using that \eqref{stop3} holds for all $1\leq m \leq N$ and combining it with \eqref{Pf-bound}, \eqref{M-bound-2}, and \eqref{Pp-bound}, we obtain \eqref{error-estimate}.}

\end{proof}

\section{Iterative algorithm}\label{sec:iter}

The algorithm discussed in the previous sections solves one Stokes and one Biot problems per time step. This is a computationally efficient choice, but it introduces a splitting error. To avoid this error, one could use the iterative version of the algorithm, which at every time $t^{n+1}$ iterates over the Stokes and Biot sub-problems until convergence. These are Richardson (also called fixed point) iterations. Let $k$ be the index for these iterations. For ease of notation, in the description of the algorithm we will drop the time step index from the variables in the Richardson iterations, i.e., we will write $\phi^{k+1}$ instead of the more rigorous (and bulkier)  $\phi^{n+1, k+1}$. 
{Finally, let $\dt\varphi^{k+1} := (\varphi^{k+1} - \varphi^{n})/{\Delta t}$ and
$\dtt\bbeta_p^{k+1} := \dt \dt\bbeta_p^{k+1} = (\dt\bbeta_p^{k+1} - \dt\bbeta_p^n)/\Delta t$. Recalling that
$\dt \bbeta_p^{n} = (\bbeta_p^{n} - \bbeta_p^{n-1})/{\Delta t}$ for $n \ge 1$, we have
  $\dtt\bbeta_p^{k+1} = (\bbeta_p^{k+1} -2 \bbeta_p^{n}+\bbeta_p^{n-1})/{\Delta t^2}$, while for $n = 0$ we have $\dtt\bbeta_p^{k+1} = ((\bbeta_p^{k+1} - \bbeta_p^{n})/\Delta t - \u_{s,0})/\Delta t$.
  }

{For simplicity we present the method in the case $\gamma_{BJS} = 0$.
At every time $t^{n+1}$, assume that $\u_f^n$, $p_p^n$, $\bbeta_p^{n}$,
  $\mu_n^{n}$, and $\mu_\tau^n$ are known. Set $\mu_n^0 = \mu_n^n$ and $\mu_\tau^0 = \mu_\tau^n$. 
  The following steps are performed at iteration $k+1$, $k \ge 0$:}

\begin{enumerate}
\item Stokes problem: Find $(\u_f^{k+1}, p_f^{k+1})$ such that
\begin{align}
  & \left(\rho_f {\dt\u_{f}^{k+1}},\bv_{f}\right)_{\O_f}+ a_{f}(\u_{f}^{k+1},\bv_{f}) + b_f(\bv_{f},p_{f}^{k+1}) + \gamma_f\<\u_f^{k+1}\cdot\n_f,\bv_f\cdot\n_f\>_{\Gamma_{fp}}
  \nonumber \\
  & \qquad  +  \gamma_f \< \u_f^{k+1}\cdot\btau_{f},\bv_f\cdot\btau_{f}\>_{\Gamma_{fp}} = (\f_{f},\bv_{f})_{\O_f} + \<\mu_n^k,\bv_f\cdot\n_{f}\>_{\Gamma_{fp}}  + \< \mu_{\tau}^k,\bv_f\cdot\btau_{f}\>_{\Gamma_{fp}}  \label{stokes-weak-1_iter}\\
 & - b_f(\u^{k+1}_{f},w_{f}) = (q_{f},w_{f})_{\O_f}. \label{stokes-weak-2_iter}
\end{align}

\item Biot problem: Find $(\bbeta^{k+1}_p, \u_p^{k+1}, p_p^{k+1})$ such that
\begin{align}
 &\big(\rho_p\dtt \bbeta_{p}^{k+1},\bxi_{p}\big)_{\O_p}+ a^d_{p}(\u_{p}^{k+1},\bv_{p}) + a^e_{p}(\bbeta_{p}^{k+1},\bxi_{p}) + b_p(\bv_{p},p_{p}^{k+1}) +
  \alpha b_p(\bxi_{p},p_{p}^{k+1}) \nonumber \\
  & \qquad  + \gamma_p\<(\u_p^{k+1} + \dt\bbeta_p^{k+1})\cdot\n_p,
        (\bv_p + \bxi_p)\cdot\n_p\>_{\Gamma_{fp}}
  + \gamma_p\<\dt\bbeta_p^{k+1}\cdot\btau_p,\bxi_p\cdot\btau_p\>_{\Gamma_{fp}}
  \nonumber \\
  & \quad = (\f_{p},\bxi_{p})_{\O_p}
  + \<\mu_n^k - (\gamma_p + \gamma_f)\u_f^{k+1}\cdot\n_f,
  (\bv_p + \bxi_p)\cdot\n_p\>_{\Gamma_{fp}} \nonumber \\
  & \qquad
  + \left\<\mu_{\tau}^k
- (\gamma_p + \gamma_f) \u_f^{k+1}\cdot\btau_f,
\bxi_p\cdot\btau_p\right\>_{\Gamma_{fp}}, \label{biot-weak-1_iter} \\
& ( s_0 \dt p_{p}^{k+1},w_{p})_{\O_p} - \alpha b_p(\dt\bbeta_{p}^{k+1},w_{p}) - b_p(\u_{p}^{k+1},w_{p}) = (q_{p},w_{p})_{\O_p}. \label{biot-weak-2_iter}
\end{align}

\item Update: 
{
\begin{subequations}\label{mu-defn-iter}   
\begin{align}
  \<\mu_n^{k+1},\chi_n\>_{\Gamma_{fp}} & = \big\<\mu_n^k -(\gamma_f + \gamma_p)\big((\dt \bbeta_p^{k+1}
  + \u_p^{k+1})\cdot\n_p + \u_f^{k+1}\cdot\n_f\big),\chi_n\big\>_{\Gamma_{fp}},
  \label{mu-n-defn-iter} \\
  \<\mu_{\tau}^{k+1},\chi_\tau\>_{\Gamma_{fp}}  & = \big\<\mu_{\tau}^k -(\gamma_f + \gamma_p)\big(
  \dt \bbeta_p^{k+1} \cdot\btau_p + \u_f^{k+1}\cdot\btau_f  \big),\chi_\tau\big\>_{\Gamma_{fp}}.
  \label{mu-t-defn-iter} 
\end{align}
\end{subequations}
}
\item Check the stopping criterion, e.g.
\begin{equation}\label{eq:stopping}
\left\lVert \u_f^{k+1}\cdot\n_f - \u_f^{k}\cdot\n_f \right\rVert_{L^2(\Gamma_{fp})}< \epsilon
\end{equation}
where $\epsilon$ is a given stopping tolerance. If not satisfied, repeat steps 1--4. If satisfied,
set $\u_f^{n+1} = \u_f^{k+1}$, $p_f^{n+1} = p_f^{k+1}$, $\bbeta^{n+1}_p = \bbeta^{k+1}_p$, $\u_p^{n+1} = \u_p^{k+1}$, $p_p^{n+1} = p_p^{k+1}$, and $\boldsymbol{\mu}^{n+1} = \boldsymbol{\mu}^{k+1}$.
\end{enumerate}

We next show that the above algorithm converges.

\begin{theorem}\label{iter-semidiscrete-converge}
  For the iterative algorithm \eqref{stokes-weak-1_iter}--\eqref{eq:stopping} with $\gamma_f = \gamma_p = \gamma$ it holds that as $k \to \infty$, $\u_f^{k+1} \to \u_f^{n+1}$ in $\V_f$, $p_f^{k+1} \to p_f^{n+1}$ in $W_f$, $\bbeta_p^{k+1} \to \bbeta_p^{n+1}$ in $\X_p$, $\u_p^{k+1} \to \u_p^{n+1}$ in $\V_p$, $p_p^{k+1} \to p_p^{n+1}$ in $W_p$, and $\bmu^{k+1} \to \bmu^{n+1}$ in $(L^2(\Gamma_{fp}))^2$.
\end{theorem}

\begin{proof}
Let $\overline\u_f^{k+1} := \u_f^{k+1} - \u_f^k$ for $k \ge 1$ with a similar notation for the rest of the variables. Subtracting \eqref{stokes-weak-1_iter}--\eqref{mu-t-defn-iter} for $k+1$ and $k$ results in the equations in the Stokes region
\begin{align}
  & \frac{1}{\Delta t}\left(\rho_f {\overline\u_{f}^{k+1}},\bv_{f}\right)_{\O_f}+ a_{f}(\ol\u_{f}^{k+1},\bv_{f}) + b_f(\bv_{f},\ol p_{f}^{k+1}) + \gamma\<\ol\u_f^{k+1}\cdot\n_f,\bv_f\cdot\n_f\>_{\Gamma_{fp}}
  \nonumber \\
  & \qquad  + \gamma \< \ol\u_f^{k+1}\cdot\btau_{f},\bv_f\cdot\btau_{f}\>_{\Gamma_{fp}} =  \<\ol\mu_n^k,\bv_f\cdot\n_{f}\>_{\Gamma_{fp}}  + \< \ol\mu_{\tau}^k,\bv_f\cdot\btau_{f}\>_{\Gamma_{fp}}  \label{stokes-diff-1}\\
 & - b_f(\ol\u^{k+1}_{f},w_{f}) = 0, \label{stokes-diff-2}
\end{align} 
and in the Biot region
\begin{align}
  &\frac{1}{\Delta t^2}\big(\rho_p\ol\bbeta_{p}^{k+1},\bxi_{p}\big)_{\O_p} + a^d_{p}(\ol\u_{p}^{k+1},\bv_{p})
  + a^e_{p}(\ol\bbeta_{p}^{k+1},\bxi_{p}) + b_p(\bv_{p},\ol p_{p}^{k+1})
  + \alpha b_p(\bxi_{p},\ol p_{p}^{k+1}) \nonumber \\
  & \qquad  + \gamma\Big\<\big(\ol\u_p^{k+1} + \frac{1}{\Delta t}\ol\bbeta_p^{k+1}\big)\cdot\n_p,
        (\bv_p + \bxi_p)\cdot\n_p\Big\>_{\Gamma_{fp}}
  + \gamma\Big\<\frac{1}{\Delta t}\ol\bbeta_p^{k+1}\cdot\btau_p,\bxi_p\cdot\btau_p\Big\>_{\Gamma_{fp}}
  \nonumber \\
  & \quad = \<\ol\mu_n^k - 2\gamma\ol\u_f^{k+1}\cdot\n_f,
  (\bv_p + \bxi_p)\cdot\n_p\>_{\Gamma_{fp}}
  + \left\<\ol\mu_{\tau}^k - 2\gamma \ol\u_f^{k+1}\cdot\btau_f,
\bxi_p\cdot\btau_p\right\>_{\Gamma_{fp}}, \label{biot-diff-1} \\
& \frac{1}{\Delta t}( s_0 \ol p_{p}^{k+1},w_{p})_{\O_p}
- \alpha b_p\Big(\frac{1}{\Delta t}\ol\bbeta_{p}^{k+1},w_{p}\Big) - b_p(\ol\u_{p}^{k+1},w_{p}) = 0, \label{biot-diff-2}
\end{align}
as well as the updates
\begin{subequations}\label{mu-diff}
\begin{align}
  \<\ol\mu_n^{k+1},\chi_n\>_{\Gamma_{fp}} & = \Big\<\ol\mu_n^k -2\gamma\Big(\Big(\frac{1}{\Delta t}\ol\bbeta_p^{k+1}
  + \ol\u_p^{k+1}\Big)\cdot\n_p + \ol\u_f^{k+1}\cdot\n_f\Big),\chi_n\Big\>_{\Gamma_{fp}}, \label{mu-n-diff} \\
  \<\ol\mu_{\tau}^{k+1},\chi_\tau\>_{\Gamma_{fp}} & = \Big\<\ol\mu_\tau^k - 2\gamma\Big(
  \frac{1}{\Delta t}\ol\bbeta_p^{k+1} \cdot\btau_p + \ol\u_f^{k+1}\cdot\btau_f\Big),\chi_\tau\Big\>_{\Gamma_{fp}}.  \label{mu-t-diff}
\end{align}
\end{subequations}
Taking $\bv_f = \ol\u_f^{k+1}$, $w_f = \ol p_f^{k+1}$, $\bv_p = \ol\u_p^{k+1}$, $w_p = \ol p_p^{k+1}$, and $\bxi_p = \frac{1}{\Delta t}\ol\bbeta_p^{k+1}$  and following the argument in the proof of Theorem~\ref{stability}, we obtain
\begin{align*}
  & \sum_{k=1}^M\left(
  \frac{\rho_f}{\Delta t}(\ol\u_f^{k+1},\ol\u_f^{k+1})_{\O_f}
  + \frac{\rho_p}{\Delta t^2}(\ol\bbeta_p^{k+1},\ol\bbeta_p^{k+1})_{\Omega_p}
  + \frac{1}{\Delta t}a^e_{p}(\ol\bbeta_{p}^{k+1},\ol\bbeta_{p}^{k+1})
  + \frac{s_0}{\Delta t}(\ol p_{p}^{k+1},\ol p_p^{k+1})_{\O_p})
  \right.\\  
  & \qquad
  \left. +  a_{f}(\ol\u_{f}^{k+1},\ol\u_f^{k+1}) + a^d_{p}(\ol\u_{p}^{k+1},\ol\u_{p}^{k+1}) \right)
  + \frac{1}{4\gamma}\int_{\Gamma_{fp}} (\ol\mu_n^{M+1})^2
  + \frac{1}{4\gamma}\int_{\Gamma_{fp}} (\ol\mu_{\tau}^{M+1})^2 \\
& \quad  \le \frac{1}{4\gamma}\int_{\Gamma_{fp}} (\ol\mu_n^{1})^2
  + \frac{1}{4\gamma}\int_{\Gamma_{fp}} (\ol\mu_{\tau}^{1})^2,
\end{align*}
which implies that the series
\begin{align*}
  & \sum_{k=1}^\infty\left(
  \frac{\rho_f}{\Delta t}(\ol\u_f^{k+1},\ol\u_f^{k+1})_{\O_f}
  + \frac{\rho_p}{\Delta t^2}(\ol\bbeta_p^{k+1},\ol\bbeta_p^{k+1})_{\Omega_p}
  + \frac{1}{\Delta t}a^e_{p}(\ol\bbeta_{p}^{k+1},\ol\bbeta_{p}^{k+1})
  + \frac{s_0}{\Delta t}(\ol p_{p}^{k+1},\ol p_p^{k+1})_{\O_p})
  \right.\\  
  & \qquad
  \left. +  a_{f}(\ol\u_{f}^{k+1},\ol\u_f^{k+1}) + a^d_{p}(\ol\u_{p}^{k+1},\ol\u_{p}^{k+1}) \right)
\end{align*}  
is convergent. Therefore $\ol\u_{f}^{k} \to 0$ in $\V_f$, $\ol\bbeta_{p}^{k} \to 0$ in $\X_p$,
$\ol\u_{p}^{k} \to 0$ in $(L^2(\Omega_p))^d$, and $\ol p_p^k \to 0$ in $W_p$. Using the inf-sup condition \eqref{stokes-inf-sup}, we conclude from \eqref{stokes-diff-1} that $\ol p_f^k \to 0$ in $W_f$.

Next, taking $w_p = \nabla\cdot\ol\u_p^{k+1}$ in \eqref{biot-diff-2} implies that $\nabla\cdot\ol\u_p^{k} \to 0$ in $L^2(\Omega_p)$. Also, the convergence $\ol\u_p^{k}\cdot\n_p \to 0$ in $L^2(\Gamma_{fp})$
follows from $\ol\u_{p}^{k} \to 0$ in $(L^2(\Omega_p))^d$ and the discrete trace-inverse inequality $\|\bv_p\cdot\n_p\|_{L^2(\Gamma_{fp})} \le C h^{-1/2}\|\bv_p\|_{L^2(\Omega_p)}$ for $\bv_p \in \V_p$. Therefore $\ol\u_p^{k} \to 0$ in $\V_p$.

Finally, since $\bL_{h} = \V_{f,h}|_{\Gamma_{fp}}$,
we can take $\bv_f = {\bf E}_{f,h}\ol\bmu^k$ in \eqref{stokes-diff-1}, where ${\bf E}_{f,h}$ is the continuous discrete Stokes extension utilized in \eqref{mu-stab}, which implies
$$
\|\ol\bmu^k\|_{L^2(\Gamma_{fp})}^2 \le C(\|\ol\u_f^{k+1}\|_{H^1(\Omega_f)} + \|\ol p_f^{k+1}\|_{L^2(\Omega_f)})\|\ol\bmu^k\|_{H^{1/2}(\Gamma_{fp})}.
$$
Thus, using that $\|\ol\bmu^k\|_{H^{1/2}(\Gamma_{fp})} \le C h^{-1/2}\|\ol\bmu^k\|_{L^2(\Gamma_{fp})}$, we conclude that
$\ol\bmu^k \to 0$ in $(L^2(\Gamma_{fp}))^2$.
\end{proof}

\subsection{Monolithic scheme}

Using the convergence established in Theorem~\ref{iter-semidiscrete-converge}, we can take $k \to \infty$ in \eqref{stokes-weak-1_iter}--\eqref{mu-t-defn-iter} to conclude that the limit functions satisfy the following fully coupled fully implicit scheme: find $(\u_{f}^{n+1}, p_{f}^{n+1}) \in \V_{f,h}\times W_{f,h}$, $(\bbeta_{p}^{n+1},\u_p^{n+1},p_{p}^{n+1}) \in \X_{p,h}\times\V_{p,h}\times W_{p,h}$, and $\bmu_n^{n+1} \in \bL_{h}$ such that for all $(\bv_f,w_f) \in \V_{f,h}\times W_{f,h}$, $(\bxi_p,\bv_p,w_p) \in \X_{p,h}\times\V_{p,h}\times W_{p,h}$, and $\bchi \in \bL_h$,
%
\begin{align}
  & \left(\rho_f \dt \u_{f}^{n+1},\bv_{f}\right)_{\O_f} + a_{f}(\u_{f}^{n+1},\bv_{f}) + b_f(\bv_{f},p_{f}^{n+1}) + \gamma\<\u_f^{n+1}\cdot\n_f,\bv_f\cdot\n_f\>_{\Gamma_{fp}}
  \nonumber \\
  & \qquad  +  \gamma \<\u_f^{n+1}\cdot\btau_{f},\bv_f\cdot\btau_{f}\>_{\Gamma_{fp}} = (\f_{f},\bv_{f})_{\O_f} + \<\mu_n^{n+1},\bv_f\cdot\n_{f}\>_{\Gamma_{fp}} 
   + \<\mu_{\tau}^{n+1},\bv_f\cdot\btau_{f}\>_{\Gamma_{fp}}, \label{stokes-coupled-1}\\
 & - b_f(\u^{n+1}_{f},w_{f}) = (q_{f},w_{f})_{\O_f}, \label{stokes-coupled-2}\\
  &\big(\rho_p\dtt\bbeta_{p}^{n+1},\bxi_{p}\big)_{\O_p}
  + a^e_{p}(\bbeta_{p}^{n+1},\bxi_{p})
  + a^d_{p}(\u_{p}^{n+1},\bv_{p})
  + \alpha b_p(\bxi_{p},p_{p}^{n+1})
  + b_p(\bv_{p},p_{p}^{n+1})
  \nonumber \\
  & \qquad\qquad  + \gamma\<(\u_p^{n+1} + \dt\bbeta_p^{n+1})\cdot\n_p,
        (\bv_p + \bxi_p)\cdot\n_p\>_{\Gamma_{fp}}
  + \gamma\<\dt\bbeta_p^{n+1}\cdot\btau_p,\bxi_p\cdot\btau_p\>_{\Gamma_{fp}}
  \nonumber \\
  & \qquad = (\f_{p},\bxi_{p})_{\O_p}
  + \<\mu_n^{n+1} - 2\gamma\u_f^{n+1}\cdot\n_f,
  (\bv_p + \bxi_p)\cdot\n_p\>_{\Gamma_{fp}} \nonumber \\
  & \qquad\qquad + \left\<\mu_{\tau}^{n+1}
  - 2\gamma\u_f^{n+1}\cdot\btau_f,\bxi_p\cdot\btau_p\right\>_{\Gamma_{fp}},
  \label{biot-coupled-1} \\
& s_0(\dt p_{p}^{n+1},w_{p})_{\O_p} - \alpha b_p(\dt\bbeta_{p}^{n+1},w_{p}) - b_p(\u_{p}^{n+1},w_{p}) = (q_{p},w_{p})_{\O_p}, \label{biot-coupled-2}\\
  & \<(\dt \bbeta_p^{n+1} + \u_p^{n+1})\cdot\n_p + \u_f^{n+1}\cdot\n_f,\chi_n\>_{\Gamma_{fp}} = 0,
  \label{mu-n-coupled} \\
  & \<\dt \bbeta_p^{n+1} \cdot\btau_p + \u_f^{n+1}\cdot\btau_f,\chi_\tau\>_{\Gamma_{fp}} = 0.
  \label{mu-t-coupled} 
\end{align}
%
We note that the Robin data variables $\mu_n^{n+1}$ and $\mu_\tau^{n+1}$ play the role of Lagrange multipliers to impose weakly the velocity continuity conditions \eqref{eq:mass-conservation} and \eqref{Gamma-fp-1} in \eqref{mu-n-coupled}--\eqref{mu-t-coupled}.
In addition, since the Stokes solution satisfies weakly the Robin boundary conditions
$$
\gamma\u_f^{n+1}\cdot\n_f + (\bs_f^{n+1} \n_f)\cdot\n_f = \mu_n^{n+1} \quad \text{and} \quad \gamma\u_f^{n+1}\cdot\btau_{f} +  (\bs_f^{n+1}\n_f)\cdot\btau_{f} = \mu_{\tau}^{n+1}
$$
and the Biot solution satisfies weakly the Robin boundary conditions
\begin{align*}
& \gamma(\u_p^{n+1} + \dt\bbeta_p^{n+1})\cdot\n_p + (\bs_p^{n+1} \n_p)\cdot\n_p
= \mu_n^{n+1} - 2\gamma\u_f^{n+1}\cdot\n_f, \\
&
\gamma \dt \bbeta_p^{n+1}\cdot\btau_{p} + (\bs_p^{n+1}\n_p)\cdot\btau_{p}
= \mu_{\tau}^{n+1} - 2\gamma\u_f^{n+1}\cdot\btau_f, \\
&\gamma(\u_p^{n+1} + \dt\bbeta_p^{n+1})\cdot\n_p - p_p^{n+1}
 = \mu_n^{n+1} - 2\gamma\u_f^{n+1}\cdot\n_f,
\end{align*}
it follows that conditions
$$
\bs_f^{n+1} \n_f + \bs_p^{n+1} \n_p = \boldsymbol{0} \quad \text{and} \quad -(\bs_f^{n+1} \n_f)\cdot\n_f = p_p^{n+1}
$$
are satisfied weakly, i.e., the balance of stress conditions \eqref{balance-stress} hold weakly for the solution of the method.

To the best of our knowledge, the fully implicit scheme \eqref{stokes-coupled-1}--\eqref{mu-t-coupled} has not been studied in the literature. The argument above proves existence of a solution. We next establish uniqueness and stability. To this end, we define a modified energy term
$$
\widehat{\mathcal{E}}^n = \frac{\rho_f}{2}\|\u_f^n\|_{L^2(\O_f)}^2 + \frac{\rho_p}{2}\|\dt\bbeta_{p}^n\|_{L^2(\O_p)}^2 + \frac{1}{2}\|\bbeta_p^n\|_e^2
  + \frac{s_0}{2}\|p_{p}^n\|_{L^2(\O_p)}^2,
$$
  which is the energy term $\mathcal{E}^n$ from \eqref{En-energy} without the terms involving $\mu_n^n$ and $\mu_\tau^n$. Again for simplicity we let $\f_f = \f_p = q_f = q_p = 0$.

\begin{theorem}
The method \eqref{stokes-coupled-1}--\eqref{mu-t-coupled} has a unique solution satisfying the energy equality
\begin{equation}\label{energy-coupled}
\widehat{\mathcal{E}}^N + \Delta t \sum_{n=1}^N \mathcal{D}^n + \sum_{n=1}^N \mathcal{S}^n \le \widehat{\mathcal{E}}^0.
\end{equation}
\end{theorem}

\begin{proof}
  We take $\bv_f = \u_f^{n+1}$, $w_f = p_f^{n+1}$, $\bv_p = \u_p^{n+1}$, $w_p = p_p^{n+1}$, $\bxi_p = \dt\bbeta_p^{n+1}$, $\chi_n = \mu_n^{n+1}$, and $\chi_\tau = \mu_\tau^{n+1}$ in \eqref{stokes-coupled-1}--\eqref{mu-t-coupled} and sum the equations. In a way similar to \eqref{stokes-energy} and \eqref{biot-energy-1}, we obtain
\begin{align*}
&  \big(\rho_f \dt \u_{f}^{n+1},\u_{f}^{n+1}\big)_{\O_f} + a_{f}(\u_{f}^{n+1},\u_{f}^{n+1})
  + \big(\rho_p\dtt\bbeta_{p}^{n+1},\dt\bbeta_{p}^{n+1}\big)_{\O_p}
  + a^e_{p}(\bbeta_{p}^{n+1},\dt\bbeta_{p}^{n+1}) \\
& \qquad  + a^d_{p}(\u_{p}^{n+1},\u_{p}^{n+1}) + s_0(\dt p_{p}^{n+1},p_{p}^{n+1})_{\O_p} \le 0,
\end{align*}
We remark that all terms involving $\mu_n^{n+1}$ and $\mu_\tau^{n+1}$ cancel out. Then \eqref{energy-coupled} follows by using \eqref{identity-A2}, multiplying by $\Delta t$, and summing over $n$.

The energy balance \eqref{energy-coupled} implies uniqueness for $\u_{f}^n$, $\bbeta_{p}^{n}$, $\u_p^n$, and $p_p^n$. Uniqueness for $p_f^n$ and $\bmu^n$ follows from the argument used to establish that $\ol p_f^k \to 0$ and $\ol\bmu^k \to 0$ in the proof of Theorem~\ref{iter-semidiscrete-converge}.
\end{proof}

\section{Numerical results}\label{sec:num_res}
This section presents results from two numerical tests in two dimensions. We start with checking the convergence rates in time for the Robin-Robin algorithm \eqref{stokes-weak-1}--\eqref{mu-defn},
its iterative version \eqref{stokes-weak-1_iter}--\eqref{eq:stopping}, and
the monolithic scheme \eqref{stokes-coupled-1}--\eqref{mu-t-coupled}.
The convergence test is also used to assess the robustness of the Robin-Robin algorithm in both the 
non-iterative and iterative versions to changes in the value of the Robin parameter. Next, we consider a simplified blood flow problem to illustrate the behavior of the methods for a computationally challenging choice of physical parameters.

All results have been obtained with FreeFem++ \cite{Hecht}, using triangular grids. For spatial discretization we use the following finite element spaces: the Taylor-Hood continuous $\mathcal{P}_2-\mathcal{P}_1$ elements for the velocity--pressure pair in the Stokes problem, the Raviart-Thomas $\mathcal{RT}_1-\mathcal{P}_1^{dc}$ elements for the Darcy velocity and pressure, and continuous $\mathcal{P}_2$ elements for the structure displacement and the trace function $\mu$. 

\subsection{Example 1: convergence test}

In order to check the convergence rates in time, we 
consider an analytical solution in domains 
$\Omega_f = (0, 1)\times
(0, 1)$ and $\Omega_p=(0,1)\times (-1,0)$ with interface
$\Gamma_{fp}= (0,1)\times \{0\}$ over 
time interval $(0,1]$. 
We take $\Gamma_f^D = (0,1)\times \{1\}$, $\Gamma_f^N = \{0\}\times(0,1) \cup \{1\} \times (0,1)$, $\Gamma_p^D = \tilde\Gamma_p^D = (0,1)\times \{-1\}$, and $\Gamma_p^N = \tilde\Gamma_p^N = \{0\}\times(-1,0) \cup \{1\} \times (-1,0) $.
See the computational domain in Figure~\ref{fig:example1} (left) and the analytical solution in Figure~\ref{fig:example1} (right).
We note that the analytical solution satisfies the appropriate interface conditions on $\Gamma_{fp}$.

\begin{figure}[ht]
\begin{minipage}{.49\textwidth}
\centering
\begin{overpic}[width=.40\textwidth, grid=false]{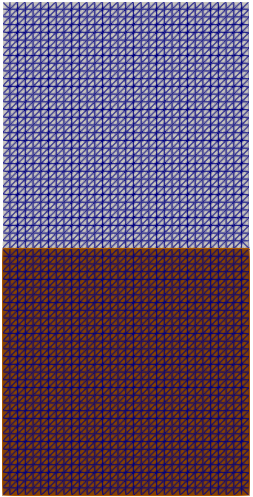}
\put(-9,72){$\Gamma_{f}^N$}
\put(52,72){$\Gamma_{f}^N$}
\put(25,103){$\Gamma_{f}^D$}
\put(52,48){$\Gamma_{fp}$}
\put(-9,23){$\Gamma_{p}^N$}
\put(52,23){$\Gamma_{p}^N$}
\put(25,-6){$\Gamma_{p}^D$}
\put(23,70){\textcolor{white}{$\Omega_f$}}
\put(23,20){\textcolor{white}{$\Omega_p$}}
\end{overpic}
\end{minipage}
\hfill
\begin{minipage}{.5\textwidth}
\bigskip
$ \ds \u_f = \pi\cos(\pi t)
\begin{pmatrix}
\ds -3 x + \cos( \, y) \\[1ex] \ds y+1
\end{pmatrix}$,\\
\bigskip
$ \ds p_f = \exp(t)\,\sin(\pi x)\cos\Big(\frac{\pi y}{2}\Big) + 2\pi \cos(\pi t)$,

$\ds p_p = \exp(t)\,\sin(\pi x)\cos\Big(\frac{\pi y}{2}\Big)$,

$\ds \u_p = -\frac{1}{\mu_{f}} K \grad p_p $,

$ \ds \bbeta_p = \sin(\pi t) \begin{pmatrix} \ds -3x+\cos(y) \\[1ex] \ds y+1 \end{pmatrix}$.
\end{minipage}
\vspace{0.2cm}
\caption{Example 1, left: computational domain and mesh; right: analytical solution.}
\label{fig:example1}
\end{figure}
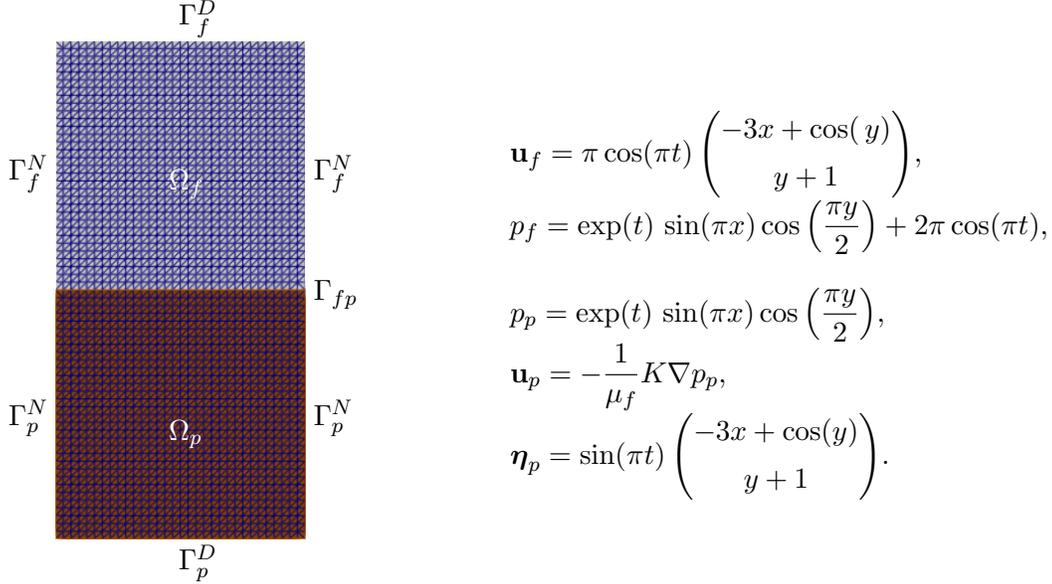

The model parameters are set as follows:
$\mu_f = 1$, $\rho_f = 1$, $\rho_p = 1$,
$\mu_p = 1$,  $\lambda_p = 1$, $s_0 = 1$, $K = \I_{2\times2}$, $\alpha = 1$,  $\gamma_{BJS} = 0$, $\gamma_f = \gamma_p = \gamma = 1$. The forcing terms $\f_f,  q_f,  \f_p$ and $q_p$ are found by plugging the analytical solution in \eqref{stokes2}--\eqref{eq:biot2}.
Similarly, appropriate data for the Dirichlet and Neumann boundary conditions and initial conditions are derived from the exact solution. 

The structured mesh used in the convergence study, obtained by setting the mesh size $h$ to $1/32$, is shown in Figure~\ref{fig:example1} (left).
The choice of the mesh size is so that the spatial discretization error does not affect the convergence rates in time. 

For the time convergence study, we consider the time interval $[0,1]$ and a sequence of progressively smaller time steps: $\Delta t = 0.2, 0.1, 0.05, 0.025, 0.0125$. 
Table~\ref{tab:Errors-Convergence rates-1} reports numerical errors for the Stokes, Biot, and auxiliary interface variables in the space-time norms bounded in the analysis, as well as
the corresponding convergence rates in time for the non-iterative Robin-Robin algorithm. We note that the error for the interface variable $\bmu$ is reported in the $L^2(\Gamma_{fp})$ norm, which is stronger, but easier to compute than the $H^{-1/2}(\Gamma_{fp})$ norm that appears in the analysis.
As the time step gets smaller, we observe that the rate of convergence
approaches one (i.e., the expected rate)
for all variables. Furthermore, 
Tables~\ref{tab:Errors-Convergence rates-2} and 
\ref{tab:Errors-Convergence rates-3} report the 
errors and rates for the iterative Robin-Robin algorithm and the monolithic scheme, respectively.
We also present the results from the iterative Robin-Robin algorithm with a fixed number of 10 iterations at each time step in Table~\ref{tab:maxiter10}.
First order convergence in time is observed for all variables for all methods.
To ease the comparison for selected variables, 
namely $\u_f$, $\u_p$, $\bbeta_p$, Figure~\ref{fig:conv_plots} (first row)
shows the convergence plots. 
We see that, as one would expect, the errors in the iterative Robin-Robin algorithm are indistinguishable from the errors in the monolithic scheme (identical for the number of digits reported in the tables), while the errors in the non-iterative Robin-Robin method are slightly larger. We recall that, in the case of the iterative algorithm, the increased accuracy
comes with an increased computational cost, since every time step requires the solution of multiple Stokes and Biot problems; see the last column in Table~\ref{tab:Errors-Convergence rates-2}
for the average number of iterations required to satisfy the stopping criterion \eqref{eq:stopping}, with the maximum number of iterations set to 100. However, we also observe from Table~\ref{tab:maxiter10} and Figure~\ref{fig:conv_plots} (first row) that even not running the iterative scheme to convergence, but only taking a small number of iterations (10), also gives results very close to the monolithic scheme. 

\begin{table}[htb!]
\centering{}%
\begin{tabular}{|c|c|c|c|c|c|c|c|c|}
\hline
{$\Delta t$} & \multicolumn{2}{c|}{{$\|e_{\u_{f}}\|_{L^{\infty}(H^1(\Omega_f))}$}}
&
\multicolumn{2}{c|}
{{$\|e_{p_{f}}\|_{L^{2}(L^2(\Omega_f))}$}}
& \multicolumn{2}{c|}{{$\|e_{\u_{p}}\|_{L^{2}(H(div;\Omega_p))}$}}
& \multicolumn{2}{c|}{{$\|e_{p_{p}}\|_{L^{\infty}(L^2(\Omega_p))}$}}
\tabularnewline
		\hline
		{\small{}0.2} & {\small{}1.663e+00   } & {\small{}Rate} & {\small{} 1.706e+00  } & {\small{}Rate} & {\small{} 1.800e+00   } & {\small{}Rate} & {\small{}3.112e-01  } & {\small{}Rate}\tabularnewline
		
		{\small{}0.1} & {\small{} 9.071e-01  } & {\small{}0.87 } & {\small{}  8.999e-01    } & {\small{}0.92 } & {\small{} 1.046e+00  } & {\small{}0.78 } & {\small{}  1.827e-01  } & {\small{} 0.72}\tabularnewline
		
		{\small{}0.05} & {\small{}4.768e-01  } & {\small{}0.92 } & {\small{}4.640e-01  } & {\small{} 0.95} & {\small{}5.825e-01   } & {\small{}0.84 } & {\small{} 1.023e-01    } & {\small{}0.83 } \tabularnewline
		
		{\small{}0.025} & {\small{}2.449e-01     } & {\small{}0.96 } & {\small{} 2.360e-01  } & {\small{}0.97 } & {\small{} 3.113e-01   } & {\small{} 0.90} & {\small{} 5.497e-02    } & {\small{}0.89 } \tabularnewline
		
		{\small{}0.0125} & {\small{} 1.247e-01    } & {\small{}0.97  } & {\small{} 1.191e-01     } & {\small{}0.98 } & {\small{}1.617e-01   } & {\small{}0.94  } & {\small{}2.855e-02     } & {\small{} 0.94 }\tabularnewline
		\hline
\end{tabular}
%
\centering{}%
\begin{tabular}{|c|c|c|c|c|c|c|c|}
\hline
{\small{}$\Delta t$} & \multicolumn{2}{c|}{{\small{}$\|e_{\bbeta_p}\|_{L^{\infty}(H^1(\Omega_p))}$}}
&\multicolumn{2}{c|}
{{\small{}$\|e_{\partial_{t}\bbeta_p}\|_{L^{\infty}(L^2(\Omega_p))}$}}
& \multicolumn{2}{c|}{{\small{}$\|e_{\mu}\|_{L^{\infty}(L^2(\Gamma_{fp}))}$}}
\tabularnewline
		\hline
		{\small{}0.2} & {\small{}1.966e+00  } & {\small{}Rate} & {\small{}1.578e+00   } & {\small{}Rate} & {\small{} 2.369e+00    } & {\small{}Rate} \tabularnewline
		
		{\small{}0.1} & {\small{} 1.183e+00  } & {\small{}0.73 } & {\small{} 8.996e-01    } & {\small{}0.81 } & {\small{} 1.311e+00    } & {\small{}0.85 } \tabularnewline
		
		{\small{}0.05} & {\small{}   6.675e-01  } & {\small{}0.82 } & {\small{}4.808e-01    } & {\small{} 0.90 } & {\small{} 6.857e-01       } & {\small{}0.93 }  \tabularnewline
		
		{\small{}0.025} & {\small{}3.589e-01    } & {\small{}0.89 } & {\small{}  2.491e-01  } & {\small{}0.94 } & {\small{} 3.479e-01  } & {\small{}0.97 }  \tabularnewline
		
		{\small{}0.0125} & {\small{}  1.868e-01     } & {\small{}0.94 } & {\small{} 1.270e-01    } & {\small{}0.97  } & {\small{}1.745e-01   } & {\small{} 0.99 } \tabularnewline
		\hline
	\end{tabular}
\caption{Example 1, numerical errors and convergence rates in time 
  for the non-iterative Robin-Robin algorithm.} \label{tab:Errors-Convergence rates-1}
\end{table}

\begin{table}[htb!]
\centering{}%
\begin{tabular}{|c|c|c|c|c|c|c|c|c|}
\hline
{$\Delta t$} & \multicolumn{2}{c|}{{$\|e_{\u_{f}}\|_{L^{\infty}(H^1(\Omega_f))}$}}
&
\multicolumn{2}{c|}
{{$\|e_{p_{f}}\|_{L^{2}(L^2(\Omega_f))}$}}
& \multicolumn{2}{c|}{{$\|e_{\u_{p}}\|_{L^{2}(H(div;\Omega_p))}$}}
& \multicolumn{2}{c|}{{$\|e_{p_{p}}\|_{L^{\infty}(L^2(\Omega_p))}$}}
\tabularnewline
		\hline
		{\small{}0.2} & {\small{} 1.233e+00    } & {\small{}Rate} & {\small{} 1.537e+00   } & {\small{}Rate} & {\small1.730e+00    } & {\small{}Rate} & {\small{}  2.855e-01  } & {\small{}Rate}\tabularnewline
		
		{\small{}0.1} & {\small{}  6.481e-01   } & {\small{}0.92 } & {\small{}  7.809e-01   } & {\small{}0.97 } & {\small{}1.005e+00   } & {\small{} 0.78} & {\small{}1.700e-01    } & {\small{}0.74 }\tabularnewline
		
		{\small{}0.05} & {\small{} 3.331e-01    } & {\small{}0.96 } & {\small{} 3.936e-01    } & {\small{}0.98 } & {\small{}  5.602e-01  } & {\small{} 0.84} & {\small{}   9.646e-02  } & {\small{}0.81 } \tabularnewline
		
		{\small{}0.025} & {\small{} 1.686e-01   } & {\small{}0.98 } & {\small{} 1.977e-01    } & {\small{} 0.99 } & {\small{} 2.998e-01   } & {\small{}0.90 } & {\small{} 5.169e-02   } & {\small{}0.90 } \tabularnewline
		
		{\small{}0.0125} & {\small{} 8.473e-02  } & {\small{} 0.99  } & {\small{}  9.911e-02    } & {\small{}0.99 } & {\small{}1.559e-01    } & {\small{} 0.94 } & {\small{} 2.686e-02 } & {\small{} 0.94} \tabularnewline
		\hline
\end{tabular}
%
\centering{}%
\begin{tabular}{|c|c|c|c|c|c|c|c|c|}
\hline
{\small{}$\Delta t$} & \multicolumn{2}{c|}{{\small{}$\|e_{\bbeta_p}\|_{L^{\infty}(H^1(\Omega_p))}$}}
&\multicolumn{2}{c|}
{{\small{}$\|e_{\partial_{t}\bbeta_p}\|_{L^{\infty}(L^2(\Omega_p))}$}}
& \multicolumn{2}{c|}{{\small{}$\|e_{\mu}\|_{L^{\infty}(L^2(\Gamma_{fp}))}$}}&{\small{}\# iter}
\tabularnewline
		\hline
		{\small{}0.2} & {\small{}1.520e+00  } & {\small{}Rate} & {\small{} 1.553e+00   } & {\small{}Rate} & {\small{}  1.853e+00       } & {\small{}Rate} & {\small{}96.60 }\tabularnewline
		
		{\small{}0.1} & {\small{}8.827e-01    } & {\small{}0.78 } & {\small{}   8.933e-01     } & {\small{} 0.79} & {\small{}9.848e-01     } & {\small{}0.91 }  &{\small{}89.20 }\tabularnewline
		
		{\small{}0.05} & {\small{} 4.938e-01     } & {\small{}0.83 } & {\small{} 4.803e-01     } & {\small{}0.89 } & {\small{}5.123e-01       } & {\small{}0.94} &{\small{}76.50 }\tabularnewline
		
		{\small{}0.025} & {\small{} 2.659e-01   } & {\small{}0.89 } & {\small{} 2.497e-01   } & {\small{}0.94 } & {\small{}2.625e-01       } & {\small{} 0.96}  &{\small{} 65.45}\tabularnewline
		
		{\small{}0.0125} & {\small{}  1.388e-01      } & {\small{}0.93  } & {\small{} 1.276e-01      } & {\small{}0.96 } & {\small{}   1.337e-01   } & {\small{} 0.97 }  &{\small{}55.10  }\tabularnewline
		\hline
	\end{tabular}
\caption{Example 1, numerical errors and convergence rates in time 
  for the iterative Robin-Robin algorithm. The last column reports the average number of iterations required to satisfy the stopping criterion \eqref{eq:stopping}.}
\label{tab:Errors-Convergence rates-2}
\end{table}

\begin{table}[htb!]
\centering{}%
\begin{tabular}{|c|c|c|c|c|c|c|c|c|}
\hline
{$\Delta t$} & \multicolumn{2}{c|}{{$\|e_{\u_{f}}\|_{L^{\infty}(H^1(\Omega_f))}$}}
&
\multicolumn{2}{c|}
{{$\|e_{p_{f}}\|_{L^{2}(L^2(\Omega_f))}$}}
& \multicolumn{2}{c|}{{$\|e_{\u_{p}}\|_{L^{2}(H(div;\Omega_p))}$}}
& \multicolumn{2}{c|}{{$\|e_{p_{p}}\|_{L^{\infty}(L^2(\Omega_p))}$}}
\tabularnewline
		\hline
		{\small{}0.2} & {\small{} 1.233e+00    } & {\small{}Rate} & {\small{} 1.537e+00   } & {\small{}Rate} & {\small1.730e+00    } & {\small{}Rate} & {\small{}  2.855e-01  } & {\small{}Rate}\tabularnewline
		
		{\small{}0.1} & {\small{}  6.481e-01   } & {\small{}0.92 } & {\small{}  7.809e-01   } & {\small{}0.97 } & {\small{}1.005e+00   } & {\small{} 0.78} & {\small{}1.700e-01    } & {\small{}0.74 }\tabularnewline
		
		{\small{}0.05} & {\small{} 3.331e-01    } & {\small{}0.96 } & {\small{} 3.936e-01    } & {\small{}0.98 } & {\small{}  5.602e-01  } & {\small{} 0.84} & {\small{}   9.646e-02  } & {\small{}0.81 } \tabularnewline
		
		{\small{}0.025} & {\small{} 1.686e-01   } & {\small{}0.98 } & {\small{} 1.977e-01    } & {\small{} 0.99 } & {\small{} 2.998e-01   } & {\small{}0.90 } & {\small{} 5.169e-02   } & {\small{}0.90 } \tabularnewline
		
		{\small{}0.0125} & {\small{} 8.473e-02  } & {\small{} 0.99  } & {\small{}  9.911e-02    } & {\small{}0.99 } & {\small{}1.559e-01    } & {\small{} 0.94 } & {\small{} 2.686e-02 } & {\small{} 0.94} \tabularnewline
		\hline
\end{tabular}

%
\centering{}%
\begin{tabular}{|c|c|c|c|c|c|c|c|}
\hline
{\small{}$\Delta t$} & \multicolumn{2}{c|}{{\small{}$\|e_{\bbeta_p}\|_{L^{\infty}(H^1(\Omega_p))}$}}
&\multicolumn{2}{c|}
{{\small{}$\|e_{\partial_{t}\bbeta_p}\|_{L^{\infty}(L^2(\Omega_p))}$}}
& \multicolumn{2}{c|}{{\small{}$\|e_{\mu}\|_{L^{\infty}(L^2(\Gamma_{fp}))}$}}
\tabularnewline
		\hline
		{\small{}0.2} & {\small{}1.520e+00  } & {\small{}Rate} & {\small{} 1.553e+00   } & {\small{}Rate} & {\small{} 1.865e+00 } & {\small{}Rate} \tabularnewline
		
		{\small{}0.1} & {\small{}8.827e-01    } & {\small{}0.78 } & {\small{}   8.933e-01     } & {\small{} 0.79} & {\small{} 9.887e-01} & {\small{}0.91 }  \tabularnewline
		
		{\small{}0.05} & {\small{} 4.938e-01     } & {\small{}0.83 } & {\small{} 4.803e-01     } & {\small{}0.89 } & {\small{}5.130e-01} & {\small{}0.94} \tabularnewline
		
		{\small{}0.025} & {\small{} 2.659e-01   } & {\small{}0.89 } & {\small{} 2.497e-01   } & {\small{}0.94 } & {\small{}2.626e-01} & {\small{} 0.96}  \tabularnewline
		
		{\small{}0.0125} & {\small{}  1.388e-01      } & {\small{}0.93  } & {\small{} 1.276e-01      } & {\small{}0.96 } & {\small{}1.337e-01} & {\small{} 0.97 } \tabularnewline
		\hline
	\end{tabular}
\caption{Example 1, numerical errors and convergence rates in time 
  for the monolithic scheme.} \label{tab:Errors-Convergence rates-3}
\end{table}

\begin{table}[htb!]
\centering{}%
\begin{tabular}{|c|c|c|c|c|c|c|c|c|}
\hline
{$\Delta t$} & \multicolumn{2}{c|}{{$\|e_{\u_{f}}\|_{L^{\infty}(H^1(\Omega_f))}$}}
&
\multicolumn{2}{c|}
{{$\|e_{p_{f}}\|_{L^{2}(L^2(\Omega_f))}$}}
& \multicolumn{2}{c|}{{$\|e_{\u_{p}}\|_{L^{2}(H(div;\Omega_p))}$}}
& \multicolumn{2}{c|}{{$\|e_{p_{p}}\|_{L^{\infty}(L^2(\Omega_p))}$}}
\tabularnewline
		\hline
		{\small{}0.2} & {\small{} 1.240e+00    } & {\small{}Rate} & {\small{} 1.537e+00   } & {\small{}Rate} & {\small1.731e+00    } & {\small{}Rate} & {\small{}  2.862e-01  } & {\small{}Rate}\tabularnewline
		
		{\small{}0.1} & {\small{}  6.491e-01   } & {\small{}0.93 } & {\small{}  7.807e-01   } & {\small{}0.97 } & {\small{}1.006e+00   } & {\small{} 0.78} & {\small{}1.704e-01    } & {\small{}0.74 }\tabularnewline
		
		{\small{}0.05} & {\small{} 3.325e-01    } & {\small{}0.96 } & {\small{} 3.934e-01    } & {\small{}0.98 } & {\small{}  5.603e-01  } & {\small{} 0.84} & {\small{}   9.659e-02  } & {\small{}0.81 } \tabularnewline
		
		{\small{}0.025} & {\small{} 1.676e-01   } & {\small{}0.98 } & {\small{} 1.974e-01    } & {\small{} 0.99 } & {\small{} 2.995e-01   } & {\small{}0.90 } & {\small{} 5.162e-02   } & {\small{}0.90 } \tabularnewline
		
		{\small{}0.0125} & {\small{} 8.365e-02  } & {\small{} 1.00  } & {\small{}  9.866e-02    } & {\small{}1.00 } & {\small{}1.554e-01    } & {\small{} 0.94 } & {\small{} 2.662e-02 } & {\small{} 0.94} \tabularnewline
		\hline
\end{tabular}
%
\centering{}%
\begin{tabular}{|c|c|c|c|c|c|c|c|c|}
\hline
{\small{}$\Delta t$} & \multicolumn{2}{c|}{{\small{}$\|e_{\bbeta_p}\|_{L^{\infty}(H^1(\Omega_p))}$}}
&\multicolumn{2}{c|}
{{\small{}$\|e_{\partial_{t}\bbeta_p}\|_{L^{\infty}(L^2(\Omega_p))}$}}
& \multicolumn{2}{c|}{{\small{}$\|e_{\mu}\|_{L^{\infty}(L^2(\Gamma_{fp}))}$}}&{\small{}\# iter}
\tabularnewline
		\hline
		{\small{}0.2} & {\small{}1.512e+00  } & {\small{}Rate} & {\small{} 1.553e+00   } & {\small{}Rate} & {\small{}  1.828e+00       } & {\small{}Rate} & {\small{}10.00 }\tabularnewline
		
		{\small{}0.1} & {\small{}8.877e-01    } & {\small{}0.78 } & {\small{}   8.933e-01     } & {\small{} 0.79} & {\small{}9.658e-01     } & {\small{}0.92 }  &{\small{}10.00 }\tabularnewline
		
		{\small{}0.05} & {\small{} 4.903e-01     } & {\small{}0.83 } & {\small{} 4.803e-01     } & {\small{}0.89 } & {\small{}5.020e-01       } & {\small{}0.94} &{\small{}10.00 }\tabularnewline
		
		{\small{}0.025} & {\small{} 2.637e-01   } & {\small{}0.89 } & {\small{} 2.497e-01   } & {\small{}0.94 } & {\small{}2.583e-01       } & {\small{} 0.95}  &{\small{} 10.00}\tabularnewline
		
		{\small{}0.0125} & {\small{}  1.373e-01      } & {\small{}0.94  } & {\small{} 1.276e-01      } & {\small{}0.96 } & {\small{}   1.324e-01   } & {\small{} 0.96 }  &{\small{}10.00  }\tabularnewline
		\hline
	\end{tabular}
\caption{Example 1, numerical errors and convergence rates in time 
  for the iterative Robin-Robin algorithm with 10 iterations per time step.}
 \label{tab:maxiter10}
\end{table}

\begin{figure}[ht!]
\begin{center}
\includegraphics[width=0.325\textwidth]{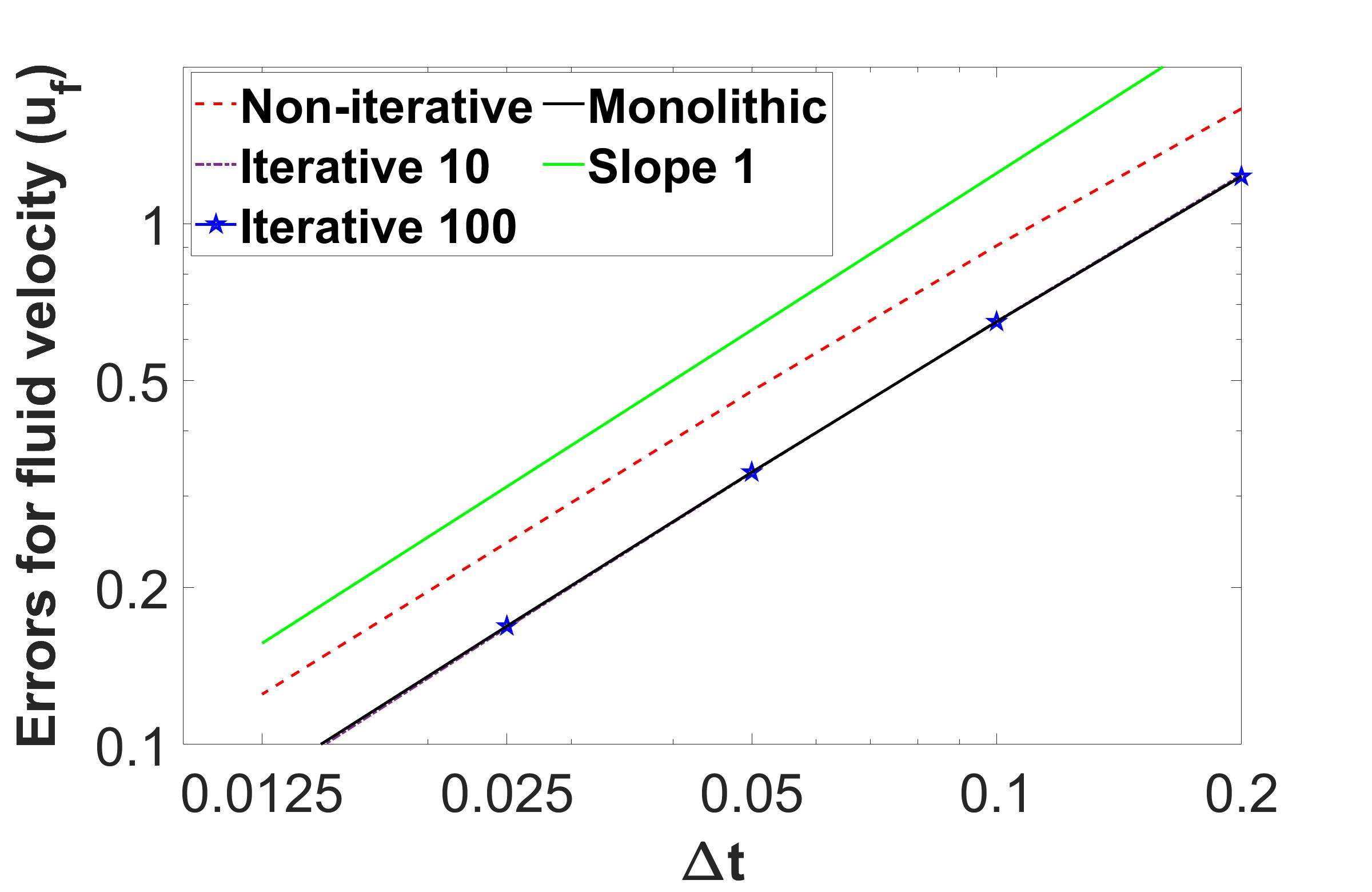}
\includegraphics[width=0.325\textwidth]{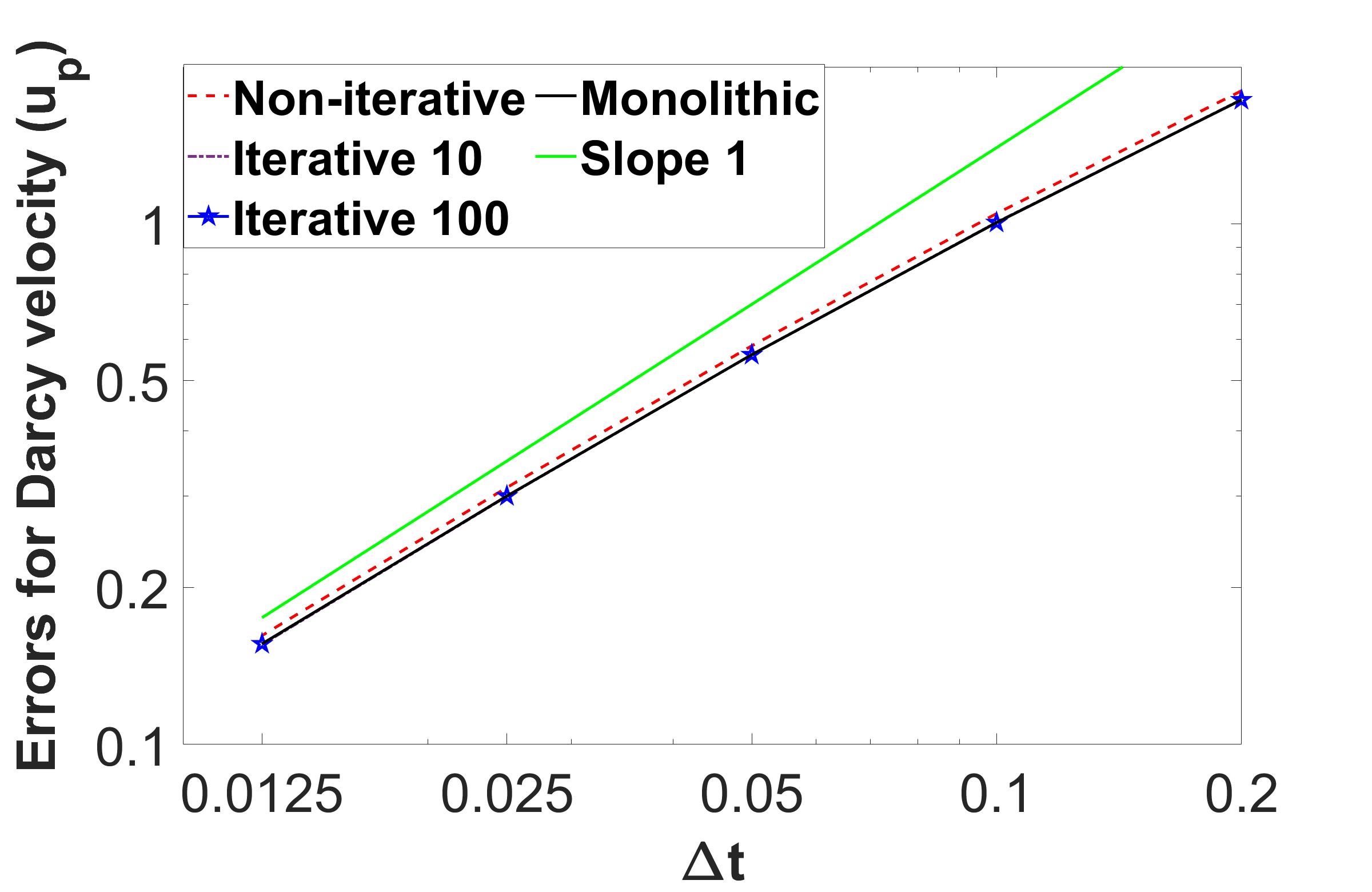}
\includegraphics[width=0.325\textwidth]{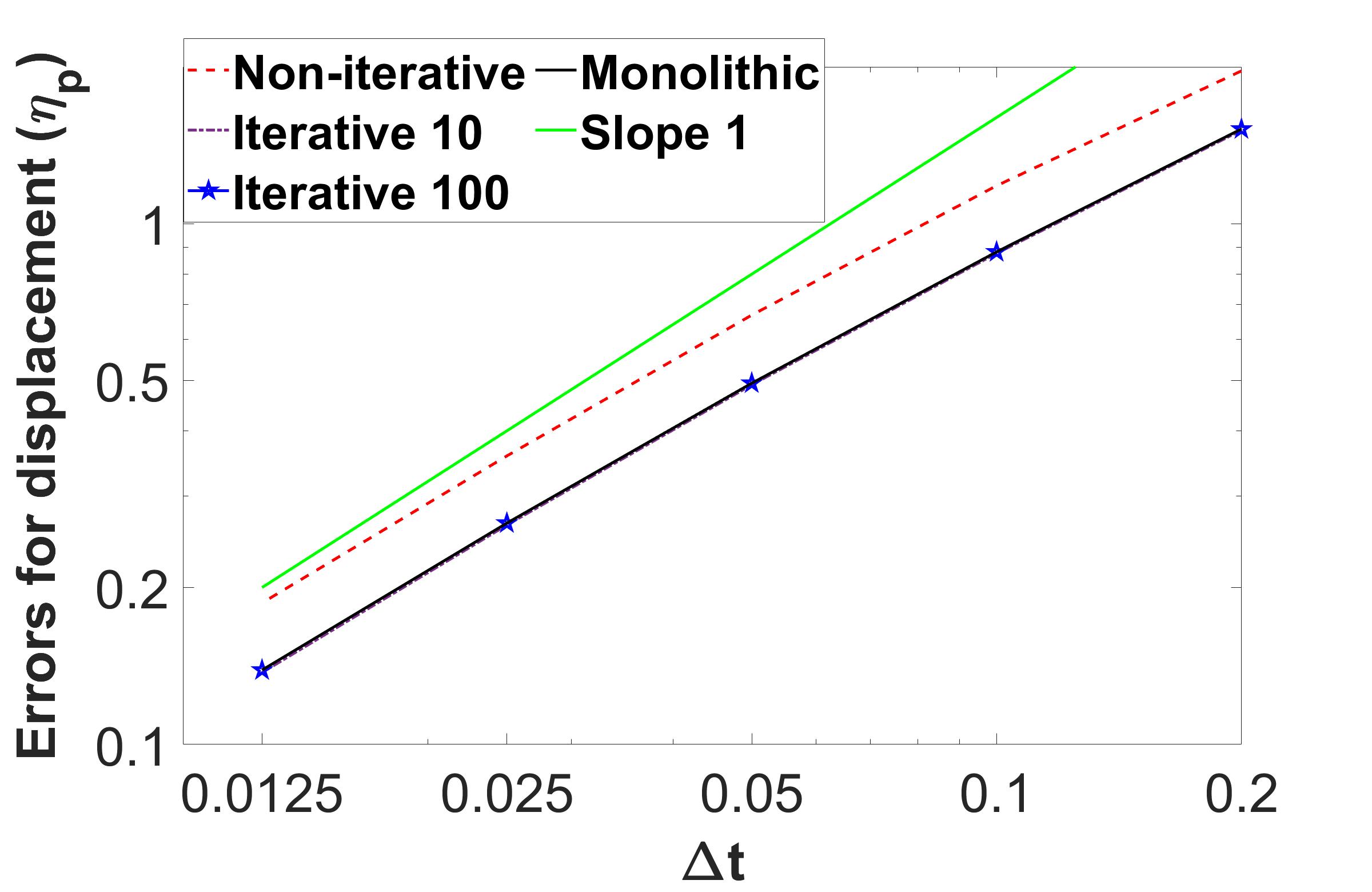}\\
\includegraphics[width=0.325\textwidth]{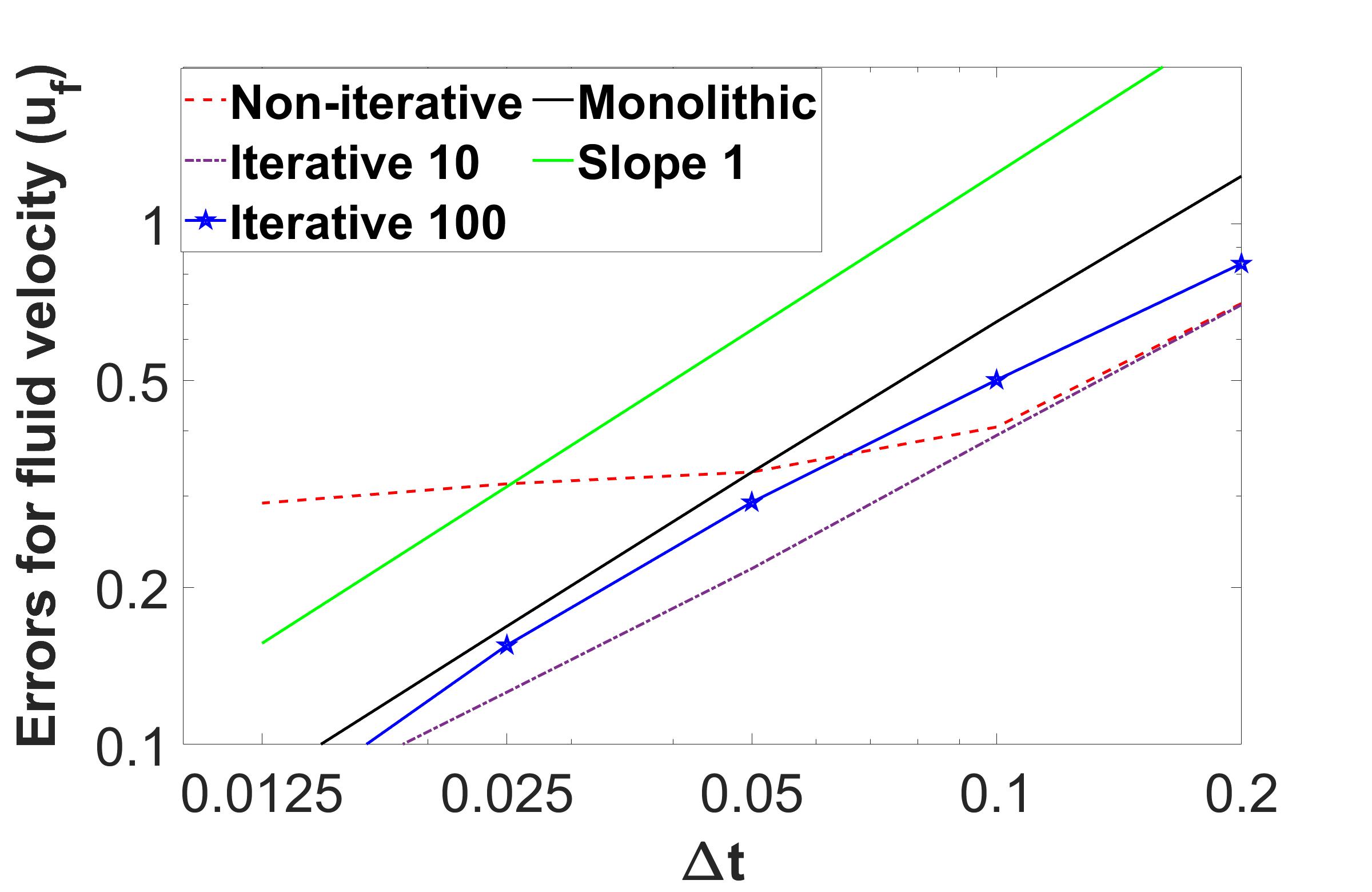}
\includegraphics[width=0.325\textwidth]{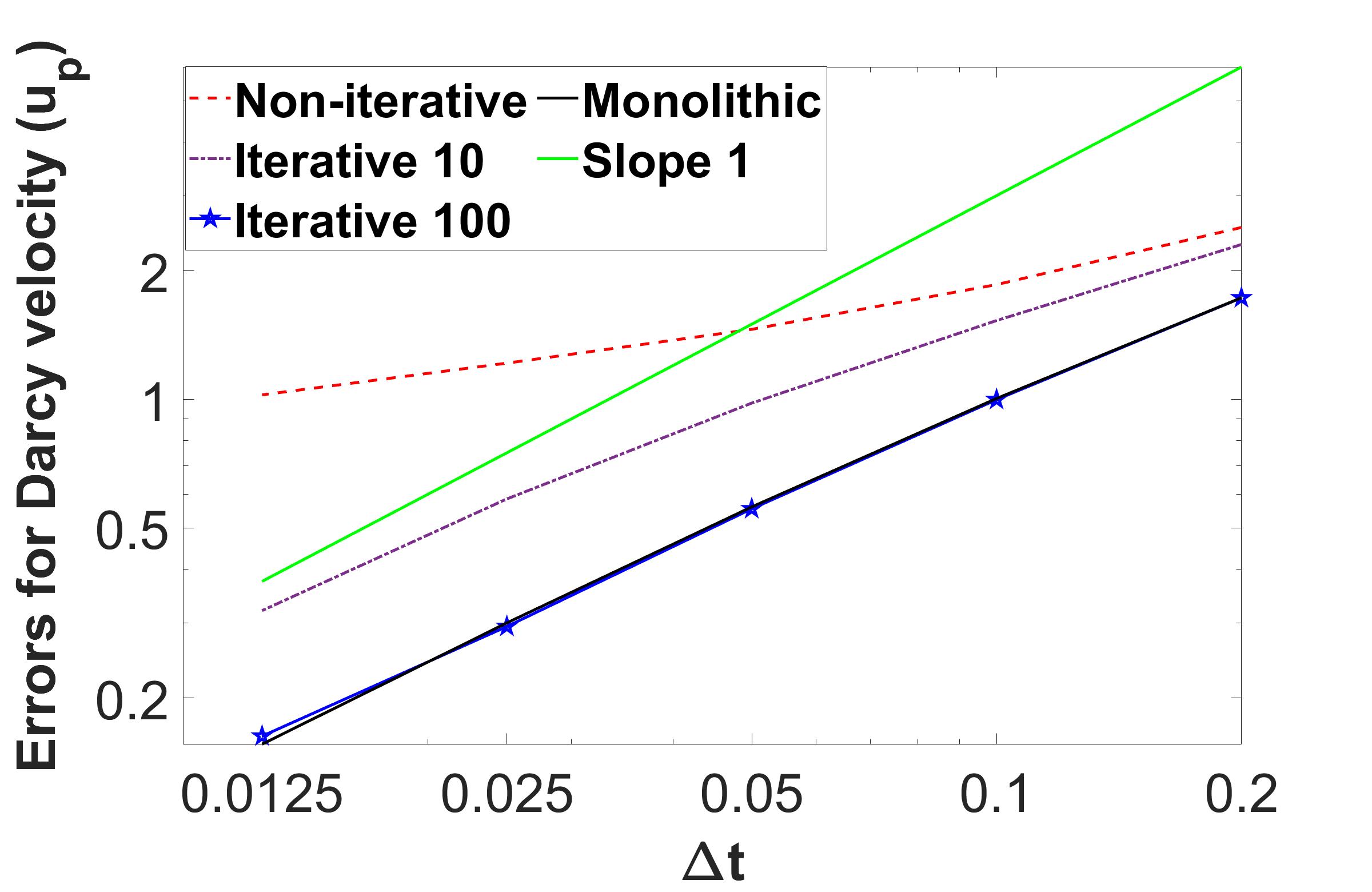}
\includegraphics[width=0.325\textwidth]{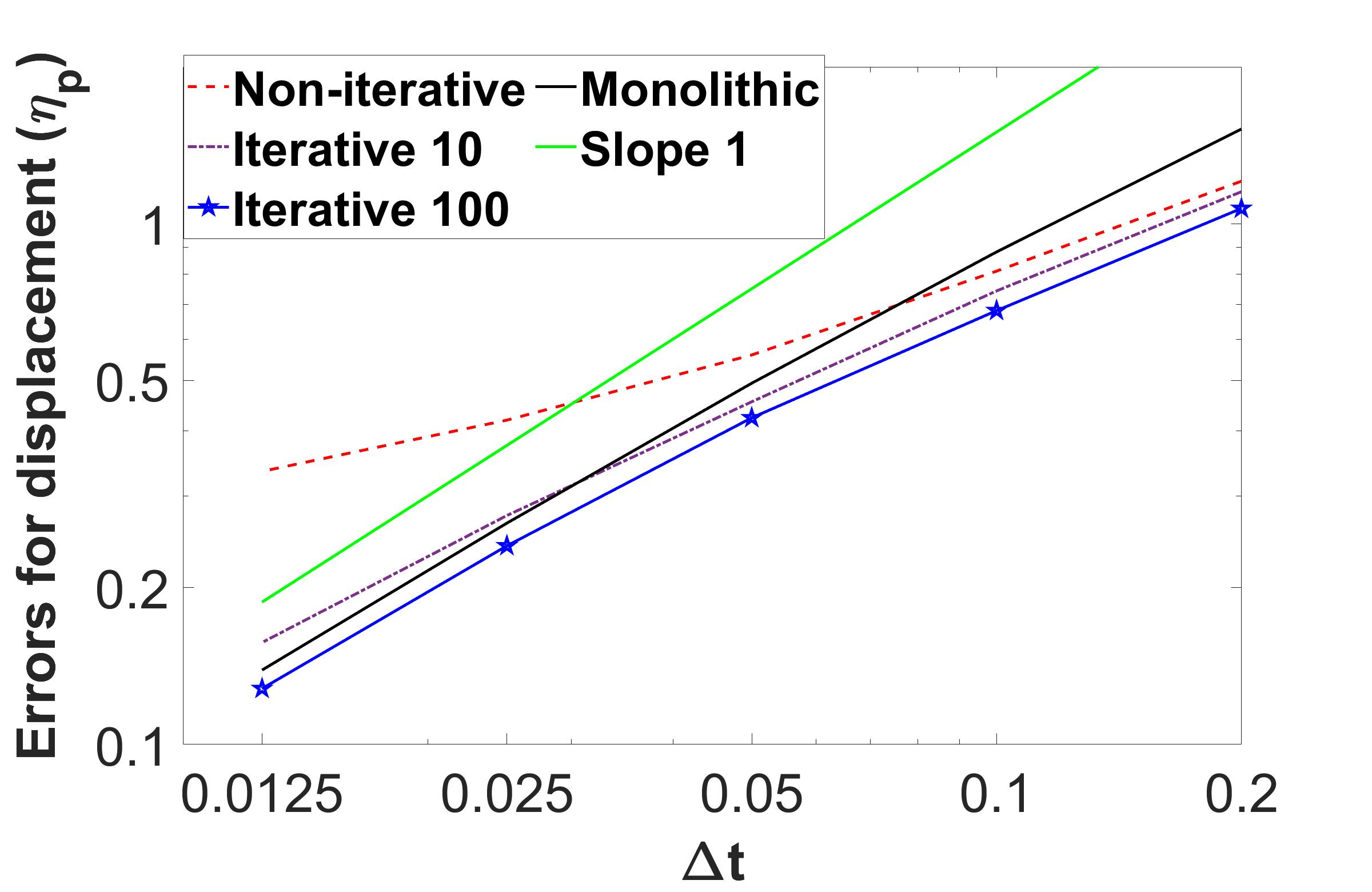}\\
\includegraphics[width=0.325\textwidth]{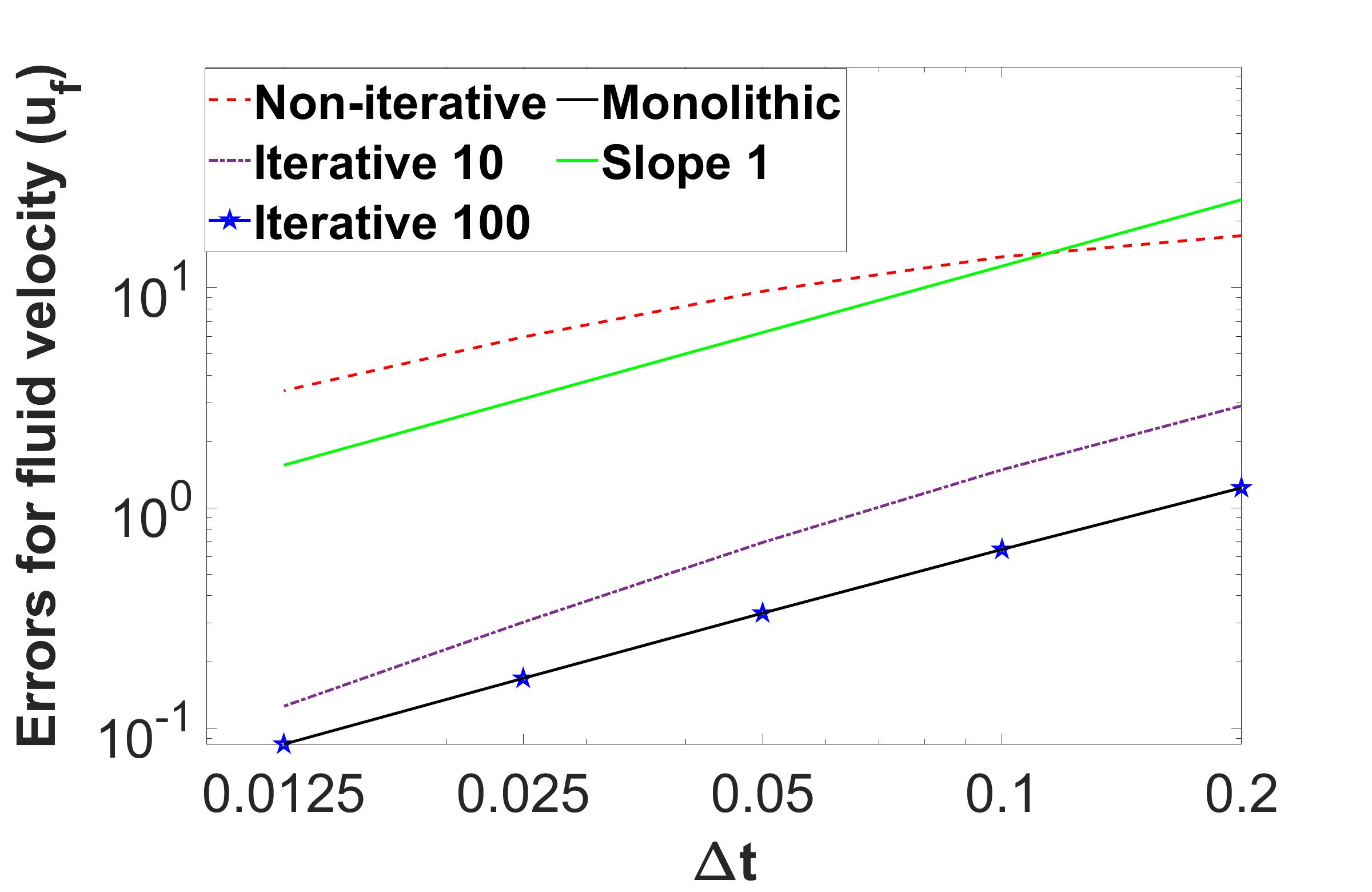}
\includegraphics[width=0.325\textwidth]{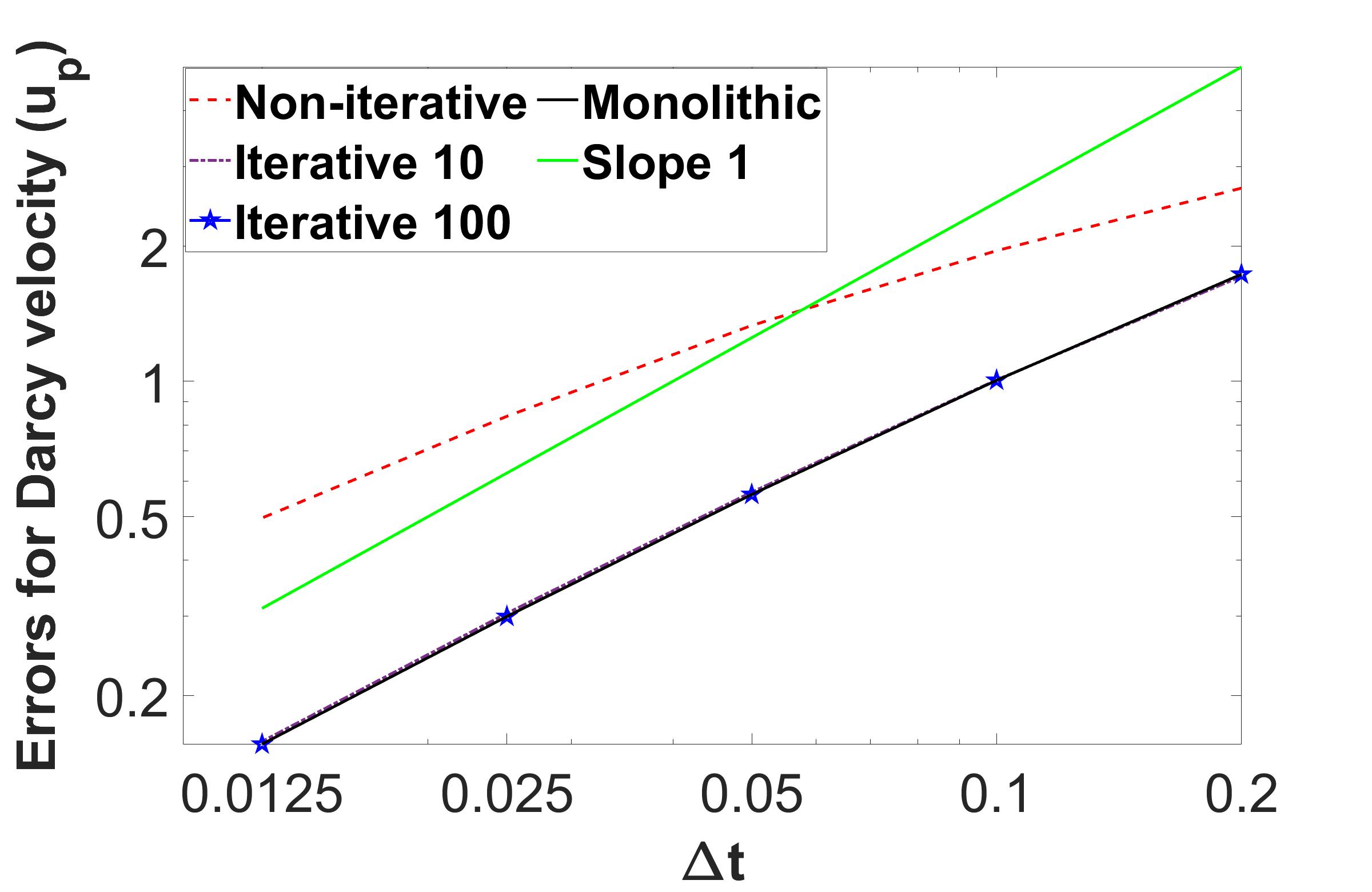}
\includegraphics[width=0.325\textwidth]{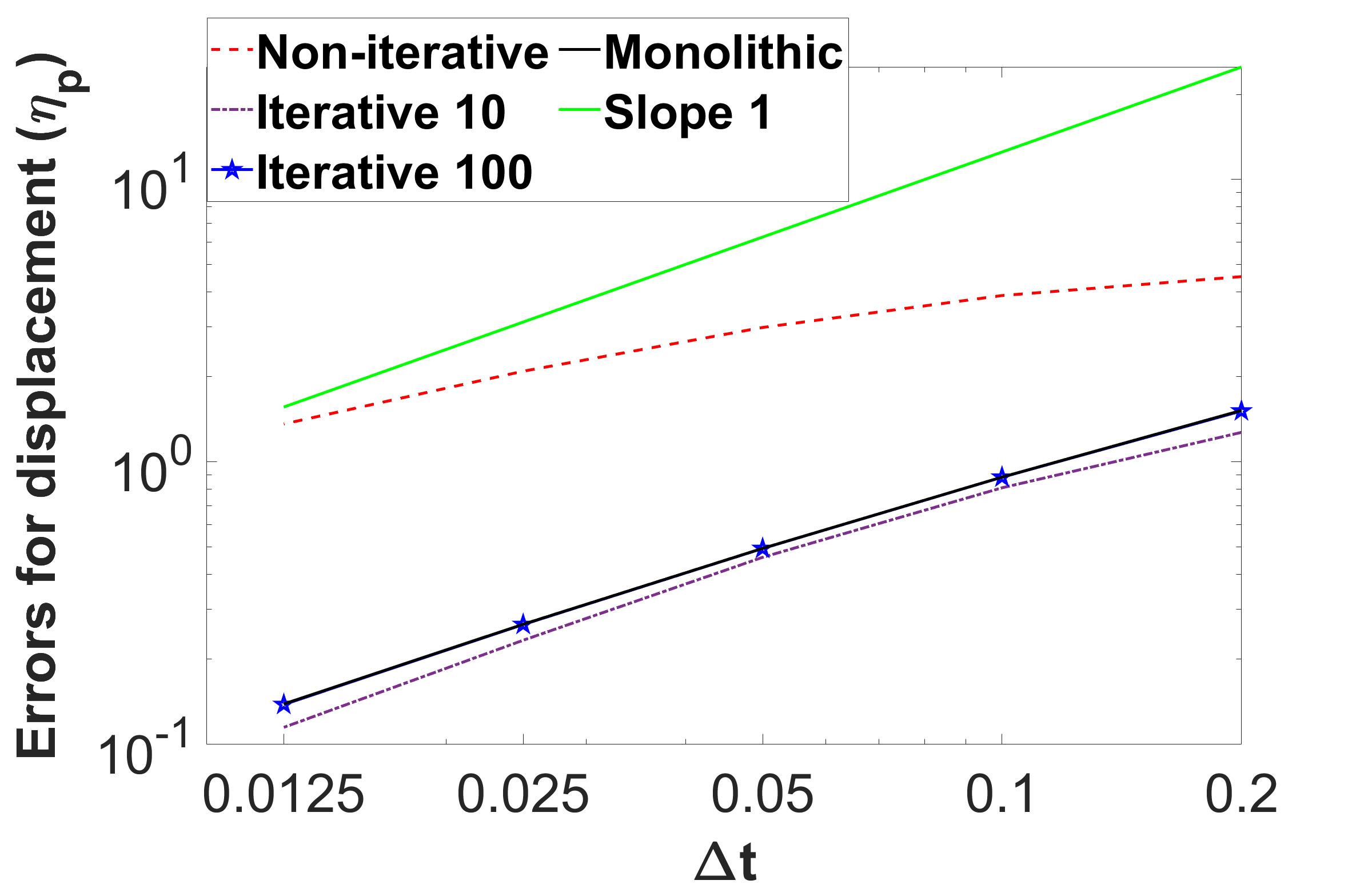}\\
\end{center}
\caption{Example 1, convergence plots for fluid velocity $\u_f$ (left), Darcy velocity $\u_p$ (center), and displacement $\bbeta_p$ (right)
for $\gamma=1$ (first row), $\gamma=0.001$ (second row), and $\gamma=100$ (third row)
  computed by the non-iterative and iterative 
  Robin-Robin methods and the monolithic method.
}\label{fig:conv_plots}
\end{figure}

\subsubsection{Robustness to the Robin parameter $\gamma$}

In this subsection, we study the robustness of the Robin-Robin schemes to the
value of $\gamma$, which is obviously a key parameter. All the parameters are set as previously, with the exception of $\gamma = \gamma_f = \gamma_p$, which will take different values. In practice, it is reasonable to set $\gamma$ so that the magnitudes of the velocity and stress terms in the Robin combination are comparable. While these can be inferred from the physical data and/or preliminary simulation results, some degree of robustness to $\gamma$ is desirable. 

First, we consider the non-iterative Robin-Robin scheme and the effect of $\gamma$ on the numerical errors. Figure~\ref{fig:conv_plots-gamma} shows the convergence plots for $\u_f$, $\u_p$, $\bbeta_p$ for $\gamma = 0.001, 0.01, 0.1, 1, 10, 100$. For all variables, we observe first order convergence for $\gamma = 0.01, 0.1, 1, 10$, with the rates approaching first order as $\Delta t \to 0$ for $\gamma = 100$. However, the convergence rate deteriorates for $\gamma = 0.001$. The different variables exhibit somewhat different sensitivity to $\gamma$, with $\u_f$ being slightly more sensitive than $\u_p$ and $\bbeta_p$. The general trend is that the errors are similar for $\gamma \in [0.01,10]$ and the errors increase if $\gamma$ is too small or too large. We note that for the range $\gamma \in [0.01,10]$ the magnitudes of the velocity and stress terms in the Robin combination are comparable. The increase in errors for extreme values of $\gamma$ can be explained from the theoretical estimate, which has terms proportional to both $\gamma^2$, cf. \eqref{H4Bound} and $1/\gamma^2$, cf. \eqref{mixed-2}. We conjecture that the reason for the reduced convergence rate with $\gamma = 0.001$ is that velocity continuity is not properly enforced by the Robin transmission conditions. We finally note that both the converged iterative scheme and the iterative scheme with 10 iterations give smaller errors and recover first order convergence in time for the extreme values $\gamma = 0.001, 100$, see the second and third rows in Figure~\ref{fig:conv_plots}.
Moreover, in Figure~\ref{fig:maxiter10} we present the convergence plots for $\u_f$, $\u_p$, $\bbeta_p$ with $\gamma = 0.001, 0.01, 0.1, 1, 10, 100$ for the iterative method with 10 iterations and note that the sensitivity to $\gamma$ is significantly reduced compared to the non-iterative method.

\begin{figure}[ht!]
\begin{center}
\includegraphics[width=0.325\textwidth]{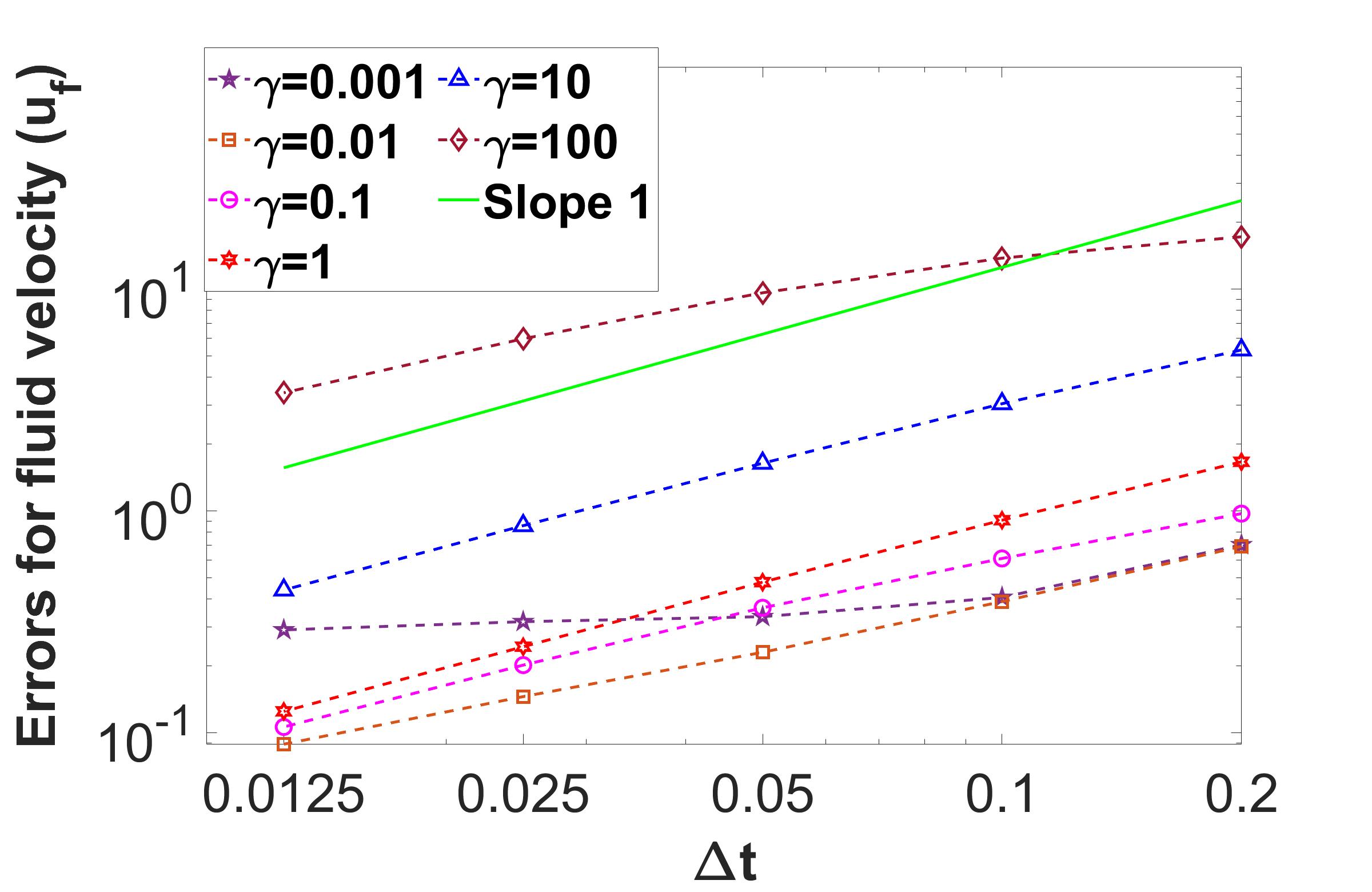}
\includegraphics[width=0.325\textwidth]{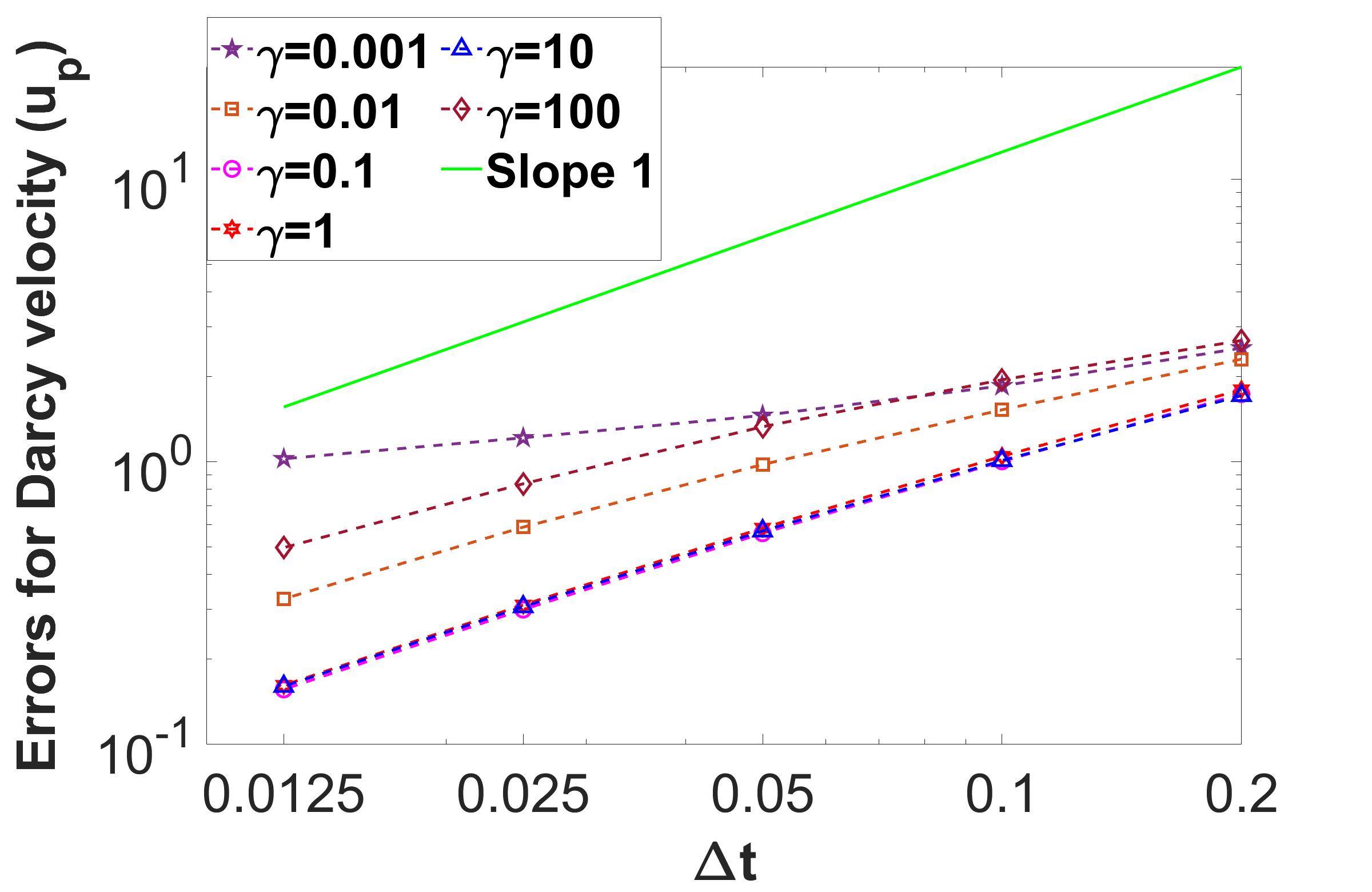}
\includegraphics[width=0.325\textwidth]{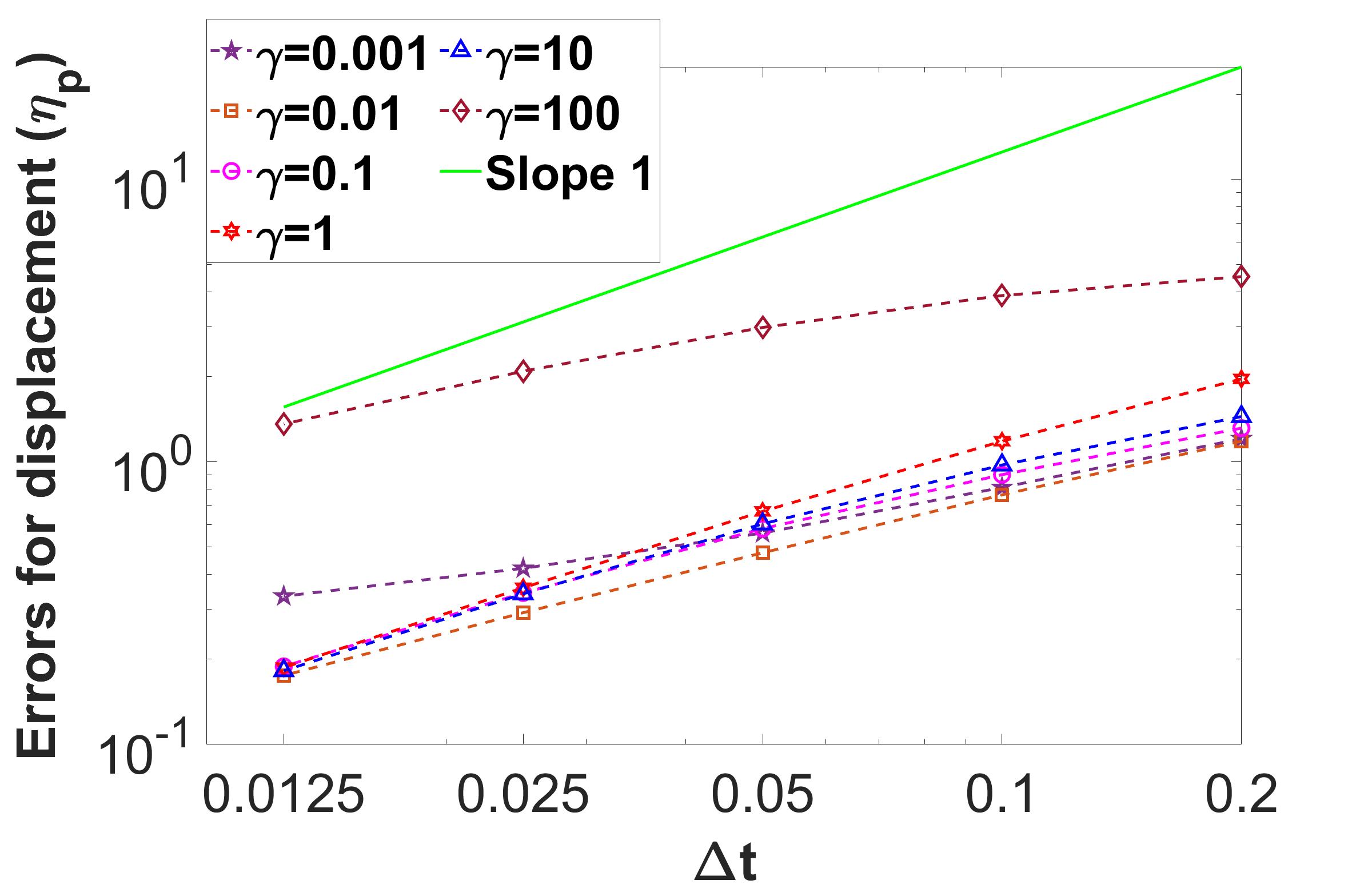}
\end{center}
\caption{Example 1, convergence plots for fluid velocity $\u_f$ (left), Darcy velocity $\u_p$ (center) and displacement $\bbeta_p$ (right) for the non-iterative Robin-Robin algorithm for different values of $\gamma$.}\label{fig:conv_plots-gamma}
\end{figure}

\begin{figure}[ht!]
\begin{center}
\includegraphics[width=0.325\textwidth]{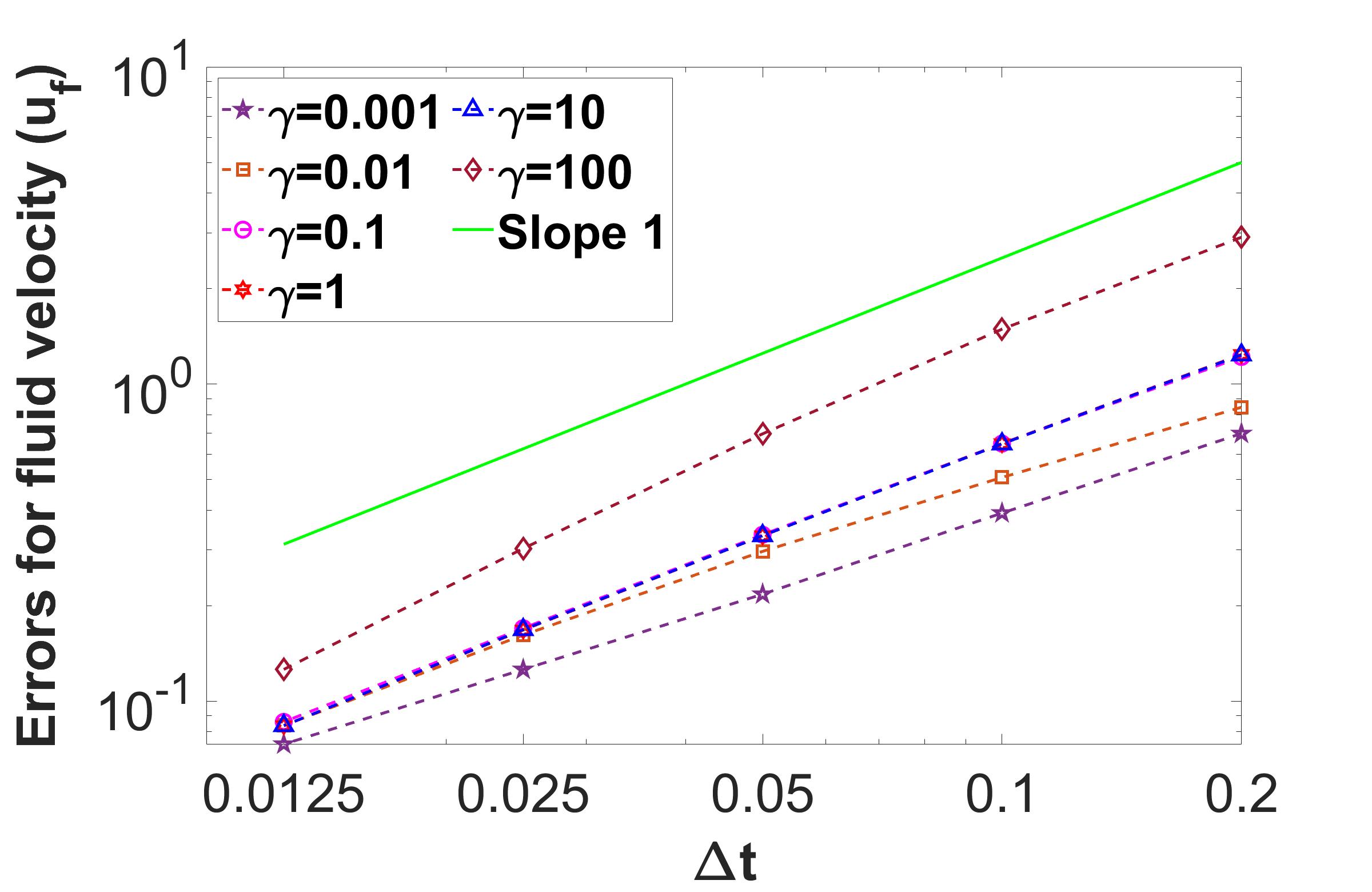}
\includegraphics[width=0.325\textwidth]{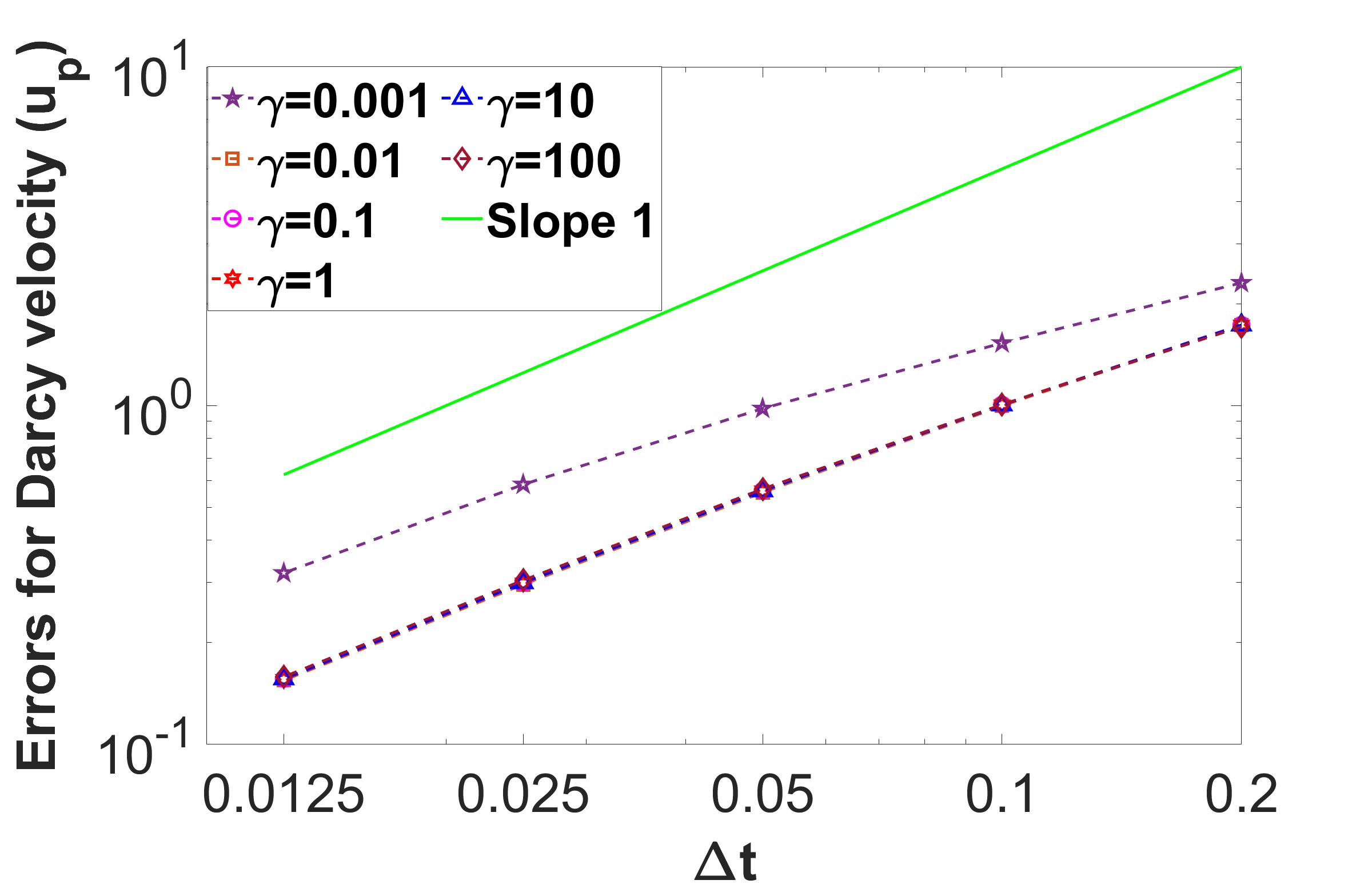}
\includegraphics[width=0.325\textwidth]{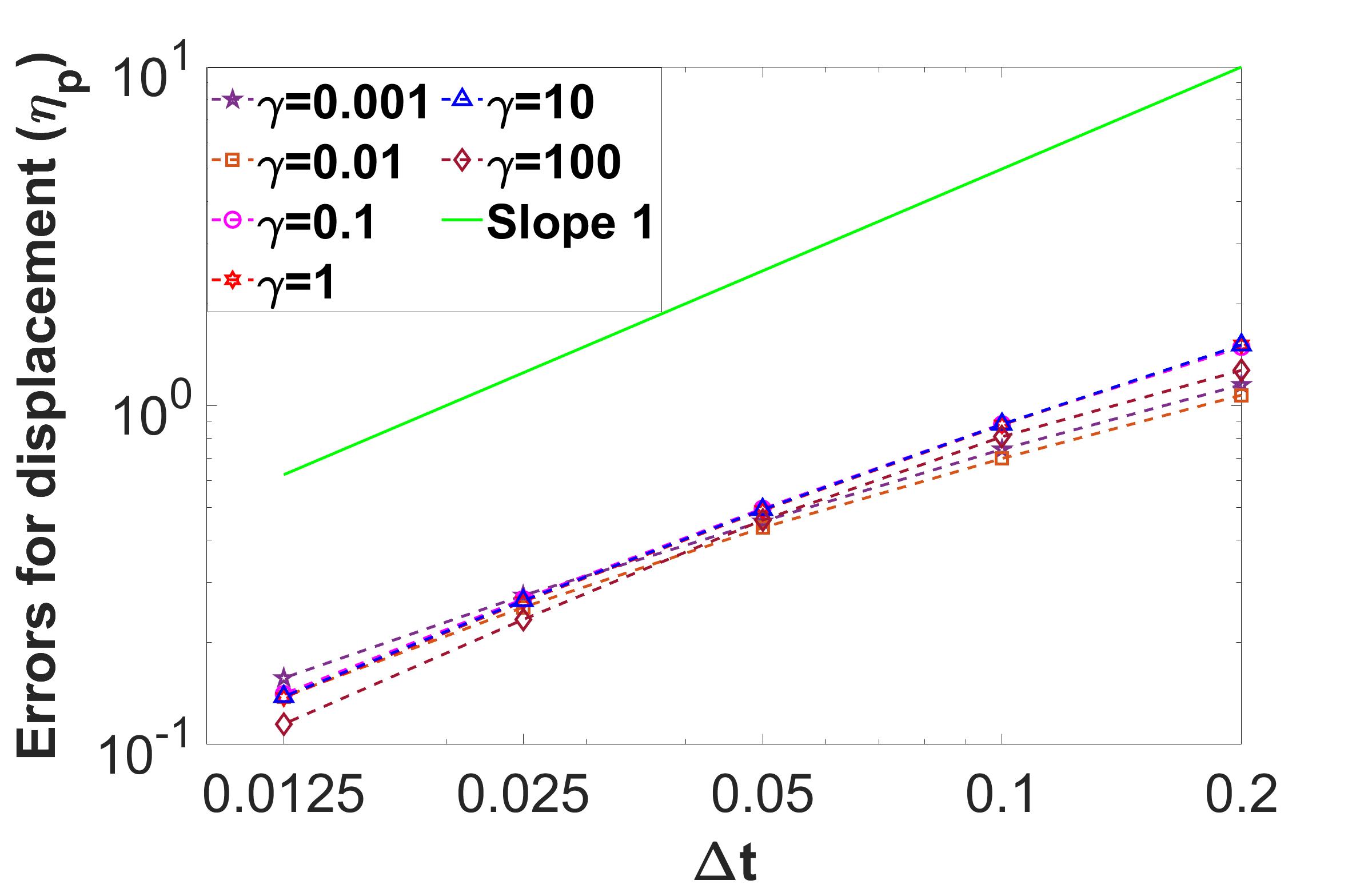}
\end{center}
\caption{Example 1, convergence plots for fluid velocity $\u_f$ (left),  Darcy velocity $\u_p$ (center) and displacement $\bbeta_p$ (right) for the iterative Robin-Robin algorithm with 10 iterations per time step for different values of $\gamma$.}\label{fig:maxiter10}
\end{figure}

Next, we consider the iterative scheme and the effect of $\gamma$ on the number of iterations.
We take the same values of $\gamma$ used above, i.e., $\gamma = 0.001, 0.01,  0.1, 1, 10, 100$.
Table~\ref{tab:iters} lists the average number of iterations needed for convergence for each $\gamma$ value. We observe that, for a given value of $\gamma$, the average number of iterations 
required to satisfy the stopping criterion \eqref{eq:stopping} per time step decreases as the time step size is reduced. On the other hand, for a given time step the average
number of iterations varies considerably with $\gamma$, with $\gamma = 0.1$ being the 
``optimal'' value. This value is in the range $[0.01,10]$ of values for $\gamma$ that resulted in smaller time discretization errors for the non-iterative scheme. We conclude that values of $\gamma$ in the ``optimal'' range, where the magnitudes of the velocity and stress terms in the Robin combination are comparable, result in both smaller errors in the non-iterative scheme and faster convergence in the iterative scheme. 

\begin{table}
\centering{}%
\begin{tabular}{ |c| c c c c c c|}
\hline
$\Delta t$& $\gamma=0.001$ & $\gamma=0.01$ & $\gamma=0.1$ & $\gamma=1$ & $\gamma=10$ & $\gamma=100$\\ 
\hline
0.2 & 100.00 &94.40&26.80 & 96.60 &49.00&99.80 \\  
0.1& 100.00&86.70&20.20 &89.20 &43.70 &93.90  \\
0.05& 100.00&54.20&17.05& 76.50  & 39.20 & 67.65\\
0.025 &76.20&23.675&12.60& 65.45 & 34.97 & 46.525\\
0.0125 &31.45&13.3375&8.72 &55.10 & 30.96 &32.325\\
 \hline
\end{tabular}
\caption{Example 1, average number of iterations 
required to satisfy the stopping criterion \eqref{eq:stopping} per time step of the iterative Robin-Robin scheme for different values of $\gamma$.}\label{tab:iters}
\end{table}

\subsection{Example 2: blood flow test}

In this example, we test the behavior of the method for a computationally challenging choice of physical parameters. We consider a benchmark on modeling blood flow through a section of an idealized artery. The Stokes equations model the blood flow in the lumen of the artery
and the Biot equations model the arterial wall. Let $R$ and $L$ be the 
radius and length of the artery, respectively.
The fluid domain is $\Omega_f = (0, L)\times
(-R, R)$. Its top and bottom boundaries are in contact
with the poroelastic arterial wall of thickness $r_p$.
See the computational domain in Figure~\ref{fig:bloodflow-domain_sketch} (left).

\begin{figure}[htb!]
\begin{minipage}{.65\textwidth}
\centering
\vspace{0.3cm}
\begin{overpic}[width=.80\textwidth, grid=false]{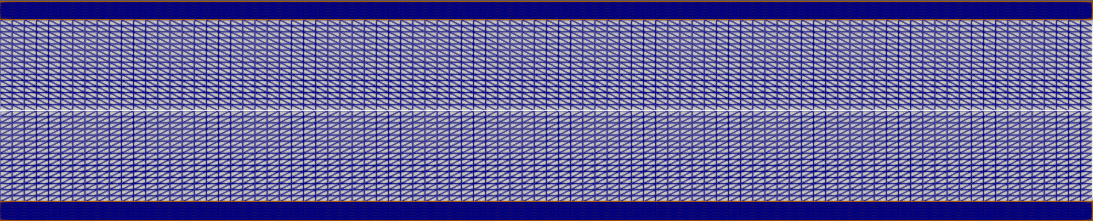}
\put(-7,8){$\Gamma_f^{in}$}
\put(101,8){$\Gamma_f^{out}$}
\put(46,-5){$\Gamma_p^{ext} $}
\put(46,23){$\Gamma_p^{ext} $}
\put(-3,22){\vector(1,-1){3}}
\put(-8,23){$\Gamma_p^{in}$}
\put(-3,-2){\vector(1,1){3}}
\put(-8,-6){$\Gamma_p^{in}$}
\put(103,-2){\vector(-1,1){3}}
\put(103,-6){$\Gamma_p^{out}$}
\put(103,22){\vector(-1,-1){3}}
\put(103,23){$\Gamma_p^{out}$}
\put(20,5){\textcolor{white}{$\Omega_f$}}
\put(75,13){\textcolor{white}{\vector(1,1){5}}}
\put(75,7){\textcolor{white}{\vector(1,-1){5}}}
\put(70,9){\textcolor{white}{$\Omega_p$}}
\end{overpic}
\end{minipage}
\hfill
\begin{minipage}{.27\textwidth}
\begin{overpic}[width=\textwidth, grid=false]{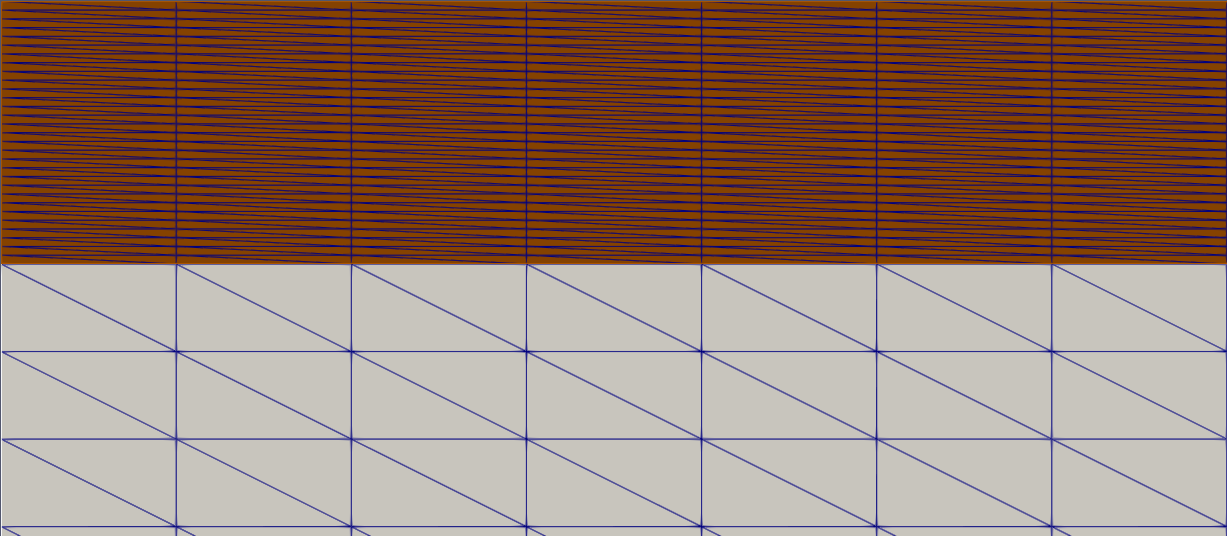}
\end{overpic}
\end{minipage}
\vspace{0.3cm}
\caption{Example 2, left: computational domain and mesh; right: zoomed-in view of the mesh at the fluid-structure interface.}
\label{fig:bloodflow-domain_sketch}
\end{figure}
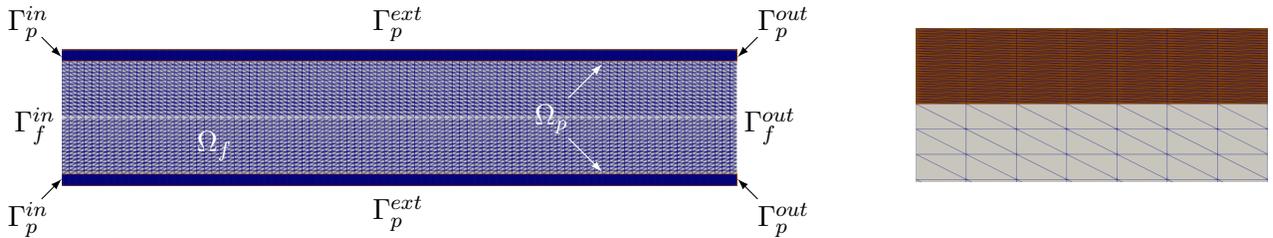

Since this is a 2D problem representing a slice of a 3D problem, 
we add an extra term to \eqref{eq:biot0} to account for
the fact that the 2D structure is actually part of a 3D
cylindrical tube:
\begin{equation*}
\rho_p \, \partial_{tt} \bbeta_p - \div \bs_p + \beta \, \bbeta_p = \f_p \quad \text{in} \quad \O_p\times (0,T],
\end{equation*}
where $\rho_p$ is the fluid density in the poroelastic region. 
The additional term (i.e., the last term at the left-hand side) 
comes from the axially symmetric two dimensional formulation, 
accounting for the recoil due to the circumferential strain \cite{Badia-Robin-FSI}. 
The body force terms $\f_f$ and $\f_p$ and external source $q_p$ are set to zero.  

Let $\Gamma_f^{in}= \{(0,y) |-R < y < R \}$ and
$\Gamma_f^{out} = \{(L,y) |-R<y<R\}$ be the 
inlet and outlet boundaries of the fluid domain, respectively. 
Following \cite{wy2022,fpsi-augmented,pb2024},
we prescribe the normal stress at both inlet and outlet:
\begin{equation}\label{bdyeqn1}
 \bs_f \, \n_f^{in} = -p_{in}(t) \, \n_f, \quad \text{ on }\quad\Gamma_f^{in}\times(0,T]; \quad  \bs_f \, \n_f^{out} = \mathbf{0}, \quad \text{ on }\quad\Gamma_f^{out}\times(0,T];
 \end{equation}
where $\n_f^{in}$ and $\n_f^{out}$ are the respective outward unit normals and 
\begin{equation}\label{pressurefunc}
p_{in}(t) = \left\{ \begin{array}{ll}
\ds \frac{P_{\max}}{2}\left( 1-\cos \big(\frac{2\pi t}{T_{\max}} \big) \right) \,, & \text{if} \,\ t\leq T_{\max}; \\[2ex]
0 \,, & \text{if}\,\  t>T_{\max},
\end{array} \right.
\end{equation}
with $P_{\max}=13,334$ dyn/cm$^2$ and $T_{\max}=0.003$ s. 
flow distribution and pressure field are often unknown, they are common in blood flow models. 
In a system in rest state, this inlet boundary condition generates a pressure pulse that travels
through the fluid and poroelastic structure domains. 
We set the end of the time interval of interest to $T = 0.021$ s
to avoid the pressure pulse reaching the outlet boundary.

With similar notation, we denote with $\Gamma_p^{in} = \{(0,y)| -R -r_p<
y < R \quad\text{or}\quad R<y <R+r_p\}$ and $\Gamma_p^{out} = \{(L,y)| -R- r_p <y <R \quad\text{or}\quad R<y <R+r_p\}$ the inlet and outlet boundaries of the poroelastic structure, respectively. We assume that the poroelastic structure is fixed at the inlet and
outlet boundaries:
  \begin{equation}\label{bdyeqn2}
  \bbeta_p= \mathbf{0},  \quad \text{ on }\quad\Gamma_p^{in}\cup\Gamma_p^{out}\times(0,T],
  \end{equation}
and for the Darcy velocity we impose the following drained boundary condition:
\begin{equation}\label{bdyeqn5}
  \u_p= \mathbf{0},  \quad \text{ on }\quad\Gamma_p^{in}\cup\Gamma_p^{out}\times(0,T].
\end{equation}
Let $\Gamma_p^{ext} = \{(x,y)| 0<x<L,y=-R -r_p \quad\text{or}\quad y =R+r_p\}$ be the external structure boundary. Following \cite{wy2022}, therein we impose:
  \begin{align}
&(\bs_e \n_p)\cdot\n_p=0,   &&\text{ on }\quad\Gamma_p^{ext}\times(0,T], \cl
& \bbeta_p\cdot \btau_{p}=0,  && \text{ on }\quad\Gamma_p^{ext}\times(0,T], \cl
&p_p=0,  && \text{ on }\quad\Gamma_p^{ext}\times(0,T].
  \end{align}

To save computational time, we halve the domain in 
Figure~\ref{fig:bloodflow-domain_sketch} (left) along the horizontal
symmetry axis, denoted with $\Gamma^{sym}_f$, and impose
the following symmetry conditions therein:
 \begin{align}
&\u_f \cdot \n_f=0,   &&\text{ on }\quad\Gamma^{sym}_f\times(0,T], \cl
& (\bs_f\n_f)\cdot \btau_f=0,  && \text{ on }\quad\Gamma^{sym}_f\times(0,T].
  \end{align}
  
The geometric and physical parameters for this test are summarized
Table \ref{tab:Parameters}. The physical parameters are chosen within the range of physical values for arterial blood flow.

\begin{remark}
The physical parameters in Table \ref{tab:Parameters} present several computational challenges. The closeness of the fluid density $\rho_f$ and poroelastic wall density $\rho_p$ may lead to the so-called added-mass-effect, which causes stability and convergence issues for classical Neumann-Dirichlet methods \cite{Causin-added-mass-effect,bqq2009}. The small values of permeability $K$ and storativity $s_0$ may lead to poroelastic locking. The high stiffness of the arterial wall due to large Lam\'e parameters $\lambda_p$ and $\mu_p$ results in large stress along the interface and affects the choice of the Robin parameter $\gamma$. 
\end{remark}

The computational mesh is shown in 
Figure~\ref{fig:bloodflow-domain_sketch} (left), with a zoomed-in view
around the fluid-structure interface in Figure~\ref{fig:bloodflow-domain_sketch} (right).
The time step is set to $\Delta t=10^{-4}$~s. Regarding the choice of the Robin parameter $\gamma$, the $O(10^6)$ values of $\lambda_p$ and $\beta$ indicate that the stress is several orders of magnitude larger than the velocity, suggesting a large value of $\gamma$. With estimated magnitudes $O(1)$ for the interface velocity and $O(10^3)$ for the interface stress, obtained from a preliminary computation, we set $\gamma_f = \gamma_p = \gamma = 1000$ in order to balance the two terms in the Robin combination. We later test the robustness of the non-iterative method for different values of $\gamma$.

\begin{table}
\centering{}%
\begin{tabular}{ c c c c}
\hline
Parameter& Symbol & Units & Reference value\\ 
\hline
Radius & $R$ & cm & 0.5 \\  
Length & $L$ &cm & 6   \\
Poroelastic wall thickness &$r_p$& cm  & 0.1\\
Poroelastic wall density &$\rho_p$& $\text{g/cm}^3$ & 1.1\\
Fluid density& $\rho_f$ & $\text{g/cm}^3$ & 1.0\\
Dyn. viscosity&$\mu_f$  & \text{g/cm-s}&0.035\\
Spring coeff.&$\beta$& $\text{dyn/cm}^4$ & $4 \times 10^{6} $\\
Mass storativity&$s_0 $& $\text{cm}^2$/\text{dyn}&$10^{-3}$\\
Permeability&K& $\text{cm}^2$ & $diag(1,1)\times10^{-6}$\\
Lam\'{e} coeff. &$\mu_p$&$\text{dyn/cm}^2$ &$5.575 \times 10^5$ \\
Lam\'{e} coeff. &$\lambda_p$&$\text{dyn/cm}^2$ &$1.7 \times 10^6$\\
 BJS coeff.&$\alpha_{BJS}$ & &1\\
 Biot-Willis constant& $\alpha$ & & 1 \\
 \hline
\end{tabular}
\caption{Example 2, geometric and physical parameters.}\label{tab:Parameters}
\end{table}

Figure \ref{paraview:computed-pressure} shows the fluid pressure $p_f$ and Darcy pressure $p_p$ in their corresponding domains computed by the Robin-Robin scheme in the non-iterative version and the monolithic scheme
at three different times. We clearly see the propagation of the pressure wave and an excellent qualitative match in the pressures computed by the two methods. 
Figure \ref{paraview:computed-velocity} shows the velocity fields $\u_f$ and $\u_p$ computed by the same two schemes at the same times used in Figure \ref{paraview:computed-pressure}.
Again, we see a great qualitative match in the solutions computed 
by the two methods.
The faint lines that can been seen along the horizontal symmetry line in 
the plots in Figure \ref{paraview:computed-pressure} and
\ref{paraview:computed-velocity} are due to the fact that 
we have solved the problem on half of the domain
and mirrored the results in Paraview to show the entire
vessel.

\begin{figure}[htb!]
\centering
\vskip .3cm
\begin{overpic}[width=0.325\textwidth,grid=false]{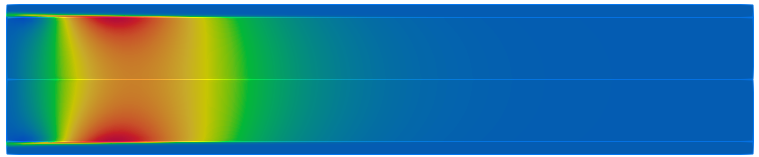}
\put(32,22){\textbf{t=0.007}}
\end{overpic}
\begin{overpic}[width=0.325\textwidth,grid=false]{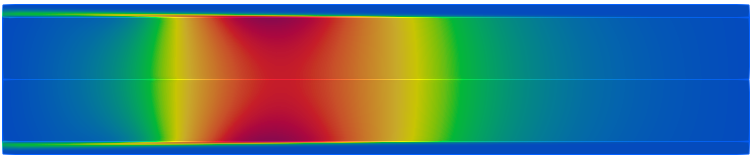}
\put(32,22){\textbf{t=0.014}}
\end{overpic}
\begin{overpic}[width=0.325\textwidth,grid=false]{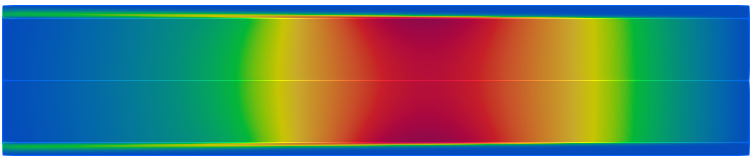}
\put(32,22){\textbf{t=0.021}}
\end{overpic}
\\
\begin{overpic}[width=0.325\textwidth,grid=false]{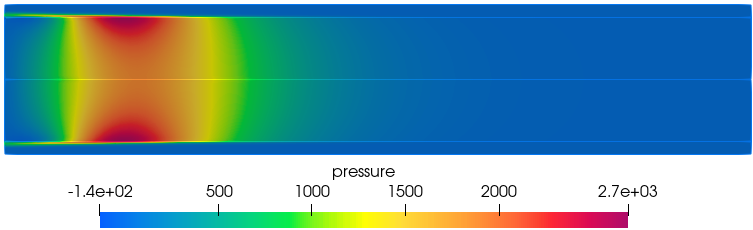}
\end{overpic}
\begin{overpic}[width=0.325\textwidth,grid=false]{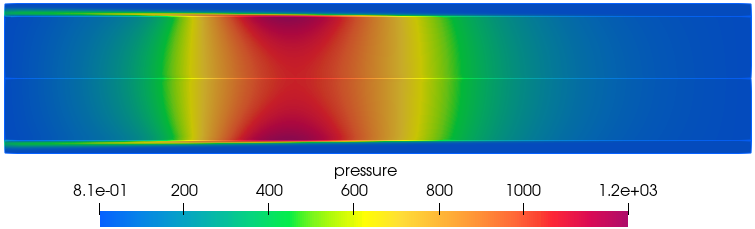}
\end{overpic}
\begin{overpic}[width=0.325\textwidth,grid=false]{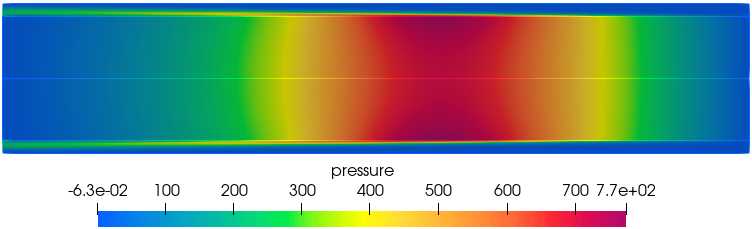}
\end{overpic}
\caption{Example 2, fluid and Darcy pressure computed by the 
non-iterative Robin-Robin scheme (top) 
and the monolithic scheme (bottom) at times $t= 0.007, 0.014, 0.021$ s 
(from left to right).}
\label{paraview:computed-pressure}

\end{figure}
\begin{figure}[htb!]
\centering
\begin{overpic}[width=0.325\textwidth,grid=false]{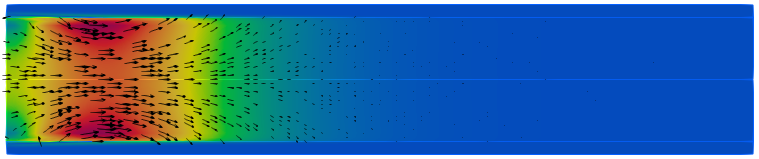}
\put(32,22){\textbf{t=0.007}}
\end{overpic}
\begin{overpic}[width=0.325\textwidth,grid=false]{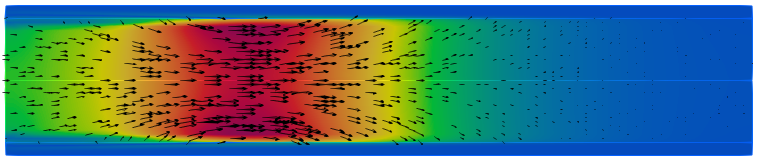}
\put(32,22){\textbf{t=0.014}}
\end{overpic}
\begin{overpic}[width=0.325\textwidth,grid=false]{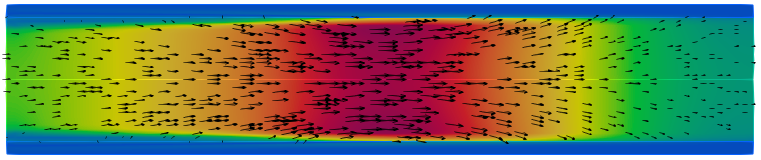}\put(32,22){\textbf{t=0.021}}
\end{overpic}
\\
\begin{overpic}[width=0.325\textwidth,grid=false]{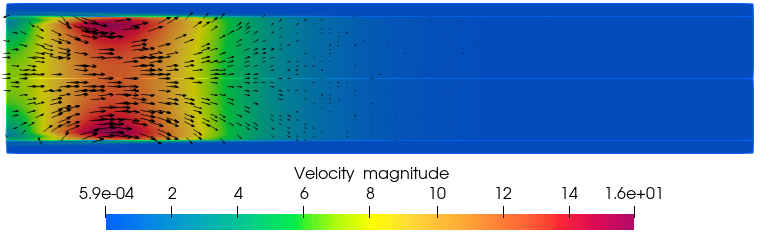}
\end{overpic}
\begin{overpic}[width=0.325\textwidth,grid=false]{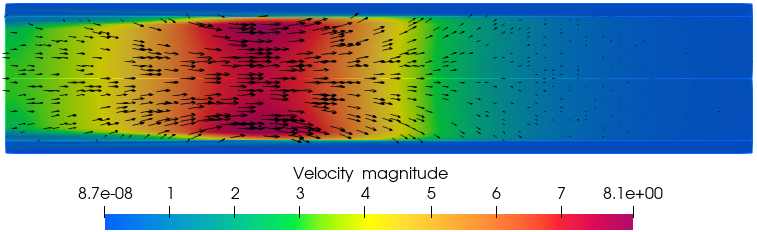}
\end{overpic}
\begin{overpic}[width=0.325\textwidth,grid=false]{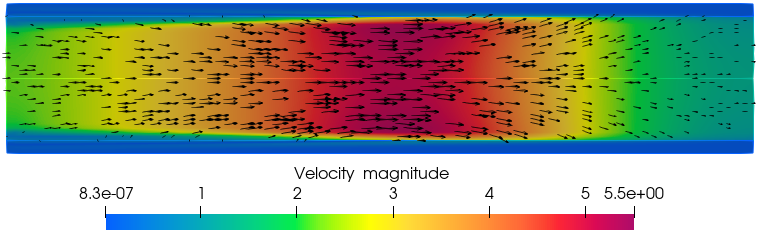}
\end{overpic}
\caption{Example 2, fluid and Darcy velocity computed by the 
non-iterative Robin-Robin scheme (top) 
and the monolithic scheme (bottom) at times $t= 0.007, 0.014, 0.021$ s 
(from left to right). The arrows represent the velocity vectors
and the color represents the velocity magnitudes.}
\label{paraview:computed-velocity}
\end{figure}

To further compare the solutions given by the different methods, Figure~\ref{matplots:pressure-velocity-disp} displays the fluid pressure, vertical fluid velocity,
vertical Darcy velocity, and vertical structure
displacement along the interface computed at different times
by the non-iterative and iterative Robin-Robin methods, and the monolithic method. We see that the curves given by the iterative Robin-Robin method and the monolithic method
overlap. There is a slight difference with the curves given by the non-iterative Robin-Robin method, which becomes less noticeable as time passes. It seems that initially the lack of iterations in the Robin-Robin method slows down the wave. Overall, we conclude that the non-iterative Robin-Robin method provides accuracy comparable to the other two methods at a significantly reduced computational cost.

\begin{figure}[ht!]
\begin{center}
\includegraphics[width=0.325\textwidth]{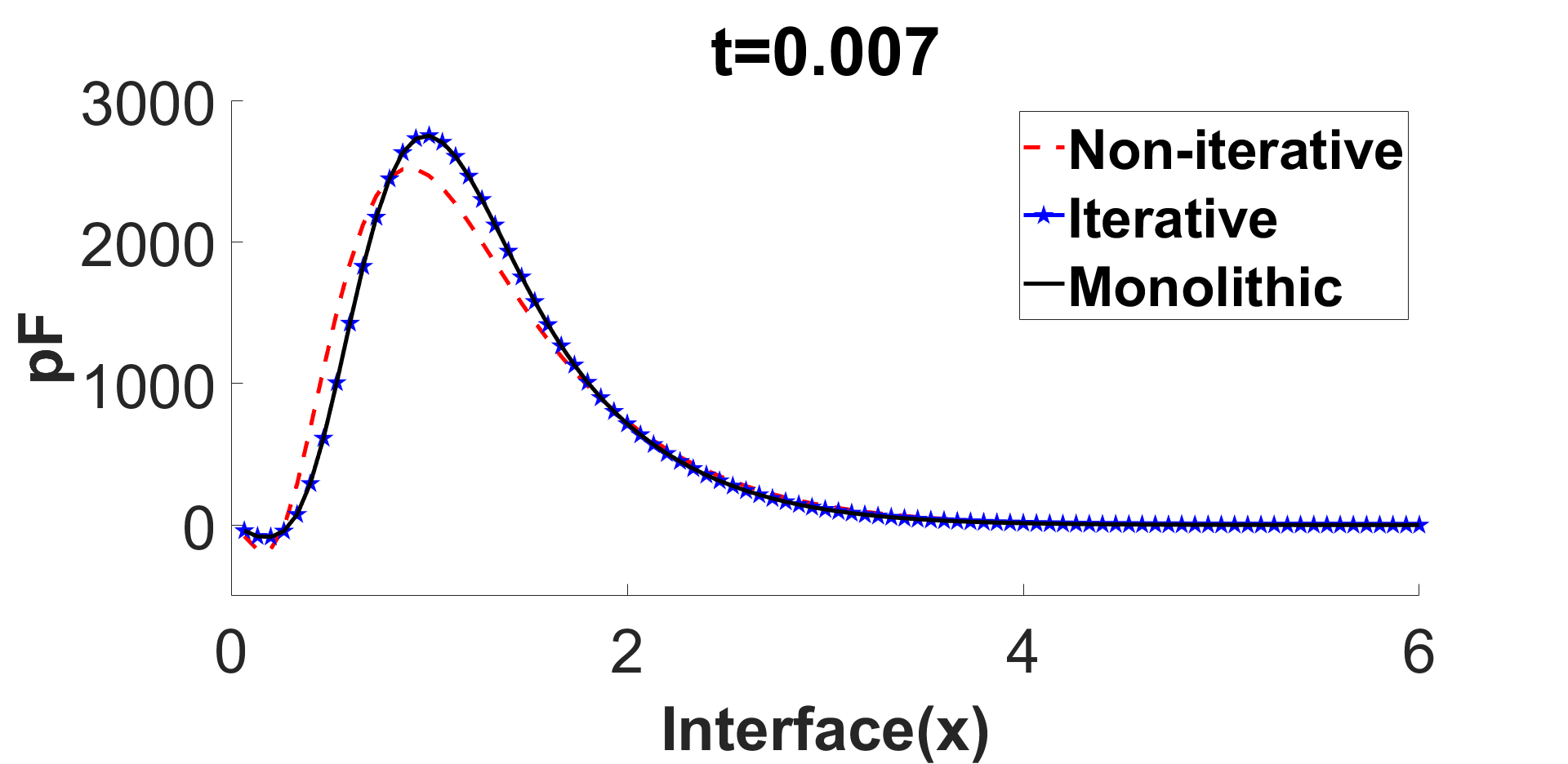}
\includegraphics[width=0.325\textwidth]{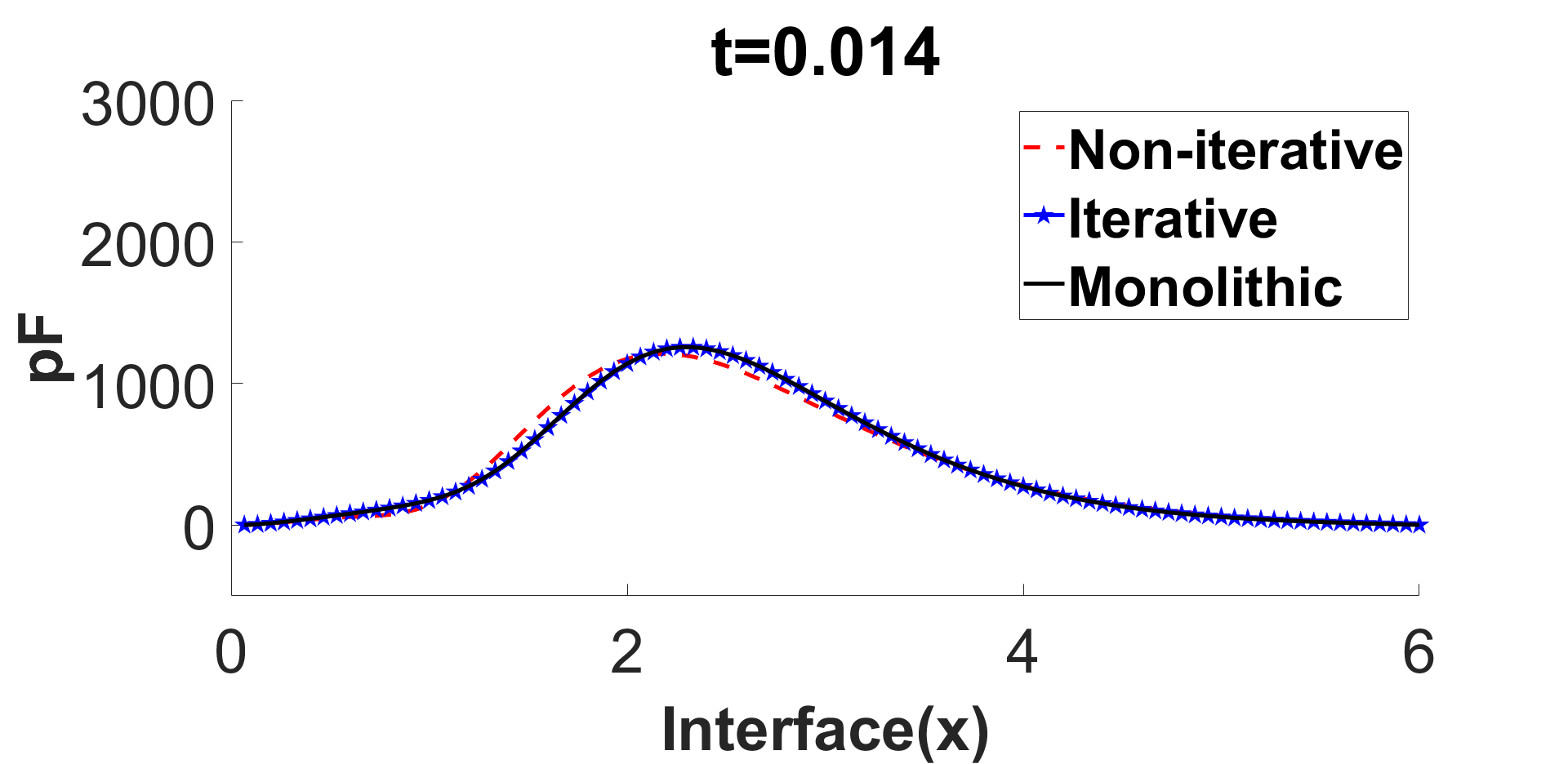}
\includegraphics[width=0.325\textwidth]{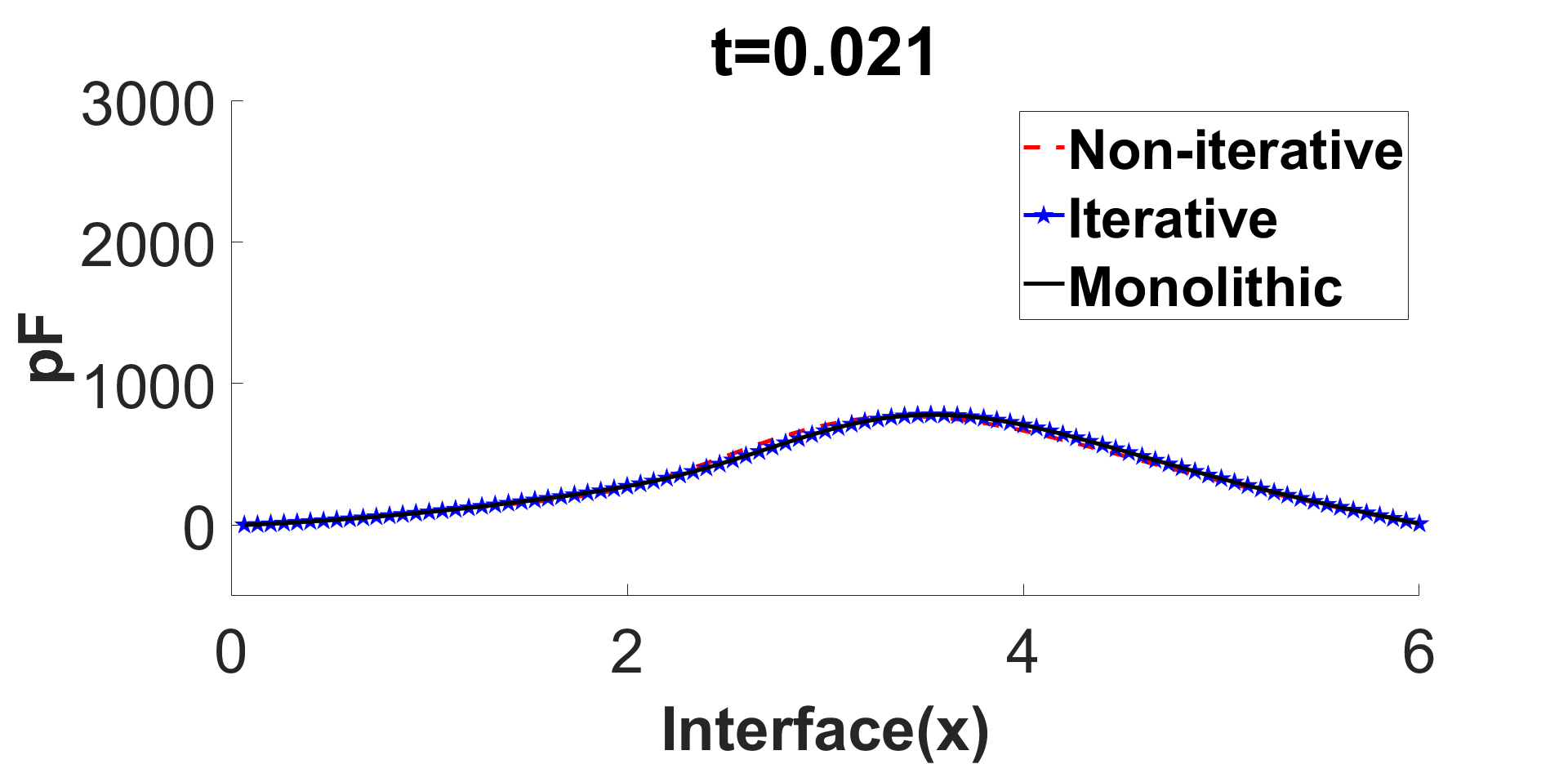}\\
\includegraphics[width=0.325\textwidth]{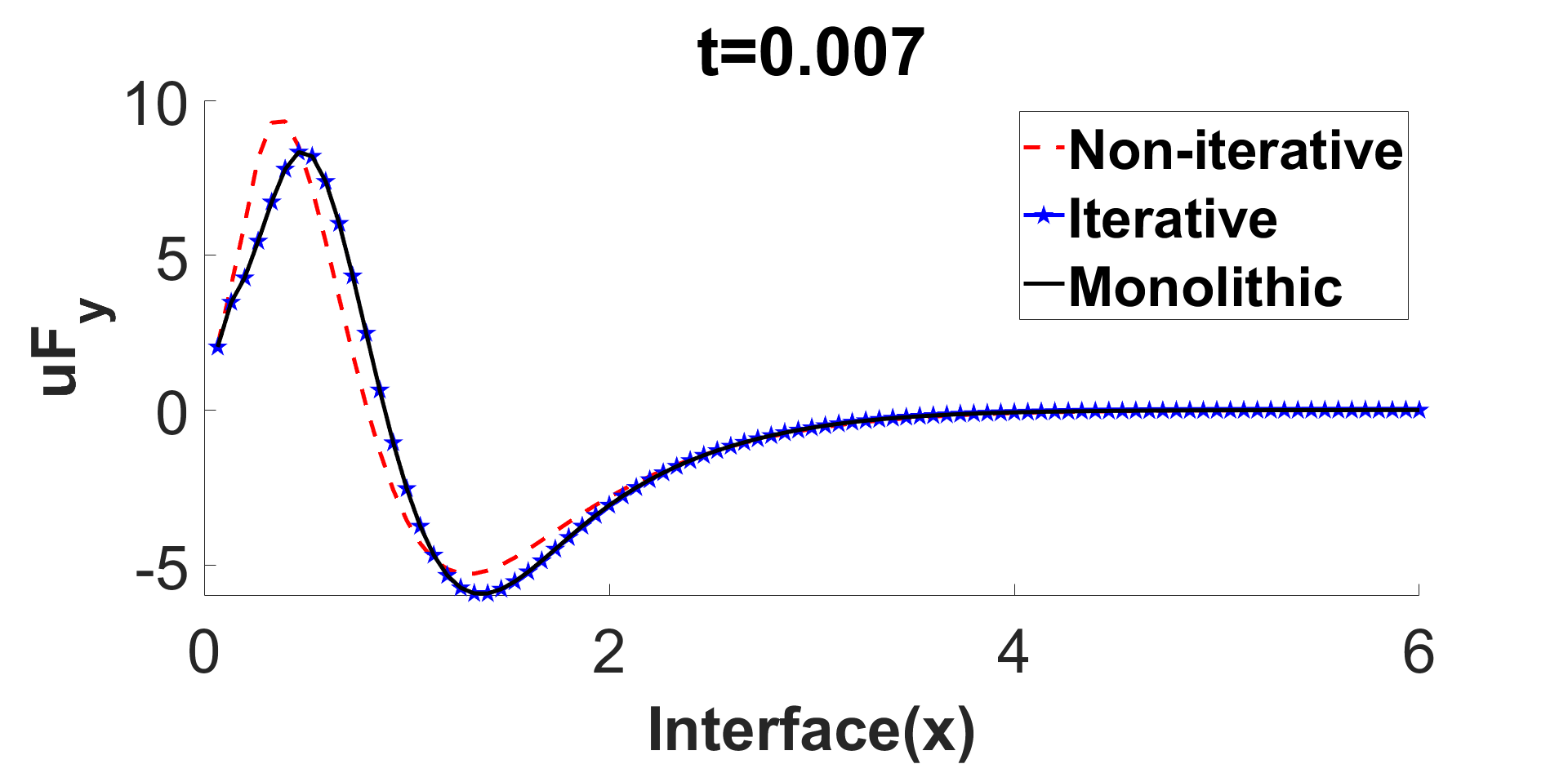}
\includegraphics[width=0.325\textwidth]{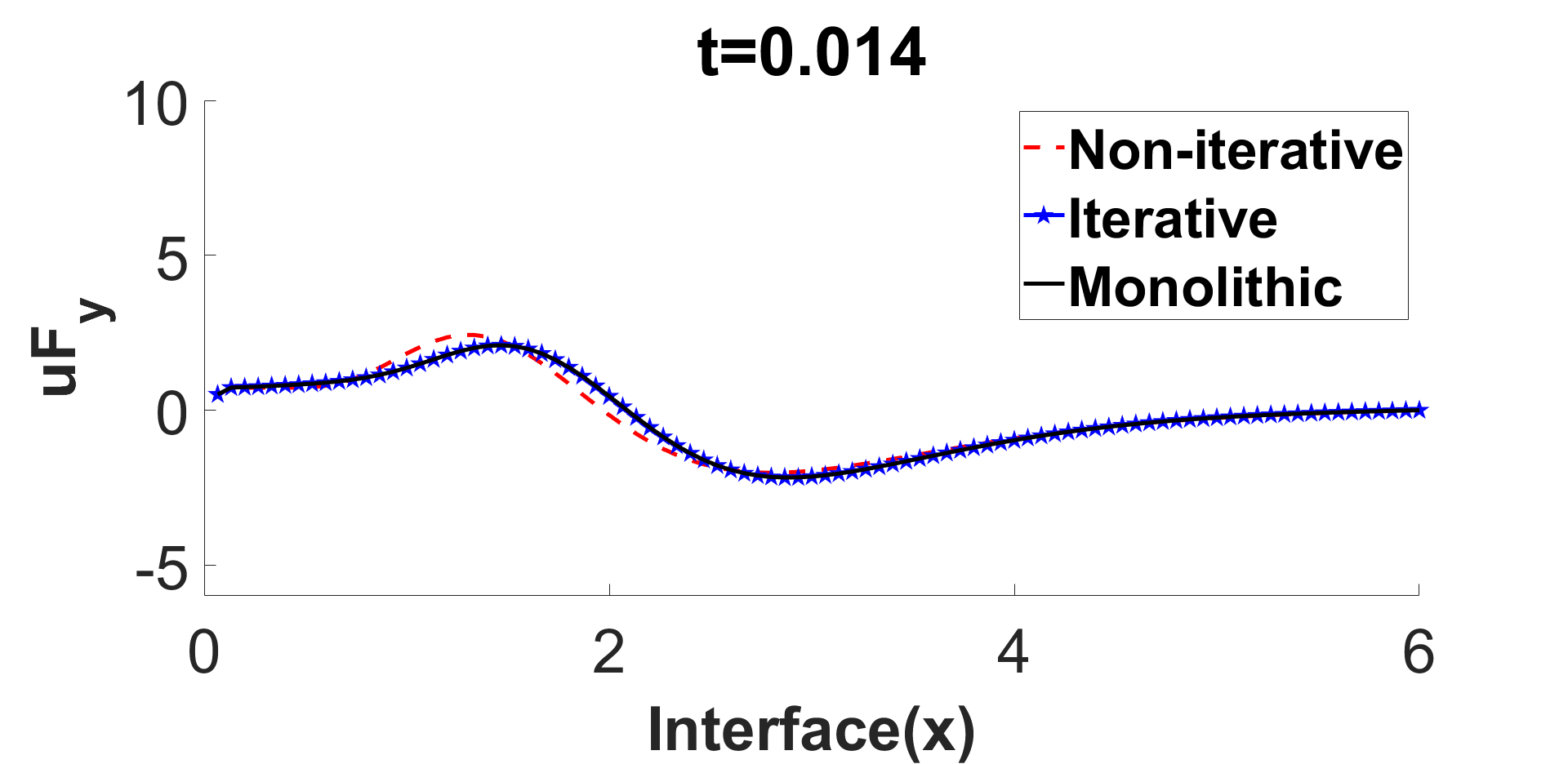}
\includegraphics[width=0.325\textwidth]{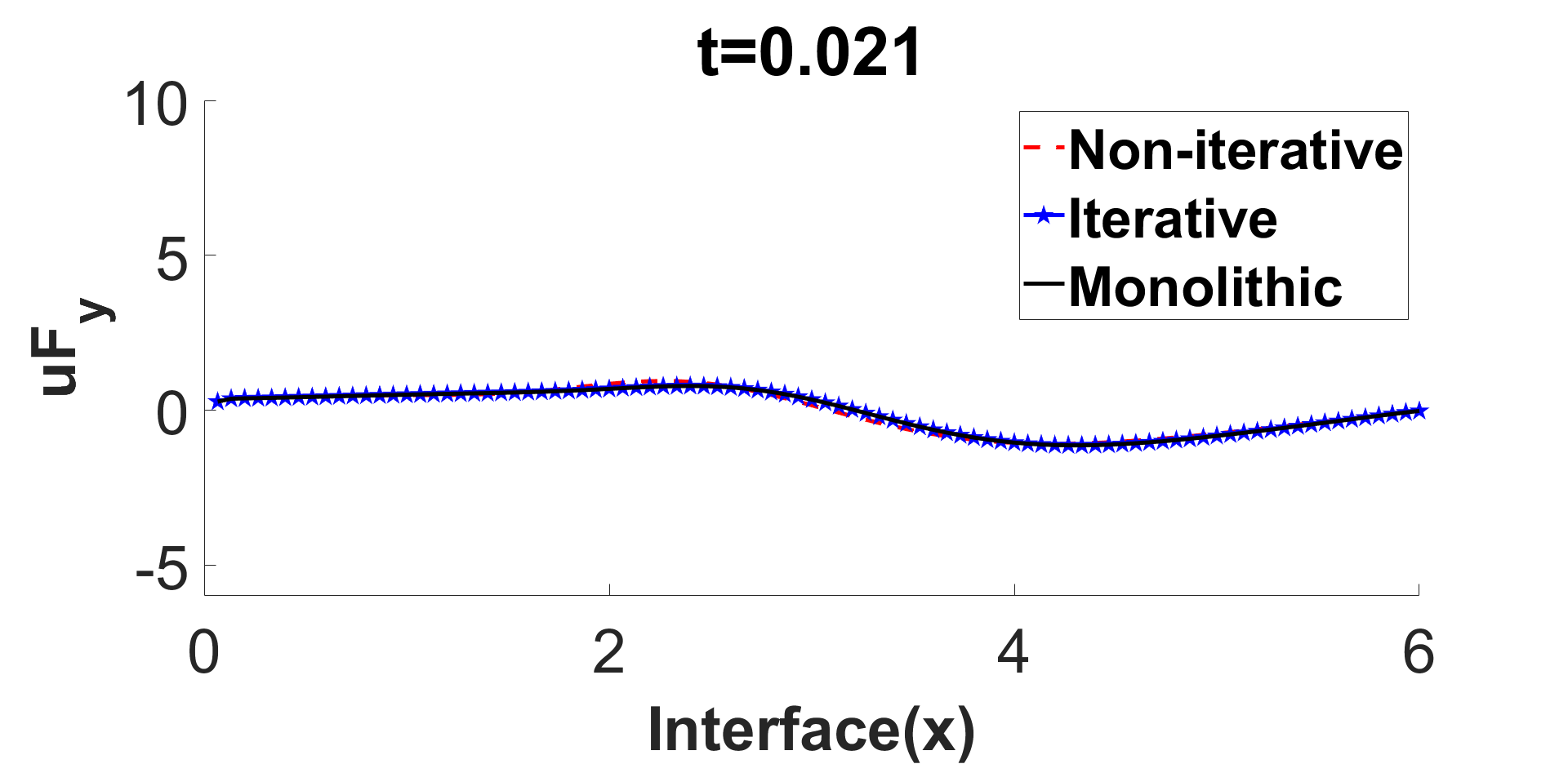}\\
\includegraphics[width=0.325\textwidth]{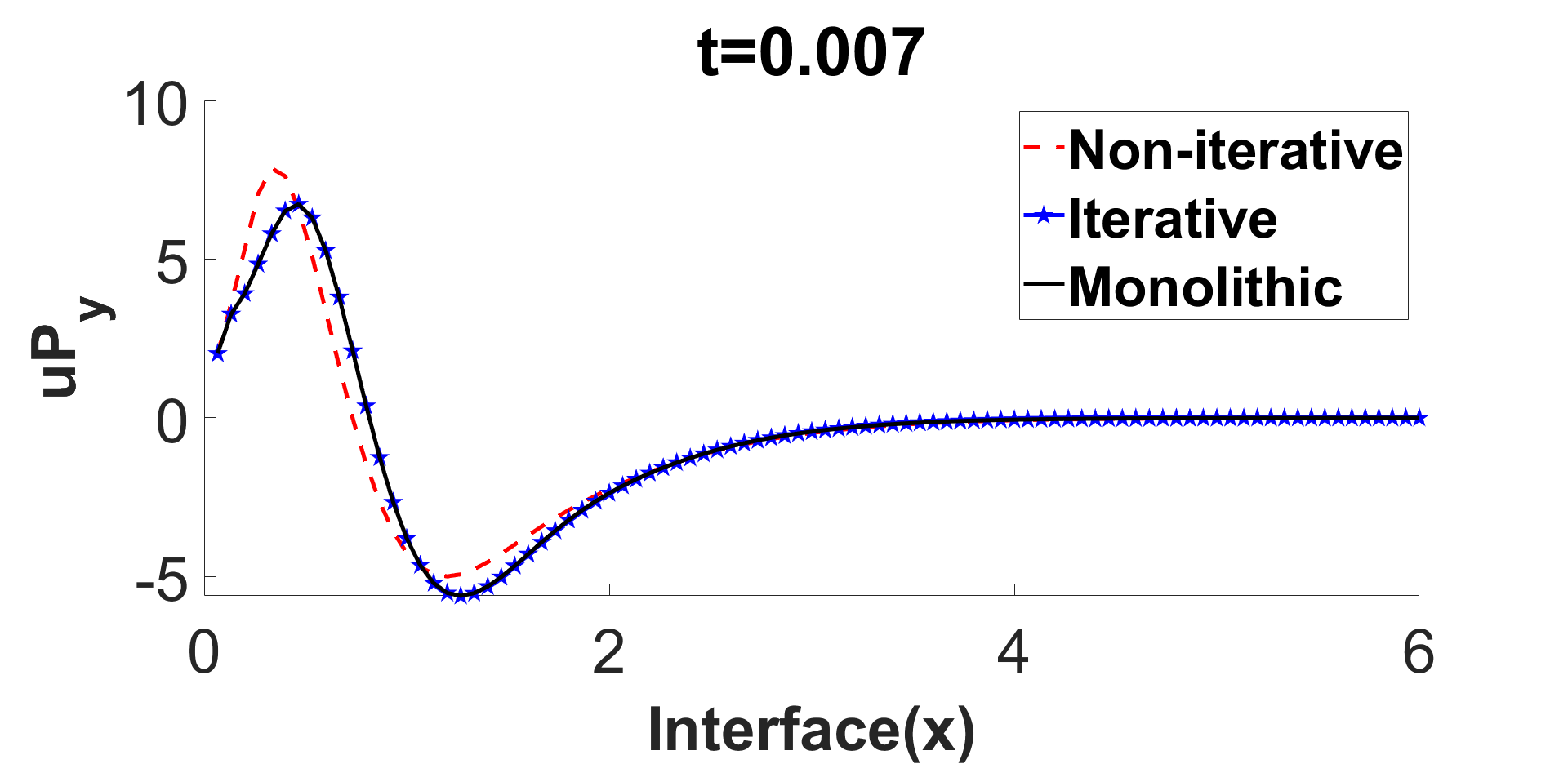}
\includegraphics[width=0.325\textwidth]{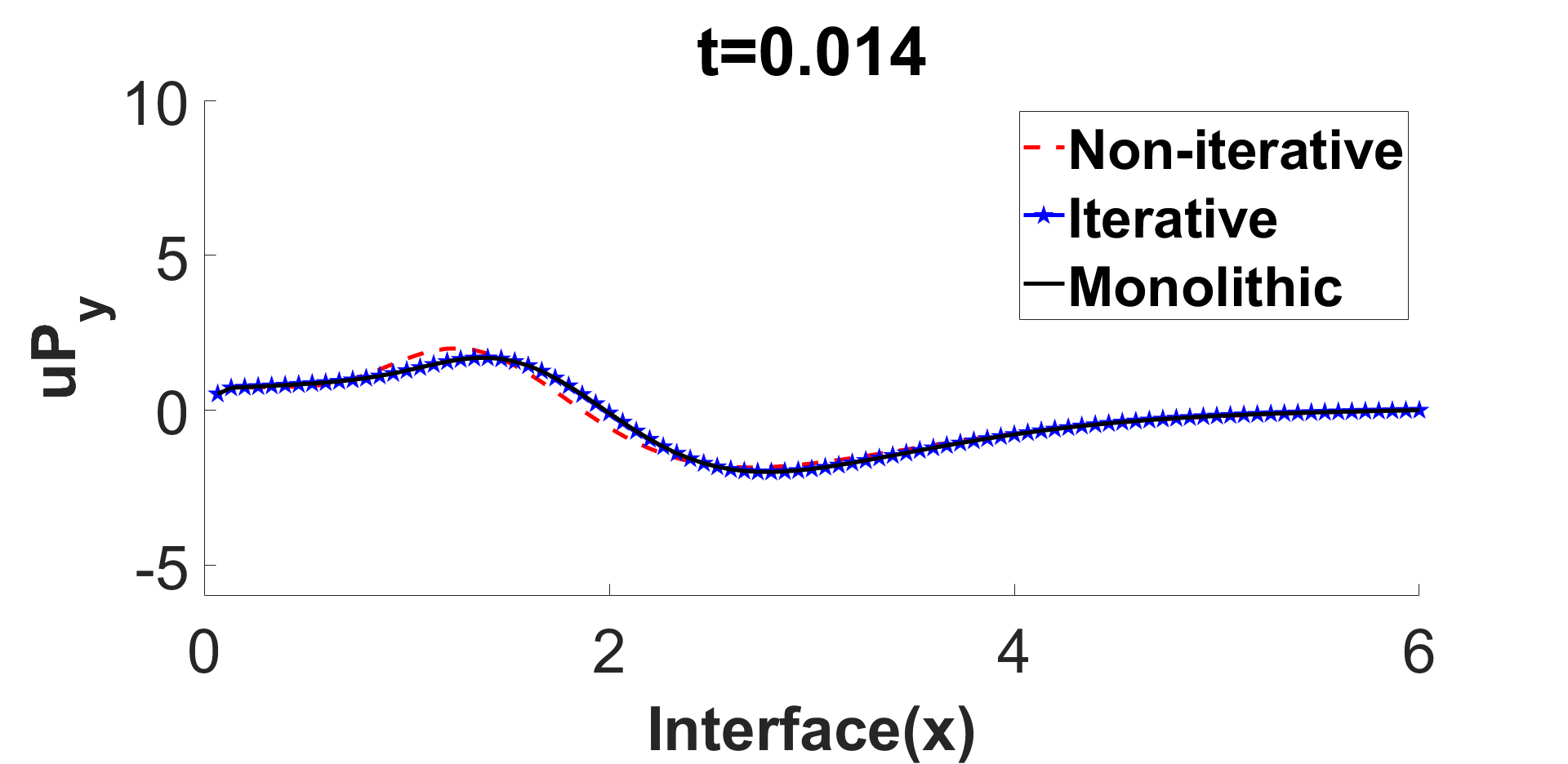}
\includegraphics[width=0.325\textwidth]{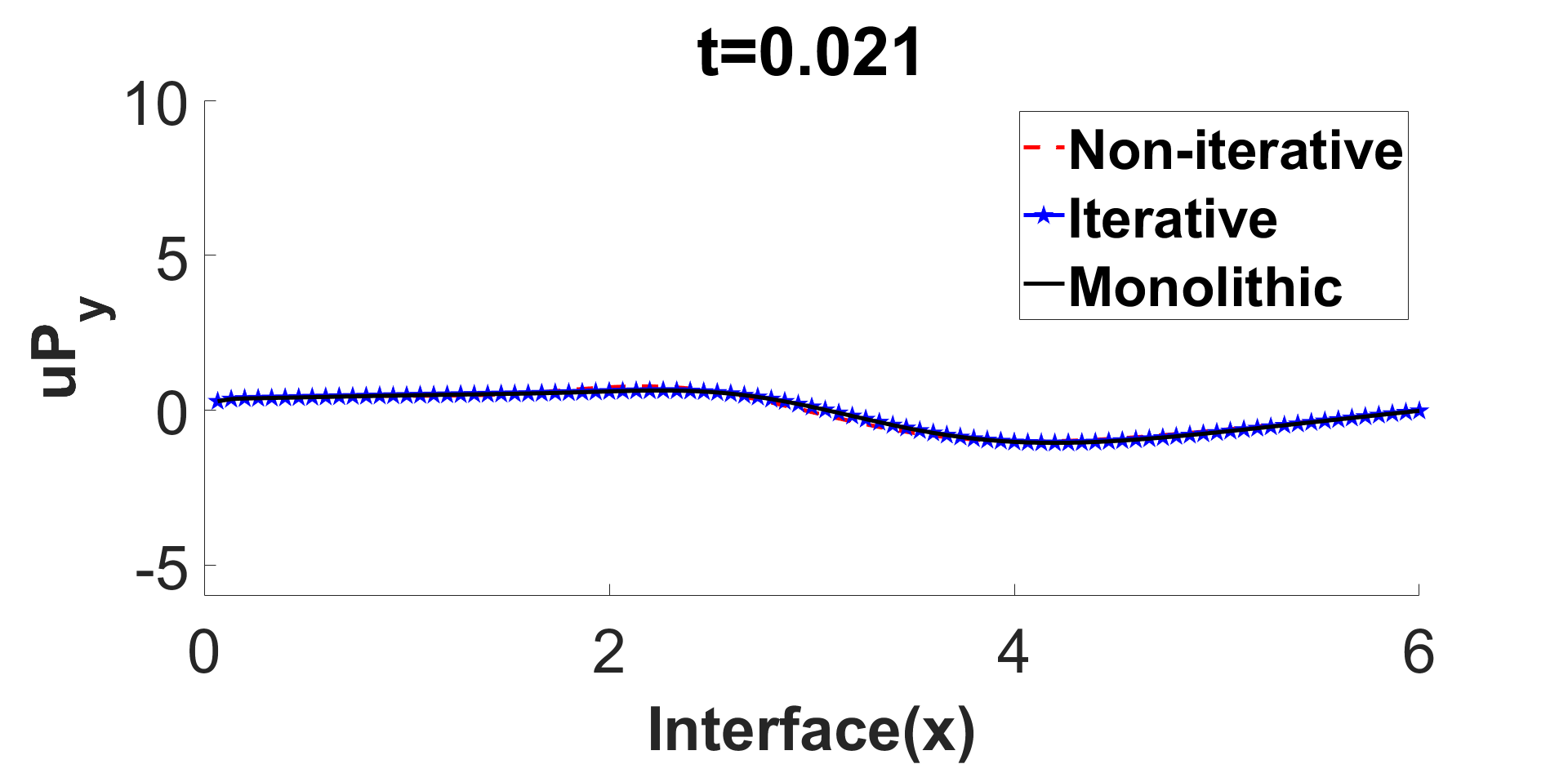}\\
\includegraphics[width=0.325\textwidth]{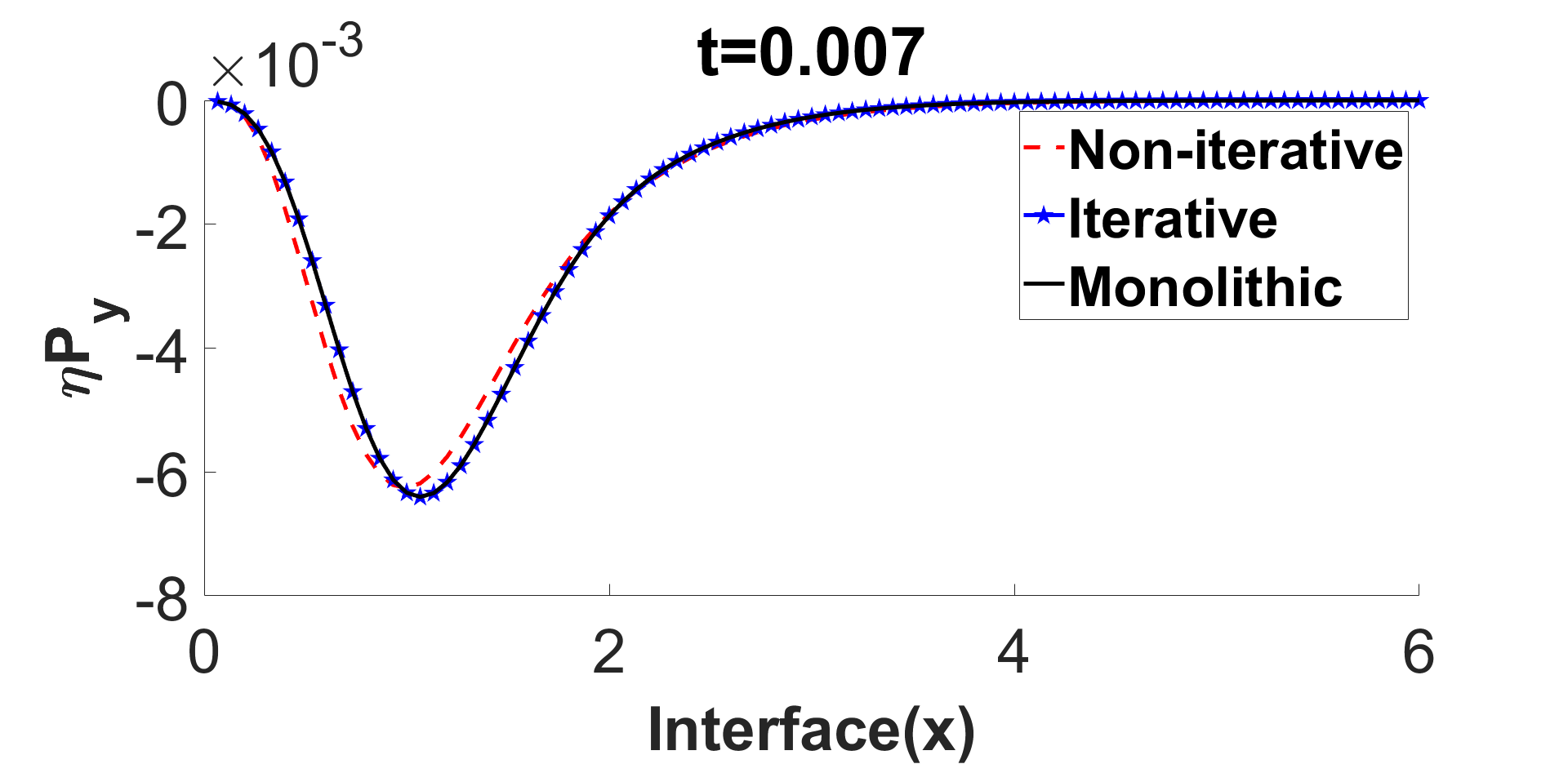}
\includegraphics[width=0.325\textwidth]{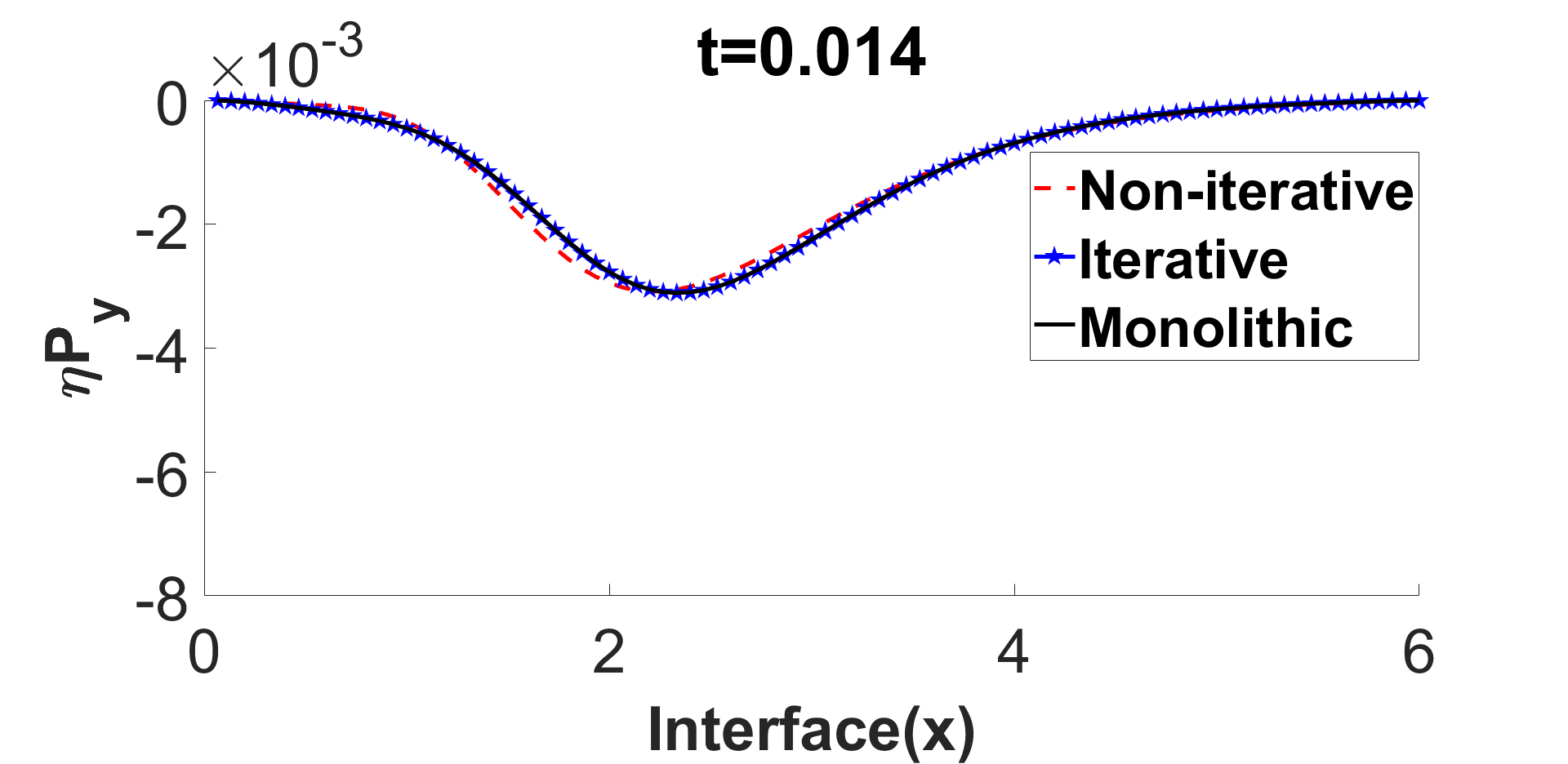}
\includegraphics[width=0.325\textwidth]{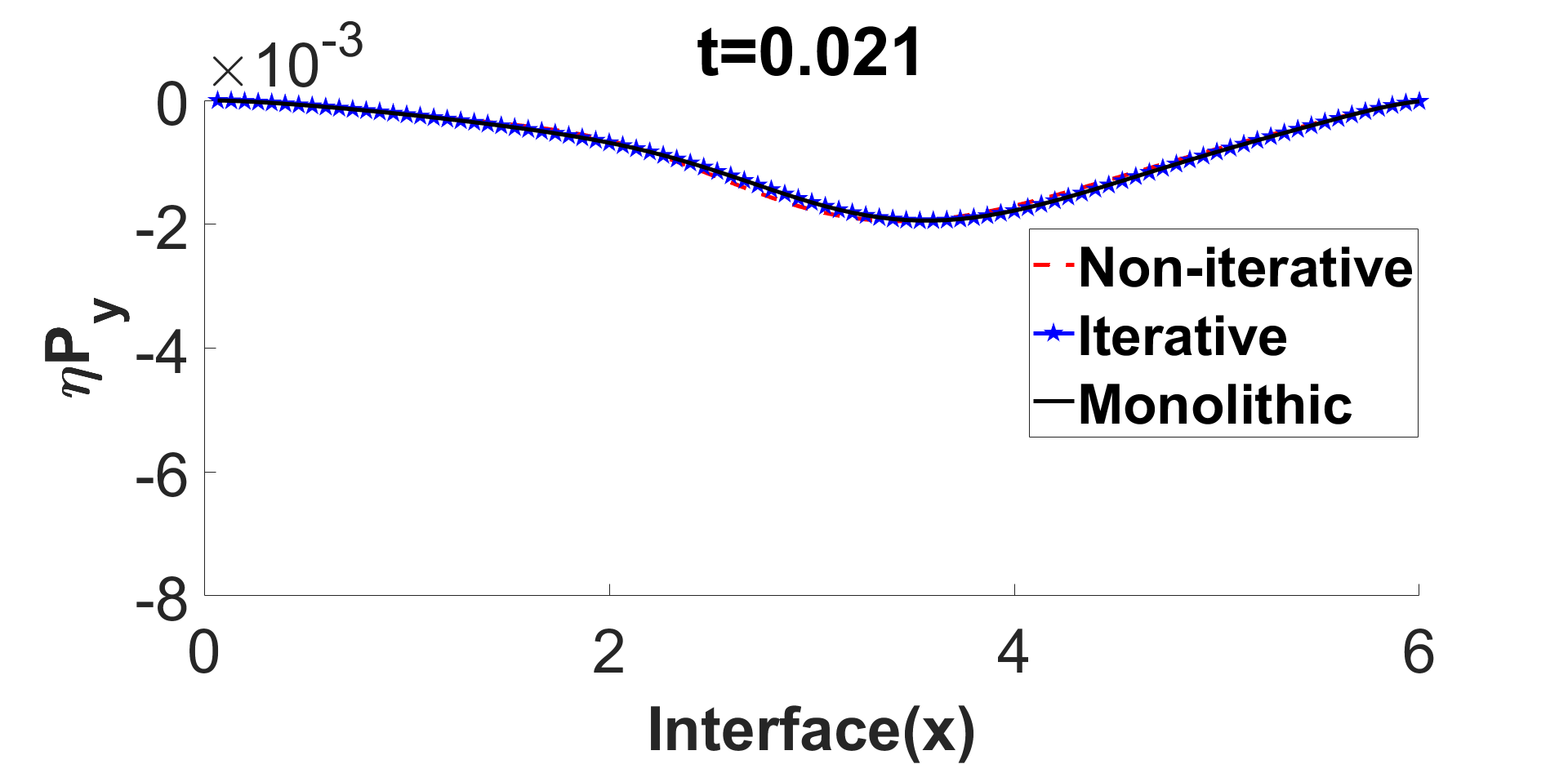}
\end{center}
\caption{Example 2, fluid pressure (first row), vertical fluid velocity (second row),
vertical Darcy velocity (third row), vertical structure
displacement (fourth row) along the interface computed by the Robin-Robin method in the non-iterative and iterative versions, and the monolithic method
at times $t= 0.007, 0.014, 0.021$ s (from left to right).}\label{matplots:pressure-velocity-disp}
\end{figure}

Finally, we test the robustness of the non-iterative Robin-Robin method to the Robin parameter $\gamma$. Figure~\ref{matplots:pressure-velocity-disp-gamma} shows the solution along the interface for $\gamma = 10, 100, 500, 1000, 2000$. The curves for $\gamma = 100, 500, 1000$ are very similar and the curves for $\gamma = 2000$ deviate from them slightly. The curves for $\gamma = 10$ are significantly different. As in the previous example, we conclude that there is a range of values of $\gamma$, for which the two terms in the Robin combination are of comparable magnitudes, that produce accurate results. Values outside of this range may result in reduction in accuracy, especially values of $\gamma$ that are too small. 

\begin{figure}[ht!]
\begin{center}
\includegraphics[width=0.325\textwidth]{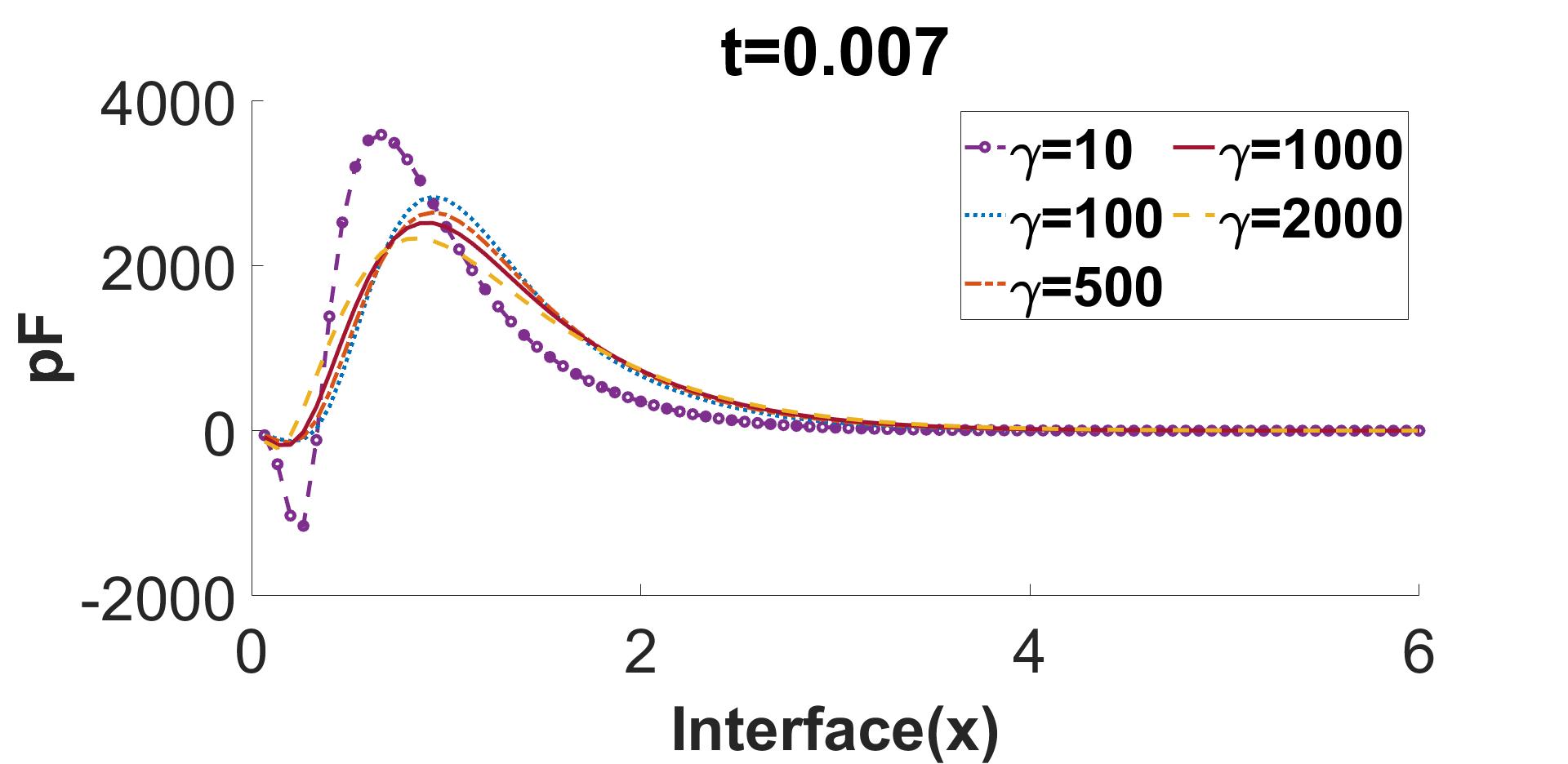}
\includegraphics[width=0.325\textwidth]{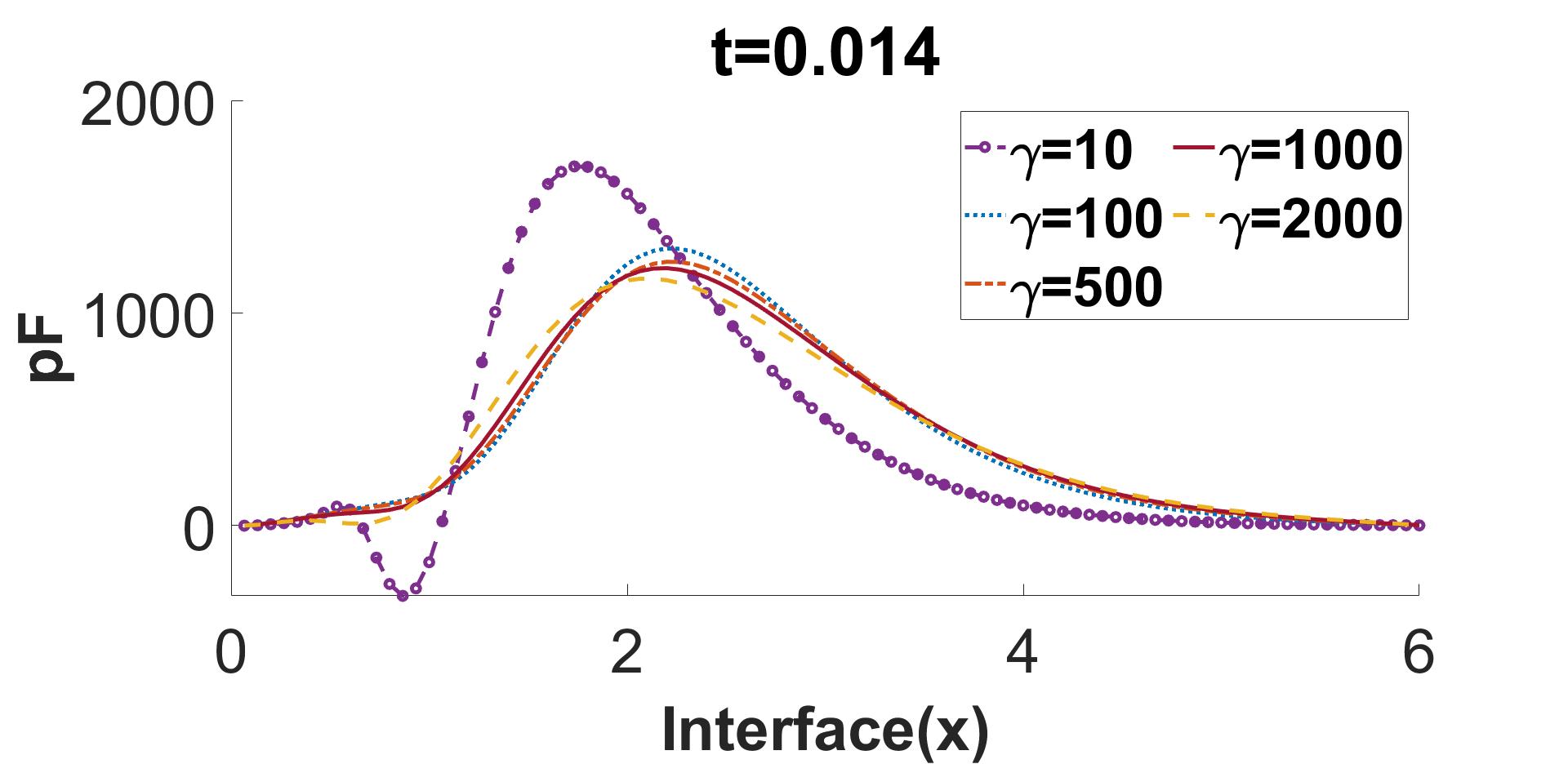}
\includegraphics[width=0.325\textwidth]{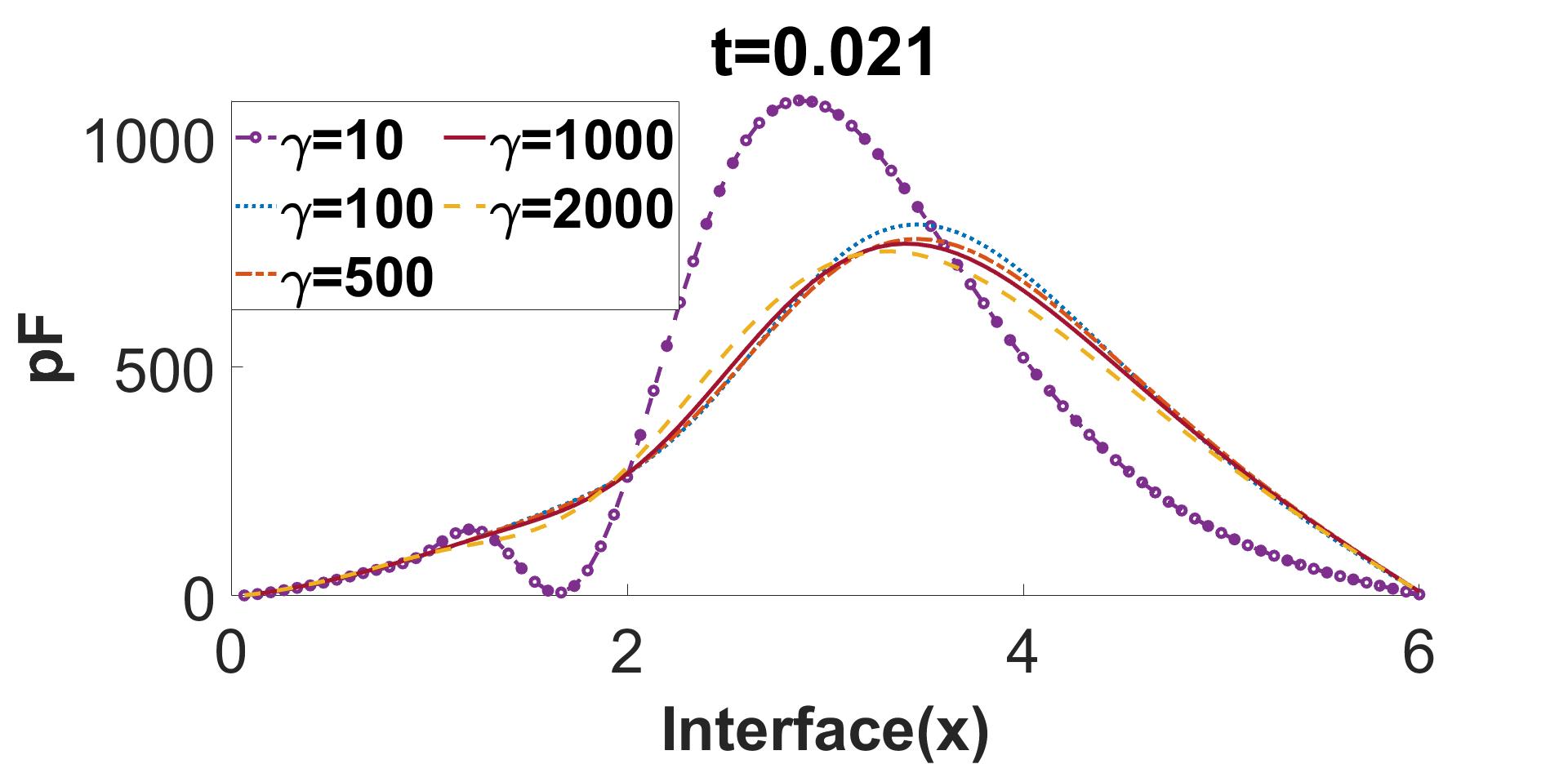}\\
\includegraphics[width=0.325\textwidth]{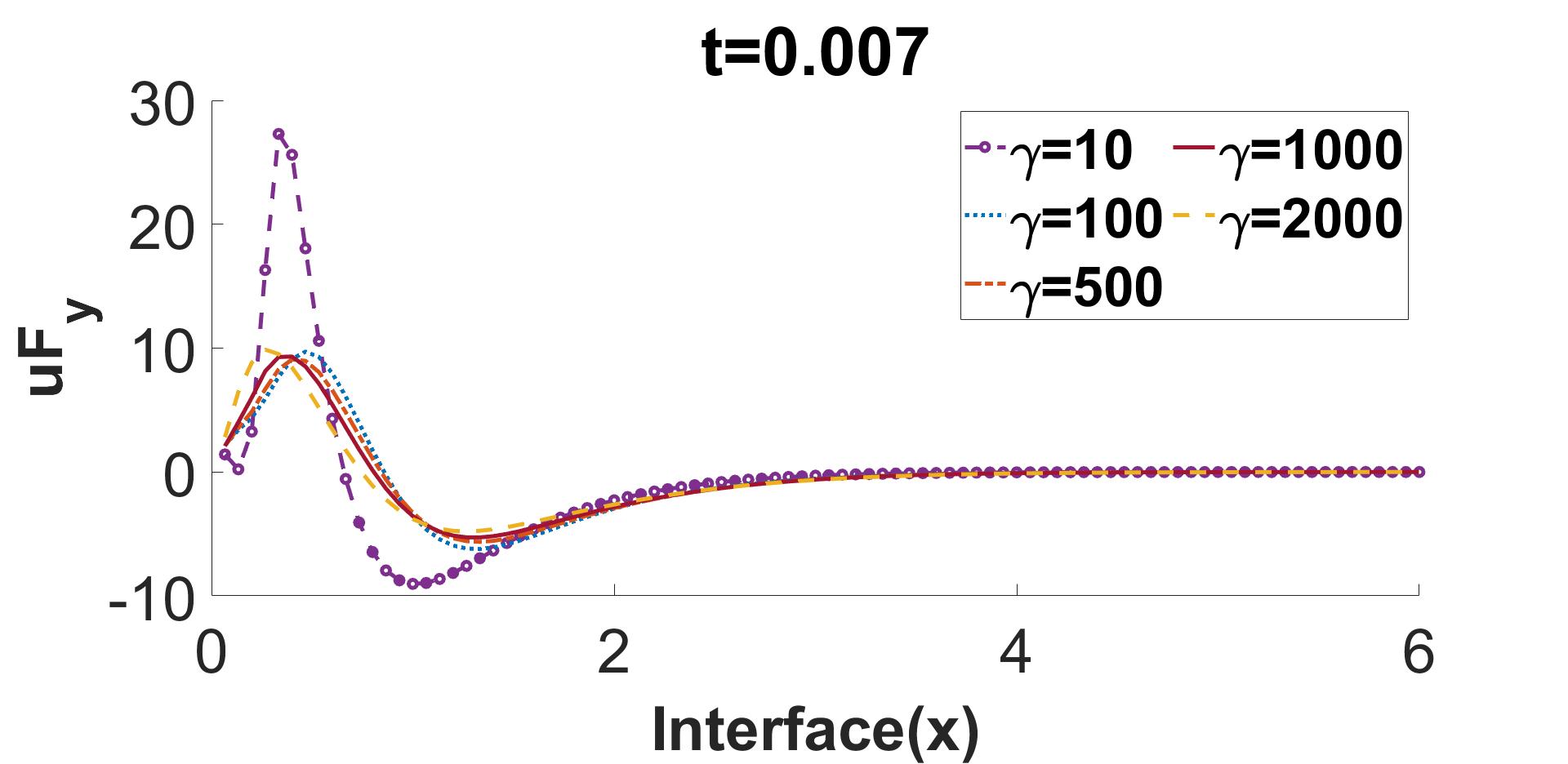}
\includegraphics[width=0.325\textwidth]{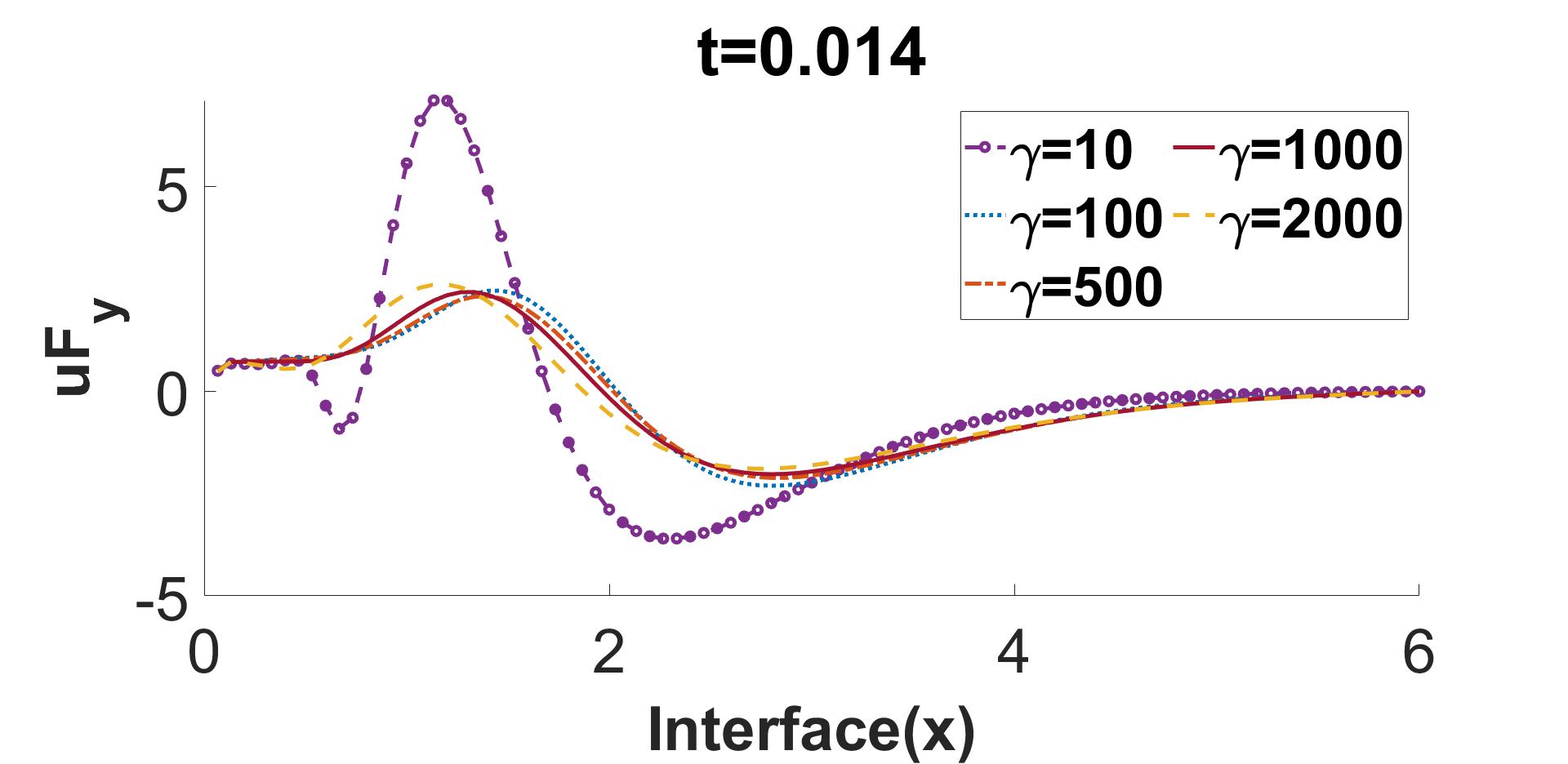}
\includegraphics[width=0.325\textwidth]{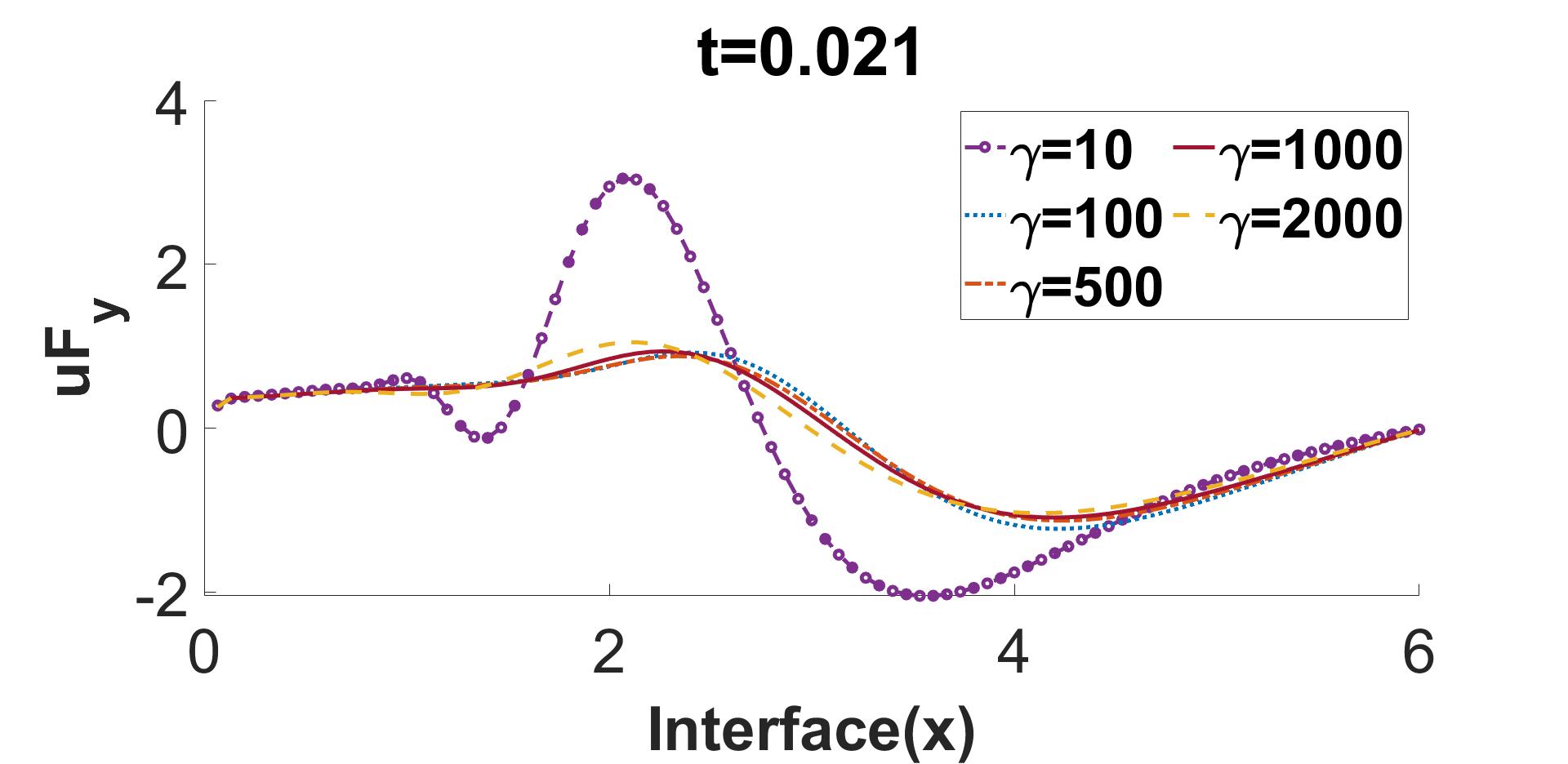}\\
\includegraphics[width=0.325\textwidth]{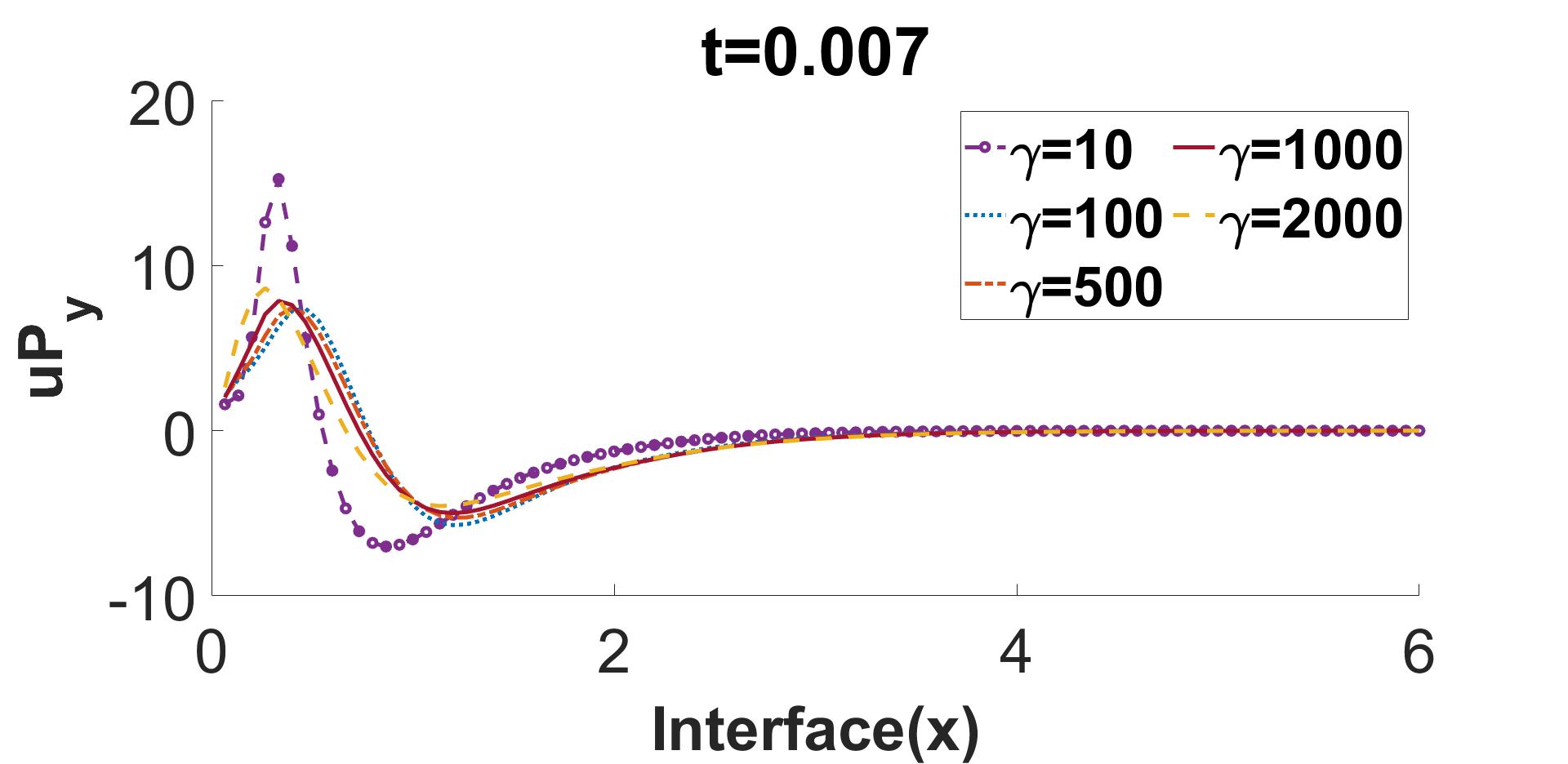}
\includegraphics[width=0.325\textwidth]{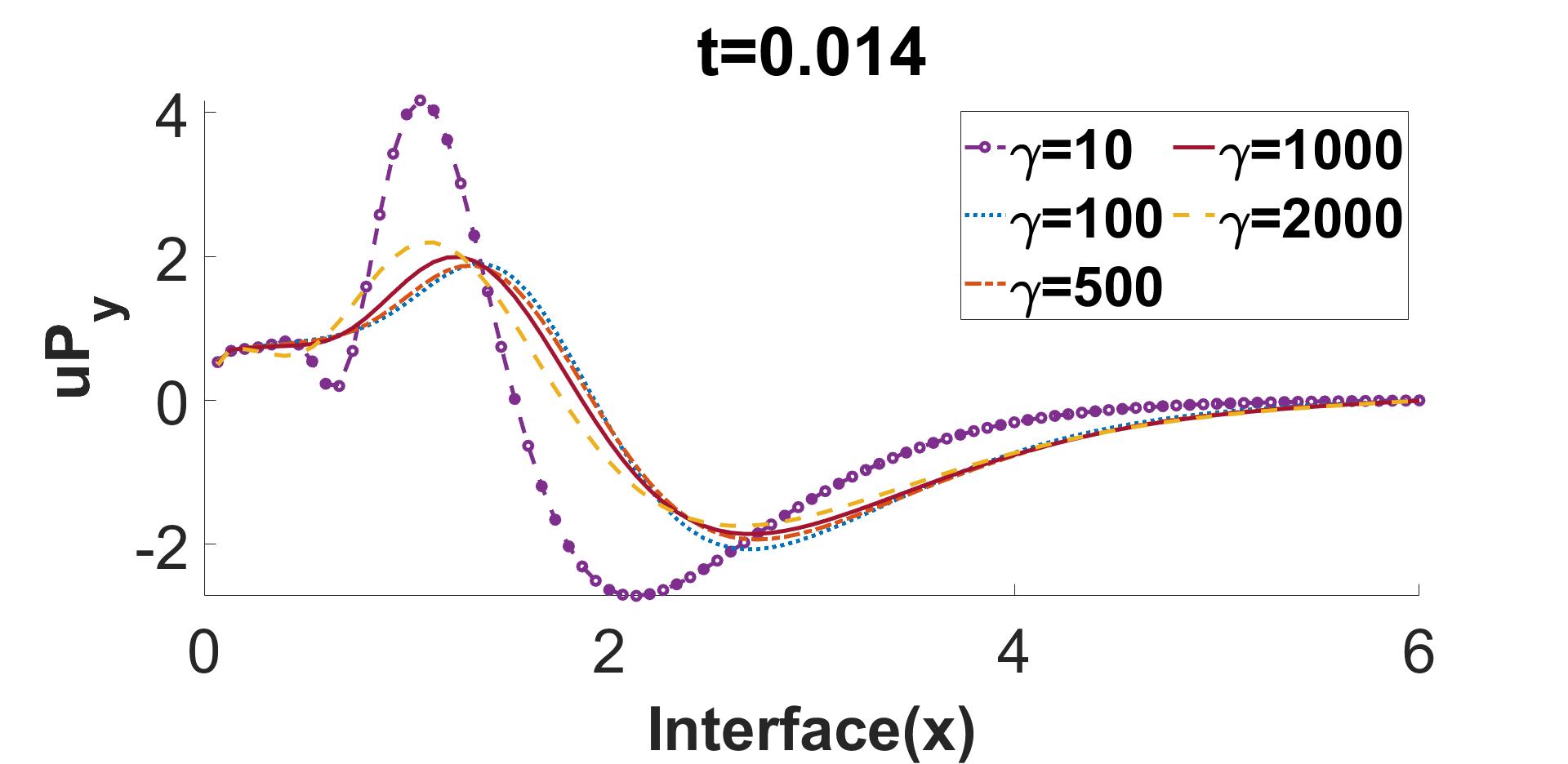}
\includegraphics[width=0.325\textwidth]{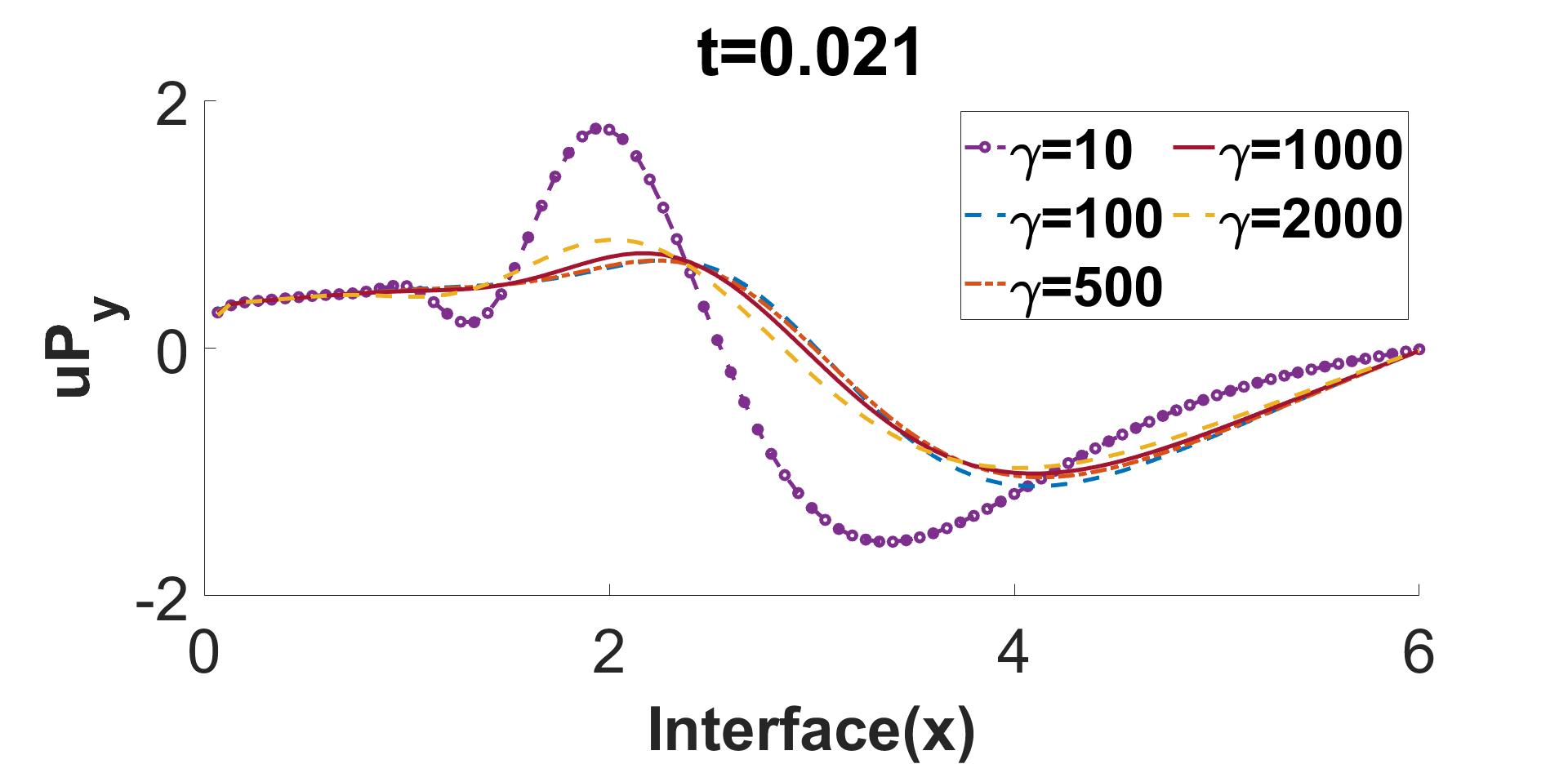}\\
\includegraphics[width=0.325\textwidth]{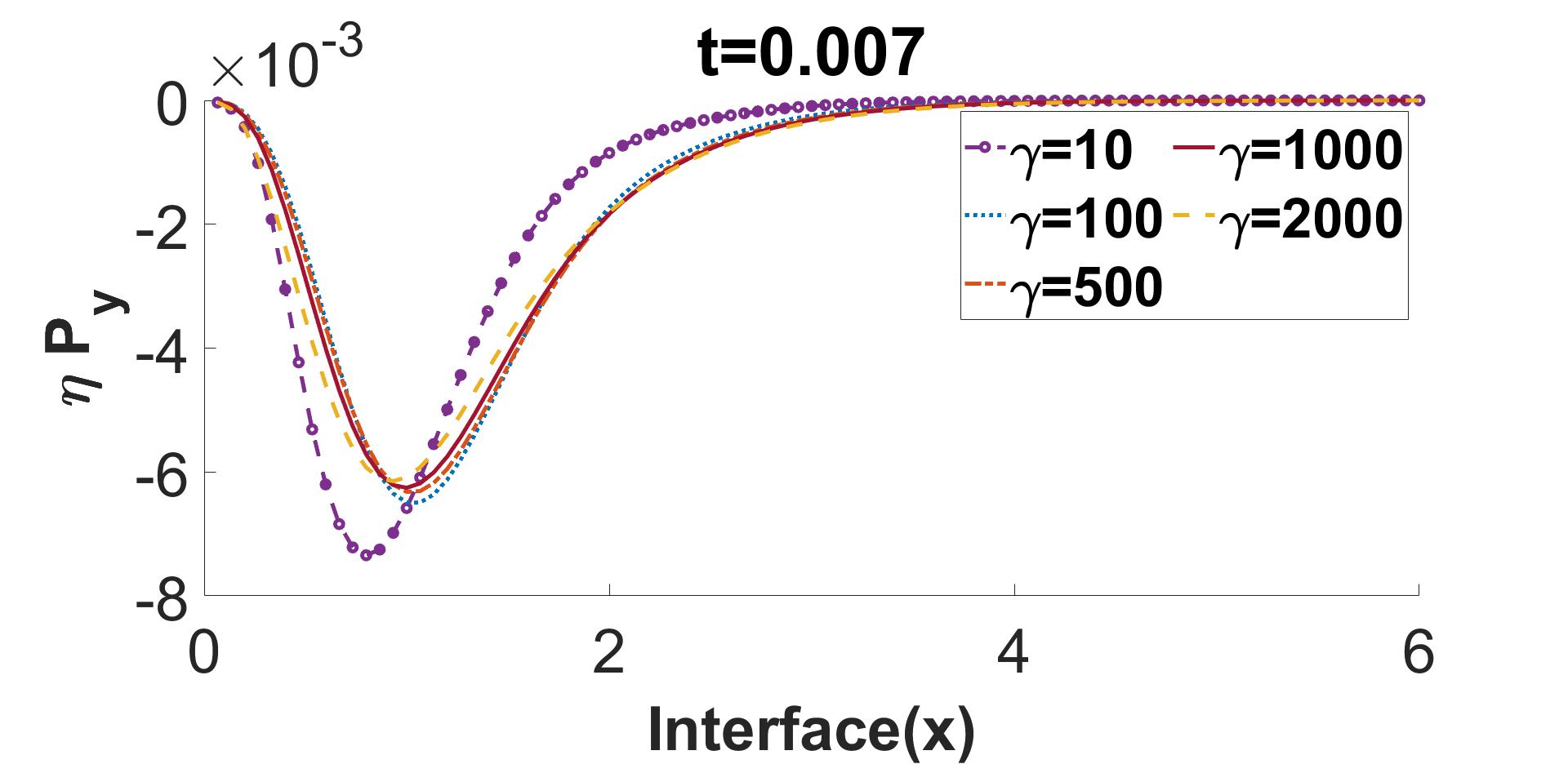}
\includegraphics[width=0.325\textwidth]{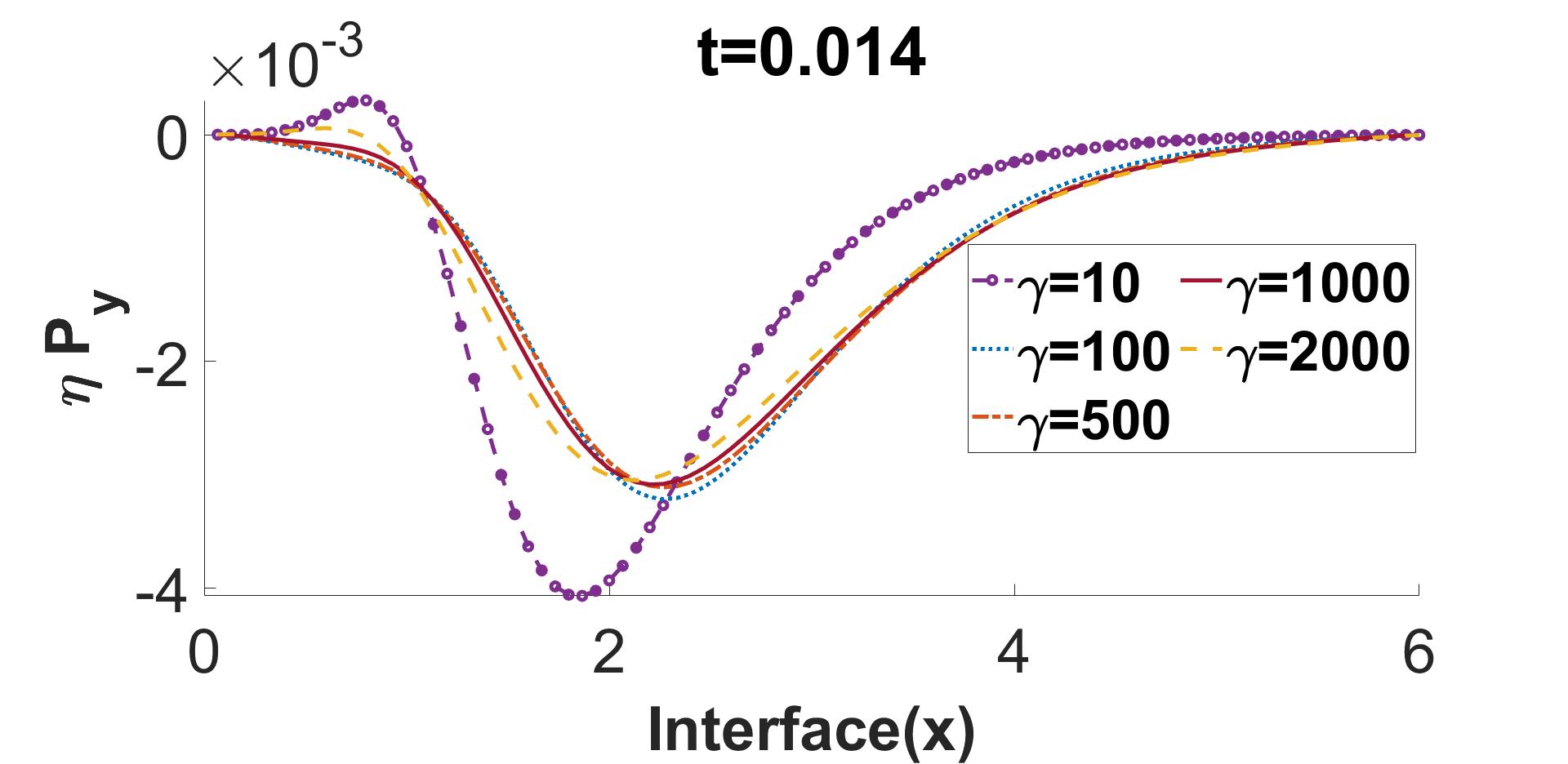}
\includegraphics[width=0.325\textwidth]{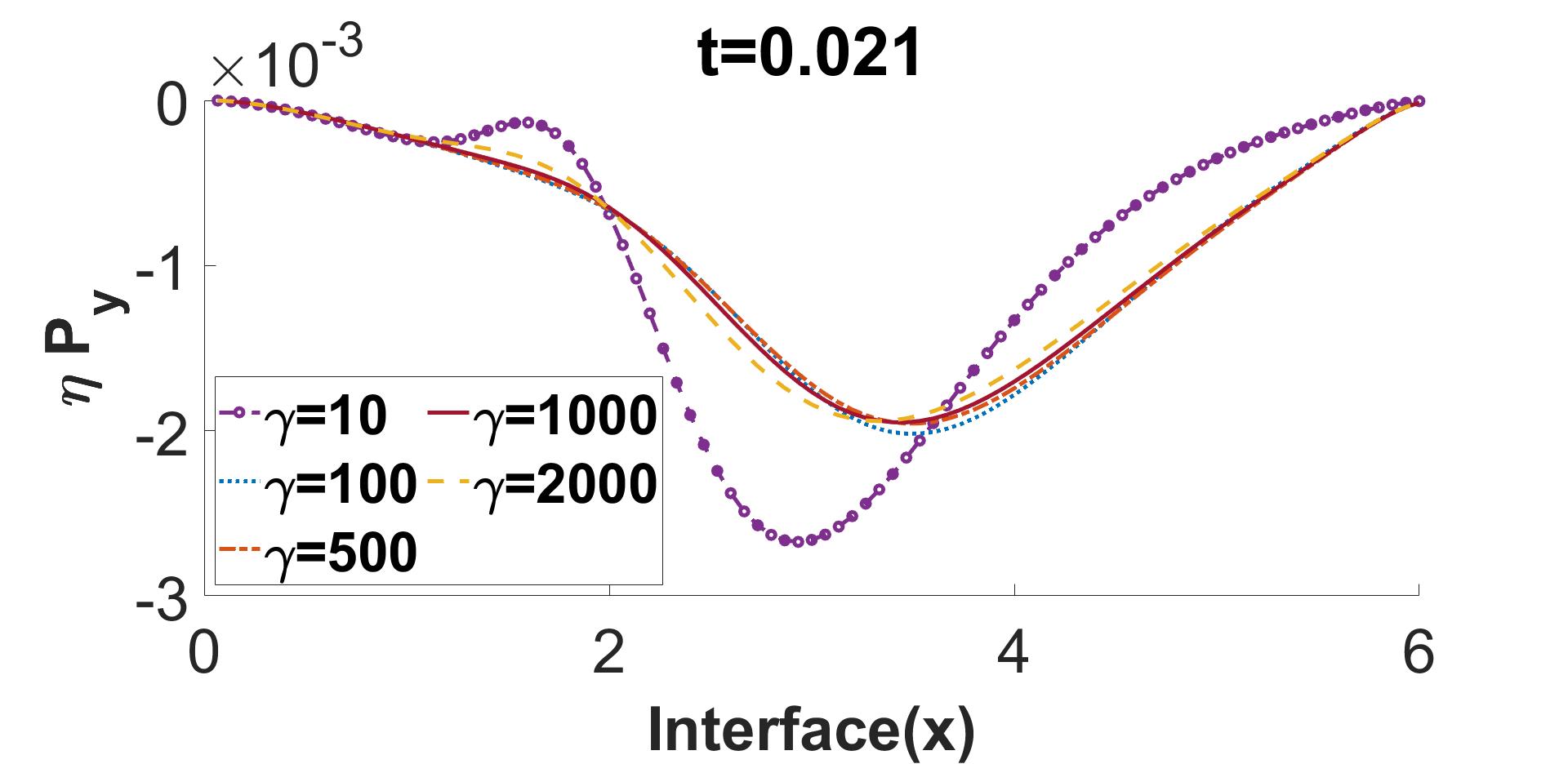}
\end{center}
\caption{Example 2, fluid pressure (first row), vertical fluid velocity (second row),
vertical Darcy velocity (third row), vertical structure
displacement (fourth row) along the interface computed by the Robin-Robin method in the non-iterative version at times $t= 0.007, 0.014, 0.021$ s (from left to right) for different values of $\gamma$.}\label{matplots:pressure-velocity-disp-gamma}
\end{figure}

\section{Conclusions}\label{sec:concl}
We developed an iterative and a non-iterative splitting methods for the Stokes-Biot model based on Robin-Robin transmission conditions. The methods utilize an auxiliary interface variable that models the Robin data. It is used to handle properly the insufficient regularity of the normal stress on the interface. With a suitable choice of negative norm for this variable, we established unconditional stability and first order convergence in time for the non-iterative scheme. To the best of our knowledge, this is the first such result in the literature in a general setting for Robin-Robin methods for both fluid-structure interaction and fluid-poroelastic structure interaction. We further studied the iterative version of the method and established convergence to a new monolithic scheme. We presented two sets of numerical experiments to illustrate the behavior of the methods. In the first example, based on a given analytical solution, we verified the first order time discretization rates and the convergence of the iterative scheme. We also studied the robustness to the Robin parameter $\gamma$. For extreme values of $\gamma$ (too small or too large), for which the accuracy of the non-iterative method deteriorates, we found that employing just a small number of iterations helps to recover optimal rates of convergence. The second example, based on a blood flow modeling benchmark, illustrated the applicability of the methods for computationally challenging physical parameters in the regimes of added-mass-effect and poroelastic locking. 

\bibliographystyle{abbrv}
\bibliography{fpsi-robin-robin}

\end{document}